\documentclass[a4paper,final]{amsart}
\usepackage[utf8]{inputenc}
\usepackage[T1]{fontenc}
\usepackage[UKenglish]{babel}
\usepackage[margin=30mm]{geometry}
\usepackage[dvipsnames]{xcolor}

\usepackage[compress]{natbib}
\setcitestyle{square} 
\setcitestyle{numbers} 

\usepackage{amsmath,amssymb,amsfonts,amsthm}
\usepackage{mathrsfs,eucal,dsfont}

\usepackage{multicol,calc}

\newenvironment{nindex}[1][51pt]%
    {\begin{list}{}%
        {%
            \setlength{\labelwidth}{#1}%
            \setlength{\leftmargin}{\labelwidth+\labelsep}%
            \setlength{\itemsep}{2.5pt}%
            \setlength{\parsep}{0pt}%
            \setlength{\rightmargin}{0pt}%
        }%
    }%
    {\end{list}}

\usepackage[colorlinks=true,
        linkcolor=black,
        citecolor=black,
        filecolor=blue,
        pagecolor=black,
        urlcolor=black,
        bookmarks=true,
        bookmarksopen=true,
        bookmarksopenlevel=3,
        plainpages=false,
        pdfpagelabels=true,
        pagebackref]{hyperref}

\usepackage{hypernat}

\numberwithin{equation}{section}
\renewcommand\labelenumi{\textup{\alph{enumi})}}
\renewcommand\theenumi\labelenumi

\theoremstyle{plain}
\newtheorem{theorem}{Theorem}[section]
\newtheorem{lemma}[theorem]{Lemma}
\newtheorem{corollary}[theorem]{Corollary}
\newtheorem{proposition}[theorem]{Proposition}

\theoremstyle{definition}
\newtheorem{definition}[theorem]{Definition}
\newtheorem{example}[theorem]{Example}

\newtheorem{remark}[theorem]{Remark}
\newtheorem*{notation}{Notation}
\newtheorem*{ack}{Acknowledgement}

\newcommand{\et}{\quad\text{and}\quad}
\newcommand{\cut}{\mathrm{cut}}
\newcommand{\tail}{\mathrm{tail}}
\newcommand{\trunc}{\mathrm{cut}}
\newcommand{\op}{\mathrm{op}}
\newcommand{\cstar}{\circledast}
\newcommand\CC{\textbf{\upshape (C)}}
\newcommand\MM{\textbf{\upshape (M)}}
\newcommand\NN{\textbf{\upshape (N)}}
\newcommand\BB{\textbf{\upshape (B)}}

\newcommand{\Sph}{{\mathds{S}}}
\newcommand{\real}{\mathds R}
\newcommand{\rn}{{\mathds{R}^n}}
\newcommand{\rd}{{\mathds{R}^d}}

\newcommand{\nat}{\mathds N}
\newcommand{\Ee}{\mathds E}
\newcommand{\Pp}{\mathds P}
\newcommand{\Cc}{\mathds C}
\newcommand{\I}{\mathds 1}
\newcommand{\Bscr}{\mathscr B}

\newcommand{\gfrak}{\mathfrak{g}}
\newcommand{\mfrak}{\mathfrak{m}}

\newcommand{\hfrak}{\mathfrak{h}}
\newcommand{\sfrak}{\mathfrak{s}}

\newcommand{\supp}{\operatorname{supp}}
\newcommand{\id}{\operatorname{id}}
\newcommand{\Lip}{\operatorname{Lip}}

\renewcommand{\leq}{\leqslant}
\renewcommand{\geq}{\geqslant}
\renewcommand{\Re}{\ensuremath{\operatorname{Re}}}

\begin{document}

\title[General stable-like Markov processes]{Construction and heat kernel estimates of general stable-like Markov processes}

\date{}

\author[V.~Knopova]{Victoria Knopova}
\address[V.~Knopova and R.L.~Schilling]{TU Dresden\\ Fakult\"{a}t Mathematik\\ Institut f\"{u}r Mathematische Stochastik\\ 01062 Dresden, Germany}
\email{victoria.knopova@tu-dresden.de, rene.schilling@tu-dresden.de}

\author[A.~Kulik]{Alexei Kulik}
\address[A.~Kulik]{Faculty of Pure and Applied Mathematics\\ Wroc{\l}aw University of Science and Technology\\ ul. Wybrze{\.z}e Wyspia{\'n}skiego 27, 50-370 Wroc{\l}aw, Poland}
\email{kulik.alex.m@gmail.com}


\author[R.L.~Schilling]{Ren\'e L.\ Schilling}
\begin{abstract}
A stable-like process is a Feller process $(X_t)_{t\geq 0}$ taking values in $\rd$ and whose generator behaves, locally, like an $\alpha$-stable L\'evy process, but the index $\alpha$ and all other characteristics may depend on the state  space. More precisely, the jump measure need not to be symmetric and it strongly depends on the current state of the process; moreover, we do not require the gradient term to be dominated by the pure jump part. Our approach is to understand the above phenomena as suitable microstructural perturbations.

We show that the corresponding martingale problem is well-posed, and its solution is a strong Feller process which admits a transition density. For the transition density we obtain a representation as a sum of an  explicitly given principal term -- this is essentially the density of an $\alpha$-stable random variable whose parameters depend on the current state $x$ -- and a residual term; the $L^\infty\otimes L^1$-norm of the residual term is negligible and so is, under an additional structural assumption, the $L^\infty\otimes L^\infty$-norm. Concrete examples illustrate the relation between the assumptions and possible transition density estimates.
\end{abstract}



\maketitle

\tableofcontents

\section*{List of Important Notation}

Throughout the paper, we use various constants, indices and conditions which appear in several places. In order to simplify reading, we indicate here where the most important of these are defined.

{\footnotesize
{\raggedright
\noindent
\begin{multicols}{3}
\raggedright
\parindent0pt
\begin{nindex}
\item[$N(x,du)$]                \eqref{set-e02}, \eqref{set-e04} p.~\pageref{set-e04}
\item[$\mu(x,du)$]              \eqref{set-e06} p.~\pageref{set-e06}
\item[$\nu(x,du)$]              \eqref{set-e11} p.~\pageref{set-e11}
\item[$\sigma(x,d\ell)$]        \eqref{set-e06} p.~\pageref{set-e06}

\item[$\alpha(x)$]              \eqref{set-e06} p.~\pageref{set-e06}
\item[$\alpha_{\min}, \alpha_{\max}$] \eqref{M0} p.~\pageref{M0}
\item[$\beta(x)$]               \eqref{N1} p.~\pageref{N1}
\item[$\gamma(x)$]              \eqref{B1} p.~\pageref{B1}
\item[$\delta(x)$]              \eqref{B1} p.~\pageref{B1}
\item[$\theta(x)$]              \eqref{pfti-e18} p.~\pageref{pfti-e18}
\item[$\lambda(x)$]             \eqref{set-e06} p.~\pageref{set-e06}
\item[$\lambda_{\min}, \lambda_{\max}$] \eqref{M0} p.~\pageref{M0}
\item[$\zeta(x)$]               \eqref{pfti-e12} p.~\pageref{pfti-e12}
\item[$\zeta_{\min}, \zeta_{\max}$] \eqref{pfti-e12} p.~\pageref{pfti-e12}

\item[$b_t(x)$]                 \eqref{set-e12} p.~\pageref{set-e12}
\item[$B_t(x)$]                 \eqref{pfti-e22} p.~\pageref{pfti-e22}
\item[$\chi_t(x)$]              \eqref{set-e18} p.~\pageref{set-e18}
\item[$\kappa_t(y)$]            \eqref{pfti-e24} p.~\pageref{pfti-e24}

\item[$p_t^0(x,y)$]             \eqref{para-e06} p.~\pageref{para-e06},\newline \eqref{pfti-e26} p.~\pageref{pfti-e26}
\item[$p_t^{z,\cut}(x)$]        \eqref{pfti-e16} p.~\pageref{pfti-e16}
\item[$K^{0;c}_t(x,y)$]         \eqref{pfti-e42} p.~\pageref{pfti-e42}
\item[$K^{1;c}_t(x,y)$]         \eqref{pfti-e44} p.~\pageref{pfti-e44}
\item[$f_{t,a,c}(x)$]           \eqref{pfti-e40} p.~\pageref{pfti-e40}

\item[$\Phi_t(x,y)$]            \eqref{para-e12} p.~\pageref{para-e12}
\item[$\Psi_t(x,y)$]            \eqref{para-e26} p.~\pageref{para-e26}

\item[$\eta$]                   \eqref{M2} p.~\pageref{M2}
\item[$\epsilon_B$]             \eqref{set-e16} p.~\pageref{set-e16}
\item[$\epsilon_\nu$]           \eqref{N1} p.~\pageref{N1}
\item[$\epsilon$]               \eqref{B1} p.~\pageref{B1}

\item[$\hfrak$]                 \eqref{B1} p.~\pageref{B1}
\item[$\mfrak$]                 \eqref{pfti-e18} p.~\pageref{pfti-e18}
\item[$\mathfrak{q}$]           \eqref{set-e32} p.~\pageref{set-e32}
\item[$\mathfrak{r}$]           \eqref{set-e36} p.~\pageref{set-e36}
\item[$\sfrak$]                 \eqref{pfti-e12} p.~\pageref{pfti-e12},
                                \eqref{pfti-e74} p.~\pageref{pfti-e74}
\item[$\epsilon_\Phi$]          \eqref{para-e18} p.~\pageref{para-e18}
\item[$\epsilon_R$]             \eqref{para-e40} p.~\pageref{para-e40}

\item[\eqref{B0}, \eqref{B1}]   p.~\pageref{B0}
\item[\eqref{C0}--\eqref{C2}]   p.~\pageref{C0}
\item[\eqref{M0}--\eqref{M2}]   p.~\pageref{M0}
\item[\eqref{N0}, \eqref{N1}]   p.~\pageref{N0}

\item[$\cstar$]                 \eqref{para-e08} p.~\pageref{para-e08}
\end{nindex}
\end{multicols}
}
}

\section{Introduction}\label{int}

Stable L\'evy processes are frequently used in physical models, where they are often called L\'evy flights; a good overview is given in the monograph~\cite{KRS04} and the survey paper~\cite{Ch06}. Our principal aim in this tract is to investigate \emph{stable-like} processes in the widest possible generality; heuristically, a stable-like process can be understood as a stable process whose parameters and characteristics depend on the current position of the process. Such an extension is highly relevant in applications; a possible example is the famous Ditlevsen model of the millennial climate changes~\cite{Dit99}.

A natural place for a mathematical treatment of general stable-like processes is within the theory of \emph{L\'evy-type process}, see~\cite{BSW} and Section~\ref{set} below.  Recall that the characteristic triplet  (or infinitesimal characteristics)  of a $d$-dimensional $\alpha$-stable distribution with $\alpha\in (0,2)$ is $(b,0,\mu)$ where $b\in\rd$ is the drift vector, and $\mu$ is the \emph{L\'evy measure} on $\rd\setminus\{0\}$ given by
\begin{align}\label{int-e02}
    \mu(A)
    = \int_0^\infty \int_{\Sph^{d-1}} \I_A(r\ell) r^{-1-\alpha} \,\Sigma(d\ell) \,dr,
    \quad A\in \Bscr(\rd\setminus\{0\}),
\end{align}
see e.g.~\cite[Theorem~14.3]{sato}; the \emph{spherical} part $\Sigma(d\ell)$ of the L\'evy measure $\mu$ is a finite measure on the sphere $\Sph^{d-1}$. It is convenient to use the normalized measure $\sigma(d\ell)= \lambda^{-1}\Sigma(d\ell)$ with $\lambda = \Sigma(\Sph^{d-1})$ and a probability measure $\sigma(d\ell)$, and to characterize a stable model using the \emph{external drift vector} $b$, the \emph{stability index} $\alpha$,  the \emph{intensity} $\lambda$, and the \emph{polarization measure} $\sigma(d\ell)$; `polarization' refers to the fact that $\sigma(d\ell)$  describes the distribution of the jump directions of the process.

A stable-like process is a L\'evy-type process with state-dependent drift $b(x)$ and \emph{L\'evy kernel} of the form similar to~\eqref{int-e02} but with state-dependent parameters, resp., characteristics $\alpha=\alpha(x)$, $\lambda=\lambda(x)$, $\sigma=\sigma(x, d\ell)$. The questions studied in this paper can be summarized as follows: Given a set of infinitesimal characteristics $b(x)$, $\alpha(x)$, $\lambda(x)$, $\sigma(x, d\ell)$, can one guarantee that there exists a corresponding stochastic process? If so, which further information on the structure of the process and its local properties can be derived -- and which type of assumption is needed for any given property?
These natural and seemingly simple questions turn out to be quite challenging.

Typically, the existence of \emph{some} stochastic process with prescribed infinitesimal characteristics is easy to derive, cf.~\cite{Ho95}, but uniqueness, hence the strong Markov property, is a quite delicate problem. For  $\alpha<1$ it is known that adding  an $x$-dependent  drift term to an $\alpha$-stable process may destroy weak uniqueness if the H\"older index $\gamma$ of the drift is small, see~\cite{TTW74}. This effect is in striking contrast to the diffusion case, and has a deep relation with the fact that for $\alpha<1$ the non-local part of the generator does not dominate the gradient part in the sense of `order of differentiation'.

Evidently, there must be an interplay between the order of the stability index $\alpha$ and the H\"older index, but it is far from being clear which is the correct condition. Assuming that the non-local part dominates the gradient, $\alpha\geq 1$ or $b(x)\equiv 0$, is apparently too restrictive. In~\cite{TTW74}, uniqueness was shown under \emph{the balance condition} $\alpha+\gamma>1$, which is quite close to being a necessary condition, since the counterexample~\cite{TTW74} mentioned in the previous paragraph works for any pair $\alpha, \gamma$ with $\alpha+\gamma<1$.

This discussion highlights the fact that one needs enough regularity in order to construct a stochastic process. If $b(x)$ is Lipschitz, we can do this by standard methods. As soon as $b(x)$ is only H\"older continuous, the defect must be compensated by a sufficiently `regular' jump behaviour. The technique used in~\cite{TTW74} is essentially one-dimensional, and we are far from having a rigorous treatment in a wider class of models. Uniqueness for the multidimensional $\alpha$-stable model with state-space dependent drift under the balance condition $\alpha+\gamma>1$ was recently proved in~\cite{Ku18}. The construction in~\cite{Ku18} gives an outline how one can treat the case  when the non-local part is not dominating. But~\cite{Ku18} covers only rotationally invariant jump measures -- that is, $\sigma(x, d\ell)=\sigma(d\ell)$ is the uniform distribution on the unit sphere; this hides other substantial difficulties which we will now  outline;  a more detailed  discussion is deferred to  Section~\ref{exa} and Section~\ref{road}.

We  will call a L\'evy-type model \emph{essentially singular} if the values of the $x$-dependent jump (L\'evy) kernel cannot be dominated by a single reference measure; in the stable-like setting this means that the distribution of the jump directions strongly varies from place to place. The analysis of such models encounters conceptual difficulties: A natural way to construct and to study such a L\'evy-type process is to interpret its transition density as (some kind of) fundamental solution to the Cauchy problem for the corresponding non-local integro-differential (or pseudo-differential) equations -- we use the abbreviation $\Psi$DE --, and to adapt classical PDE-methods, such as the \emph{parametrix method} which goes back to  Levi~\cite{Le1907},  Hadamard~\cite{had1911} and Gevrey~\cite{gev13}  (see also Feller~\cite{Fe36} for a simple non-local setting, and Friedman~\cite{Fr64} or Eidel'man~\cite{eid69} for two classic treatments of the parametrix method).  For non-local operators with L\'evy kernels which are comparable with \emph{a single} (e.g.\ $\alpha$-stable) reference L\'evy measure, this programme is indeed feasible, see~\cite{KKK19} for an overview and an extensive literature survey. An important feature of the classical parametrix construction is the property, that all approximating kernels as well as the solution obtained by the parametrix method obey certain `universal' kernel estimates, e.g.\ Gaussian estimates for 2nd order PDEs; in the L\'evy-type setting such kernel estimates come from the reference L\'evy measure. Essentially singular models, however, may behave badly, see Example~\ref{exa-01}, Example~\ref{exa-03}, and the discussion following these examples. In particular, the transition density may be unbounded, thus estimates w.r.t.\ a bounded heat kernel (e.g.\ Gaussian or stable) are bound to fail. The heuristics behind this effect is quite obvious: If the `jump patterns' in various states are substantially different, they can not (or only very roughly) be covered by a \emph{single} kernel, and this may be passed on to the entire dynamics. This means, however, that such models require a very different approach.

Our analysis of stable-like processes does not require that the gradient term is dominated and it works for essentially singular models. In the list of main assumptions in Section~\ref{set} only the condition \eqref{B1}, p.~\pageref{B1},   looks special; in fact, it is a version of the balance condition, which is inevitable, as we know from~\cite{TTW74,Ku18}. The model under investigation is much more general: Along with a stable-like part, the  L\'evy kernel will contain a lower local activity `perturbation' part. A strong motivation to include such `microstructural' noise comes from models in mathematical finance, e.g.\ for high-frequency trading, where a well-structured driving noise (which reflects the rules of the market) with microstructural terms (which correspond to erroneous moves of some agents or unexpected external influence) is needed; a detailed discussion can be found in~\cite{AJ07}. Such perturbation terms may also lead to \emph{L\'evy-type systems with resetting} (see Example~\ref{exa-01})  which are actively studied in  physics literature, see~\cite{KMSS14} and the references given there. The presence of a `perturbation' part reveals some unexpected mathematical properties of the model; in particular, we will see in Section~\ref{exa} that the tail of a stable-like kernel in the essentially singular setting may act as a `perturbation' of the entire kernel.

We show that the martingale problem for such stable-like models with microstructural perturbations is well-posed, and that its solution is a L\'evy-type (or Feller) process which has a transition probability density $p_t(x,y)$; this answers the first of the two general questions formulated above. In order to approach the second question, we adapt the  framework, used in the paper~\cite{Ku} where a simpler one-dimensional model with constant $\alpha(x)\equiv \alpha$ is considered. In the present tract, we represent the
transition density $p_t(x,y)$ as a sum of an explicitly given `principal' part and a `residual' part which has explicit  bounds, see~\eqref{set-e24}, \eqref{set-e26} and~\eqref{set-e38} below. Such a representation describes the law of  the process at any fixed time $t$  `locally', i.e.\ near the starting point $x_0$. The approximation is of the form
\begin{align*}
    \widetilde{X}^{x_0}_{t}=\mathfrak{f}_t(x_0)+t^{1/\alpha(x_0)}U^{x_0}_t
\end{align*}
with a deterministic \emph{regressor term} $\mathfrak{f}_t(x_0)$  and an \emph{$\alpha(x_0)$-stable innovation term} $U_t^{x_0}$ with explicit characteristics.

The paper is organized in the following way:  In Section~\ref{set} we give some preliminaries and present the main results.
Section~\ref{exa} contains   examples  which illustrate (a) the new effects occurring in the essentially singular setting and (b) the relation between various types of estimates.  In Section~\ref{rel} we give an overview of the known results on the parametrix and the heat kernel estimates, and explain the novelty of our results and methodology.  For the reader's convenience, we explain the methodology in Section~\ref{para}, where the general argument is presented without getting into technical details.  The details of the proofs are given in  Section~\ref{pfti}, Section~\ref{pti} and Section~\ref{pftii}. Technical auxiliary statements are collected in the appendices.

\begin{notation}
Most of our notation will be standard or self-explanatory. Notation which is only used locally, is introduced where it is needed. We write $a\wedge b :=\min\{a,b\}$ and $a \vee b:=\max\{a,b\}$ for the minvimum and maximum of $a$, $b$. By $|\cdot|$ we denote both the modulus of real numbers and the Euclidean norm of a vector.

Throughout, $c$  and $C$ are positive constants, which may change from line to line. $f\asymp g$ means that $cg\leq f\leq C g$. If $L$ is an operator, we write  $L_x$ to emphasize that $L$ acts on a function $f(x,y)$ with respect to the variable $x$,   i.e.\  $L_xf(x,y)= Lf(\cdot,y)(x)$. $C_b(\rd)$, resp., $C_\infty(\rd)$, resp., $C_0(\rd)$ are the families of continuous functions which are bounded, resp.\ vanish at infinity, resp.\ have compact support. A superscript $k$ indicates that the functions are continuously differentiable and that all their derivatives are bounded, resp.\ vanish at infinity or have compact support.
\end{notation}

\begin{ack}
    Financial support through the joint Polish--German ``Beethoven 3'' grant (A.\ Kulik: 2018/31/G/ST1/02252; R.\ Schilling: SCHI 419/11-1) is gratefully acknowledged. Part of this work was done while A.\ Kulik was visiting the Mathematics Department at TU Dresden in May--July 2018; he is grateful for the hospitality and perfect working conditions.
\end{ack}

\section{Related work}\label{rel}

Let us briefly give an overview on the existing literature. It is well-known that, if $(L,D(L))$ is a Feller generator and the set $C_c^\infty(\rd)$ of smooth compactly supported  functions belongs to the domain $D(L)$ of $L$, then  for $f\in C_c^\infty(\rd)$
\begin{align}\label{rel-e00}
    L f(x)
    = b(x)\cdot \nabla f(x) + \int_{\rd\setminus \{0\}} \left(f(x+u)-f(x)  - \nabla f(x) \cdot u \I_{\{|u|\leq 1\}}\right) N(x,du)
\end{align}
with measurable and locally bounded coefficients $b:\rd\to\rd$, $\sigma:\rd\to\real^{d\times d}$ and  L\'evy kernel $N(x,du)$.  It is easy to see that the operator $L$ extends to all $f\in C_\infty^2(\rd)$.  This is the well-known Courr\`ege theorem, cf.~\cite[Section 2.3]{BSW} or Jacob~\cite[Theorem 4.5.21]{Ja01}.

It is a delicate problem to establish whether an operator $(L,C_\infty^2(\rd))$ given above generates a Feller process. One of the possible approaches is to establish the well-posedness of the martingale problem for $(L, C_\infty^2(\rd))$. Work on the martingale problem for operators of type~\eqref{rel-e00} started with Grigelionis~\cite{G68},  Komatsu~\cite{Ko73,Ko84a} and Mikulevicius \& Pragarauskas~\cite{MP92a,MP92b}; see also~\cite{MP14a,MP14b}. For variable stability index  the uniqueness of the solution was  investigated first by  Bass~\cite{Ba88} (in dimension $d=1$), and later (by a completely different Hilbert-space approach based on $\Psi$DEs) by Jacob \& Leopold~\cite{JL93},  Negoro~\cite{Ne94}, Kikuchi \&  Negoro~\cite{KN97} and Hoh~\cite{Ho00}. Both approaches require at least some smoothness of the symbol of the characteristics and that the non-local $\alpha(x)$-stable part dominates the gradient term; this is a principal difference to our approach.

The well-posedness  of the martingale problem is intrinsically related to the parametrix construction of the solution to the Cauchy problem for $L$. As we have already mentioned in the introduction, the parametrix method was first proposed by Levi~\cite{Le1907},   Hadamard~\cite{had1911} and Gevrey~\cite{gev13} for differential operators  and later extended  by Feller~\cite{Fe36} to a simple non-local setting. An extensive overview of the existing literature on this method is available in~\cite{KKK19}; let us mention here only the most closely related treatments. For non-local operators of hyper-singular type the parametrix method was developed  by Kochubei~\cite{K89}, see also Drin' \& Eidelman~\cite{Dr77,DE81}, and  the monograph by  Eidelman, Ivasyshen \& Kochubei~\cite{EIK04}. As part of the method, one obtains upper and lower  estimates on of the so constructed solution. In particular, the approach~\cite{K89} allows us to handle the case where the measure  $N(x,du)$ (in our notation) has a principal component $\mu(x,du)$ which is of \emph{stable type}, i.e.
\begin{align*}
    \mu(x,du) = c(x,u) \frac{du}{|u|^{d+\alpha}}.
\end{align*}
This approach essentially requires that  the non-local part dominates the gradient, i.e.\ $\alpha>1$. In~\cite{KK18,Ku18} (symmetric setting, $c(x,u) = c(x)$, $N=\mu$) this approach was further extended,  in particular without the domination assumption on the jump kernel.   In this particular setting it is possible to get two-sided bounds for the transition probability density $p_t(x,y)$ in the following form
\begin{align}\label{rel-e02}
    p_t(x,y)\asymp  \frac{1}{t^{d/\alpha}}\frac{1}{(1+|y-\chi_t(x)|t^{-1/\alpha})^{d+\alpha}},  \quad x,y\in \rd, \, t\in (0,1],
\end{align}
where $\chi_t(x)$ is the tuning flow, which is crucial if $\alpha\in (0,1]$ and in the presence of the drift $b(x)$ in~\eqref{rel-e00}. In the simplest case when the drift is absent and the L\'evy measure is symmetric, the estimate~\eqref{rel-e02} simplifies to
\begin{align}\label{rel-e04}
    p_t(x,y)\asymp \frac{1}{t^{d/\alpha}}\frac{1}{(1+|y-x|t^{-1/\alpha})^{d+\alpha}},  \quad x,y\in \rd, \, t\in (0,1];
\end{align}
Chen \& Zhang~\cite{CZ16} have a similar result in the symmetric setting and for $N=\mu$.

Such estimates are no longer true if the measure $\mu(x,du)$ is not rotationally invariant. In Bogdan, Knopova \& Sztonyk~\cite{BKS17} it was shown that under the condition that the measure $\sigma(x,d\ell)$ on the sphere is i) symmetric, i.e.\ $\sigma(x,S)=\sigma(x,-S)$ for any measurable set $S\subset \Sph^{d-1}$, and ii)  a so-called \emph{$\gamma-1$-measure}, i.e.\ for some $\gamma\in (0,d)$
\begin{align}\label{rel-e06}
     \sigma(x, B(\ell,r))\leq  c r^{\gamma-1}, \quad r\in (0,1),
\end{align}
then the kernel $p_t(x,y)$ satisfies the upper estimate
\begin{align}\label{rel-e08}
   p_t(x,y)\leq \frac{ C}{t^{d/\alpha}}\frac{1}{(1+|y-x|t^{-1/\alpha})^{\gamma+\alpha}}, \quad x,y\in \rd, \, t\in (0,1].
\end{align}
This estimate is the generalization of the results for the transition probability density  of a  L\'evy process, obtained in a series of papers by Sztonyk~\cite{Sz10a,Sz10b,Sz11,Sz17}, and Kaleta \& Sztonyk~\cite{KS13,KS15,KS17}. In general,  even in the L\'evy case  without additional assumptions on the spectral measure $\sigma(d\ell)$, it is impossible to get a lower bound with the same rate as the upper bound, see~\cite{Sz17}. In the recent works of Kulczycki, Ryznar \& Sztonyk~\cite{KR17,KR19,KRS18} systems of SDEs of the type $dX_t = A(X_{t-}) \,dZ_t$, driven by cylindrical $\alpha$-stable processes are studied. The authors used the parametrix method in order to construct the transition density of the solution, to obtain estimates for it, and to prove the strong Feller property of the solution. The case where the matrix $A$ is non-diagonal is particularly interesting, since the structure of the transition density estimate is completely different; in particular, it is  impossible to obtain an estimate of type~\eqref{rel-e08}. Similar effects have also been observed in~\cite{Ku} in the non-symmetric scalar setting ($N=\mu$ plus a further perturbation); see Section~\ref{exa} for a detailed discussion.

The first parametrix construction in the \emph{stable-like} case, i.e.\  when the \emph{stability index} is $x$-dependent  and the L\'evy kernel is given by
\begin{align*}
    N(x,du)= |u|^{-d-\alpha(x)}\,du, \quad 0<\alpha_{\min}\leq \alpha(x)\leq \alpha_{\max}<2.
\end{align*}
is due to Kolokoltsov~\cite{Ko00}. In the papers of K\"uhn~\cite{Kue17b,Kue17a} this problem was treated for different kernels assuming a kind of sector condition for the symbol of the operators. More precisely, this  method needs that the symbol of respective operator can be extended to a hour-glass shaped sector of the complex domain, which implies the exponential decay of the tails of the measure $N(x,du)$ at infinity. Both approaches yield bounds for $p_t(x,y)$ from above and below by  power-type~\cite{Ko00} and exponential~\cite{Kue17b,Kue17a} functions, with growth parameters depending on $\alpha_{\min}$ and $\alpha_{\max}$. See also  Knopova \& Kulik~\cite{KK17} for the parametrix construction for a general L\'evy-type model, where so-called \emph{compound  kernel estimates} on the kernel $p_t(x,y)$ are constructed.

There exists also a completely different version to the parametrix method.  In a Hilbert space  setting a version of the parametrix method for operators of type~\eqref{rel-e00} was developed in  the works of  Ch.\ Iwasaki (Tsutsumi) and  N.\ Iwasaki~\cite{Ts74,Iw77,IwIw79,IwIw81} and  Kumano-go~\cite{Ku76a,Ku76b,Ku77};  see also the monograph by Kumano-go~\cite{Ku81}, and the work  Negoro~\cite{Ne94} and Kikuchi \& Negoro~\cite{KN97};  these papers are all in the framework of classical (H\"{o}rmander-type) symbols; for rough negative-definite symbols, we mention the pioneering work of Hoh~\cite{Ho98a,Ho98b}, Jacob~\cite{Ja93,Ja94,Ja02} and  B\"ottcher~\cite{Boe05,Boe08}. Both approaches are based on a symbolic calculus for pseudo differential operators, which allows one to prove the existence of the fundamental solution in a certain functional space assuming  that the symbol is of class $C^\infty$- or at least $C^k$.  The solution is constructed in the form of an $L_2$-convergent series. Note that in this version of the parametrix method  we do not get any explicit information on the probability heat kernel of the related process.

\section{Setting and main results}\label{set}
\subsection{Preliminaries}\label{pre}

The central object of our study are integro-differential operators of the form
\begin{align}\label{set-e02}
    L f(x)
    = b(x)\cdot \nabla f(x) + \int_{\rd\setminus \{0\}} \left(f(x+u)-f(x)  - \nabla f(x) \cdot u \I_{\{|u|\leq 1\}}\right) N(x,du)
\end{align}
which are defined for all functions $f\in C_\infty^2(\rd)$. As we have explained in the introduction, our aim is to show that $L$ is the generator of a L\'evy-type process $X=(X_t)_{t\geq 0}$ and to understand the structure of the transition probability density of $X_t$.

The function $b:\rd\to\rd$ is a drift vector, and $N(x,du)$ is a \emph{L\'evy kernel}, i.e.\ a kernel which is defined for all sets $A\in \Bscr(\rd\setminus \{0\})$ and satisfies $\int_{\rd\setminus\{0\}} \min\{1,|u|^2\}\,N(x,du)<\infty$; for fixed $x$ we also speak of a \emph{L\'evy measure}. In order to simplify notation, we will frequently write $\int_\rd \dots  N(x,du)$ or $\int \dots  N(x,du)$ instead of the more precise $\int_{\rd\setminus \{0\}} \dots N(x,du)$.

Throughout the paper we need the following \emph{continuity conditions}.

  \begin{align}
\tag{\bfseries C0}\label{C0}
&\parbox[t]{.9\linewidth}{%
    The function $b:\rd\to\rd$ is continuous;
}
\\
\tag{\bfseries C1}\label{C1}
&\parbox[t]{.9\linewidth}{%
    For any compactly supported continuous function $f\in C_0(\rd)$, $\supp f\subset \rd\setminus\{0\}$ the function
    \newline\mbox{}\hfill$\displaystyle
        x\mapsto
    \int_\rd f(u) \,N(x,du)
    $\hfill\mbox{}%
\newline
is continuous;}
\\
\tag{\bfseries C2}\label{C2}
&\parbox[t]{.9\linewidth}{%
    \mbox{}$\displaystyle  N\big(x,\{u:|u|=1\}\big)=0, \quad x\in \rd.
    $%
}
\end{align}

Split the kernel $N(x,du)$ into two parts, a \emph{principal part} $\mu(x,du)$ and a \emph{residual part} $\nu(x,du)$
\begin{align}\label{set-e04}
    N(x,du)
    = \mu(x,du) + \nu(x,du).
\end{align}
The principal part $\mu(x,du)$ is a stable-like kernel of the form
\begin{align}\label{set-e06}
    \mu(x,A)
    = \lambda(x)  \int_0^\infty \int_{\Sph^{d-1}} \I_A(r\ell) r^{-1-\alpha(x)} \,\sigma(x,d\ell) \,dr,
    \quad A\in \Bscr(\rd\setminus\{0\}),
\end{align}
with a state-dependent \emph{`stability' index}  $\alpha(x)\in (0,2)$,  the \emph{intensity} $\lambda(x)>0$, and the \emph{polarization kernel} $\sigma(x, d\ell)$; without loss of generality,  we assume that $\sigma(x,\cdot)$ is a probability measure.  The principal part $\mu(x,A)$ has the following \emph{scaling property}
\begin{align}\label{set-e08}
    \mu(x,t A)
    = t^{-\alpha(x)}\mu(x,A).
\end{align}
We assume that the principal part $\mu(x,du)$ satisfies, in a suitable way, a H\"{o}lder condition in the state space variable $x$: The stability index $\alpha(x)$ and the intensity $\lambda(x)$ are assumed to be be H\"{o}lder continuous. For the polarization kernel $\sigma(x, d\ell)$ we want to avoid a single reference kernel (thus, total variation norm) and include the possibility that the model is \emph{essentially singular}. Therefore, we will use the \emph{Wasserstein-$1$ distance}  $W_1$, see e.g.~\cite[\S11.8]{Du02}. The $W_1$-distance of  two probability measures $P$ and $Q$ on $\Sph^{d-1}$ is defined as
\begin{align}\label{set-e10}
     W_1(P,Q)
     := \inf\left\{\int_{\Sph^{d-1}\times \Sph^{d-1}}  |\ell_1-\ell_2| \,\Lambda(d\ell_1,d\ell_2) \::\:  \Lambda \in \mathscr{M}^1(P,Q)\right\},
\end{align}
where $\mathscr{M}^1(P,Q)$ denotes the set of all probability measures on $\Sph^{d-1}\times \Sph^{d-1}$ with marginals $P$ and $Q$.

The \emph{residual} or  \emph{perturbation} part $\nu(x,du)$ is a \emph{signed} kernel, and we denote by
\begin{align}\label{set-e11}
    \nu(x,du) = \nu_+(x,du) - \nu_-(x,du)
    \et
    |\nu|(x,du) = \nu_+(x,du) + \nu_-(x,du)
\end{align}
its Hahn decomposition and total variation, respectively. Since $N(x,du)$ is positive, we get
\begin{align*}
    \nu_-(x,du)\leq \mu(x, du)+\nu_+(x,du).
\end{align*}
Since $\nu_-(x,du)$ and $\nu_+(x,du)$ are supported in disjoint sets, this is the same as
\begin{align*}
    \nu_-(x,du)\leq \mu(x, du).
\end{align*}

\medskip
In order to show that the L\'evy-type operator~\eqref{set-e02} is the generator of a stochastic process, we will use the \emph{martingale problem} approach.  Recall that a process $X$ is said to be a \emph{solution} to the martingale problem $(L,C_\infty^2(\rd))$, if for every  $f\in C_\infty^2(\rd)$ the process
\begin{align*}
    f(X_t)-\int_0^t Lf(X_s)\, ds, \quad t\geq 0
\end{align*}
is a martingale w.r.t.\ the natural filtration of $X$. The martingale problem is said to be \emph{well-posed in $D(\mathds{R}^+, \rd)$} -- the Skorokhod space of c\`adl\`ag (right-continuous, finite left limits) functions with values in $\rd$ --, if
\begin{enumerate}
\item
    for any probability measure $\pi$ on $\rd$ there exists a solution $X$ to the martingale problem such that the trajectories of $X$ are c\`adl\`ag and $\mathrm{Law}(X_0)=\pi$, and
\item
    any two solutions have the same distribution in $D(\mathds{R}^+,\rd)$.
\end{enumerate}

\subsection{Conditions and main results}\label{ssec-main}
We require three groups \MM, \NN\ and \BB\ of conditions. The first group  \MM\  is related to the parameters of the stable-like principal part of the jump kernel. It comprises the natural requirements of boundedness, H\"older regularity, and non-degeneracy:
\begin{align}
\tag{\bfseries M0}\label{M0}
&\parbox[t]{.9\linewidth}{%
    \mbox{}$\displaystyle
    \begin{aligned}[t]
     0<\lambda_{\min} &:=\inf_{x\in \rd} \lambda (x)\leq \sup_{x\in \rd} \lambda(x)=:\lambda_{\max}<\infty,
     \\
     0<\alpha_{\min} &:= \inf_{x\in\rd} \alpha(x)\leq \sup_{x\in\rd} \alpha(x)=:\alpha_{\max}<2.
    \end{aligned}$\hfill\mbox{}%
}
\\
\tag{\bfseries M1}\label{M1}
&\parbox[t]{.9\linewidth}{%
    The measure $\sigma(x,d\ell)$ from the representation~\eqref{set-e06} is non-degenerate:
    \newline\mbox{}\hfill$\displaystyle
        \inf_{x\in\rd} \inf_{v\in \Sph^{d-1}}  \int_{\Sph^{d-1}} (v \cdot \ell)^2 \,\sigma(x,d\ell) >0.
    $\hfill\hfill\mbox{}%
}
\\
\tag{\bfseries M2}\label{M2}
&\parbox[t]{.9\linewidth}{%
    There exist an exponent $\eta\in (0,1]$ and a constant $C>0$ such that for all $x,y\in \rd$
    \newline\mbox{}\hfill$\displaystyle\vphantom{\int}
        |\alpha(x)-\alpha(y)| + |\lambda(x)-\lambda(y)|+ W_1\Big(\sigma(x,\cdot),  \sigma(y,\cdot)\Big)
        \leq C |x-y|^\eta.
    $\hfill\hfill\mbox{}%
}
\end{align}

\medskip

The second group \NN\ deals with the \emph{residual} or  \emph{perturbation} part $\nu(x,du)$.
\begin{align}
\tag{\bfseries N0}\label{N0}
&\parbox[t]{.9\linewidth}{%
   The kernel $|\nu|(x,du)$ has uniformly integrable tails, i.e.\ it satisfies
     \newline\mbox{}\hfill$\displaystyle\vphantom{\int^f}\sup_{x\in\rd} \lim_{R\to\infty}\int_{|u|>R} |\nu|(x,du) = 0.
   $\hfill\hfill\mbox{}%
}\\
\tag{\bfseries N1}\label{N1}
&\parbox[t]{.9\linewidth}{%
    There exist  $\epsilon_\nu$ and  $\beta(x) \geq 0$ such that $\alpha(x)-\beta(x)\geq \epsilon_\nu>0$ and for all $r\in (0,1]$
    \newline\mbox{}\hfill$\displaystyle\vphantom{\int}
        |\nu|(x,\{|u|\geq r\})\leq C r^{-\beta(x)}.
    $\hfill\hfill\mbox{}%
}
\end{align}
Condition~\eqref{N1} actually requires that the Blumenthal--Getoor index for $|\nu|(x,du)$ is, uniformly in $x$, smaller than the Blumenthal-Getoor index for $\mu(x,du)$. Heuristically, this means that the majority of small jumps for the entire kernel $N(x,du)$ comes from $\mu(x,du)$; this motivates our terminology `principal' part vs.\ `perturbation' part. Note that~\eqref{N1} does \emph{not} require that the \emph{uniform} upper bound for $\beta(x)$ is comparable with the \emph{uniform} lower bound for $\alpha(x)$: Our setting is much more flexible and only assumes a state-by-state comparison.

The last group \BB\  of conditions are related to the drift coefficient $b(x)$. In order to formulate them, we need to introduce the \emph{dynamically compensated drift}
\begin{align}\label{set-e12}
    b_t(x):=  b(x)- \int_{(1\wedge t)^{1/\alpha(x)} <|u|\leq 1} u \,N(x,du), \quad x\in \rd,\; t>0.
\end{align}
\begin{align}
\tag{\bfseries B0}\label{B0}
&\parbox[t]{.9\linewidth}{%
    The function $t\mapsto b_t(x)$ is continuous for every $x\in \rd$ and  there exists a constant $C>0$ such that $\sup_{x\in \rd} |b(x)|\leq C$.%
}
\\
\tag{\bfseries B1}\label{B1}
&\parbox[t]{.9\linewidth}{%
    There exist constants $\epsilon, \,\hfrak>0$ such that for
    \newline\mbox{}\hfill$\displaystyle\vphantom{\int}
    \notag \gamma(x):=1-\alpha(x)+\hfrak, \quad \delta(x):=-1+\frac{1}{\alpha(x)},
    $\hfill\hfill\mbox{}
    \newline
  \ for some $C=C_{\hfrak,T}$ and all $|x-y|\leq 1$, $t\in (0,T]$
    \newline\mbox{}\hfill$\displaystyle\vphantom{\int}
        |b_t(x)- b_t(y)|\leq C\left(|x-y|^{\gamma(x)}+|x-y|^{\gamma(y)}+(t^{\delta(x)}+t^{\delta(y)})|x-y|^\epsilon\right).
    $\hfill\hfill\mbox{}
}
\end{align}
Condition~\eqref{B1} is the only non-trivial structural assumption within \MM, \NN\ and \BB. This is the proper \emph{dynamic} version of the \emph{balance condition} from~\cite{TTW74,Ku18} in the current -- substantially more complicated -- setting. Indeed, \eqref{B1} in a sense requires a space- and time-dependent version of the H\"older continuity of $b_t(x)$ with the index $\gamma(x)$, which should satisfy the corresponding state-dependent balance condition $\alpha(x)+\gamma(x)>1$ uniformly. Note, however, that this analogy is not complete, since  $\gamma(x)$ and $\delta(x)$ in~\eqref{B1} may be negative. Let us make few more remarks clarifying the condition~\eqref{B1}.
\begin{enumerate}
  \item
    If the kernel $N(x,\cdot)$ is symmetric, we have $b_t(x)\equiv b(x)$, i.e.\ the time dependent part in~\eqref{B1} is only relevant  in the non-symmetric case.
  \item
    Since $\alpha(x)$ is $x$-dependent, the dynamic balance condition appears in a natural way. As a toy example take $b(x) = \sqrt{|x|}\wedge 1$, which is $\frac 12$-H\"{o}lder continuous near $x=0$ and Lipschitz continuous otherwise; in this case the dynamic balance condition requires only $\alpha(0)>1/2$.
  \item
    Condition~\eqref{B1} ensures that we may approximate the not necessarily smooth dynamically compensated drift $b_t(x)$ by a Lipschitz continuous function $B_t(x)$: There exists some $\epsilon_B>0$ such that
    \begin{align}\label{set-e14}
        |b_t(x)- B_t  (x)|
        &\leq C t^{1/\alpha(x)}  t^{-1+\epsilon_B},\quad x\in \rd,\; t\in (0,T],\\
    \label{set-e16}
        \left|B_t(x)-B_t(y)\right|
        &\leq  C t^{-1+\epsilon_B}|x-y|, \quad x,y\in \rd,\;  t\in (0,T],
    \end{align}
    see~\eqref{pfti-e22} for the definition of $B_t$ and Proposition~\ref{appA-13} for the proof of~\eqref{set-e14} and~\eqref{set-e16}. The estimate~\eqref{set-e16} says that $B_t(x)$ is Lipschitz continuous with the Lipschitz constant $\Lip(B_t)\leq C t^{-1+\epsilon_B}$, which is integrable in $t\in (0, T]$.
\end{enumerate}

We can now state and explain our main results: Theorem~\ref{t1}, Theorem~\ref{t2} and Theorem~\ref{t3}. The proofs will be deferred to Sections~\ref{pfti}-\ref{pftii}.

\begin{theorem}[Existence and uniqueness]\label{t1}
    Let $L$ be the integro-differential operator given by~\eqref{set-e02} and assume that conditions \CC, \MM, \NN\ and  \BB\ are satisfied. The martingale problem for $(L, C_\infty^2(\rd))$ is well-posed and its unique solution is a Feller process $X=(X_t)_{t\geq 0}$ whose generator $(A,D(A))$ extends the operator $(L,C^2_\infty(\rd))$. Moreover, $X$ is a strong Feller process which has a transition density $p_t(x,y)$.
\end{theorem}

In order to  obtain a representation for the transition function $p_t(x,y)$ of the process $X$ generated by the operator~\eqref{set-e02},  we need a few more concepts.  Define $\chi_t(x)$ as the solution to the following ordinary differential equation (ODE)
\begin{align}\label{set-e18}\begin{aligned}
    \frac{d}{dt}\chi_t(x) &= B_t (\chi_t(x)), \quad t>0\\
    \chi_0(x) &= x.
\end{aligned}\end{align}
Because of~\eqref{set-e16}, the constant $\Lip(B_t)$ for the mollified version $B_t(x)$ of $b_t(x)$ is integrable in $t\in (0,1]$;  this means that we can use Picard iteration to construct the unique  solution to~\eqref{set-e18}.  Next, we consider the following natural (intrinsic) drift given by the jump  kernel $\mu(x,dy)$,
\begin{align}\label{set-e20}
    \upsilon(z):=\lambda(z)\int_{\Sph^{d-1}} \ell \,\sigma(z,d\ell).
\end{align}
For fixed $z\in \rd$ we denote by  $g^z(\cdot)$  the probability density of a (not necessarily symmetric) stable random variable; its characteristic function is of the form $e^{-\psi^{z,\upsilon}(\xi)}$ with exponent
\begin{align}\label{set-e22}\begin{aligned}
    \psi^{z,\upsilon} (\xi)
    :=& -i\xi\upsilon(z)+  \int_\rd \left(1- e^{i\xi u} + i \xi \cdot u \I_{\{|u|\leq 1\}}\right) \mu(z,du)\\
    =& -i\xi\upsilon(z) + \psi^z(\xi).
\end{aligned}\end{align}

\begin{theorem}[Structure of the transition density]\label{t2}
    Let $L$ be the integro-differential operator given by~\eqref{set-e02} and assume that the conditions~\CC, \MM, \NN\ and \BB\ are satisfied. The transition  density  $p_t(x,y)$ of the process $X$ constructed in Theorem~\ref{t1} has the following representation
    \begin{align}\label{set-e24}
        p_t(x,y)
        = \frac{1}{t^{d/\alpha(x)}}g^x\left(\frac{y- \chi_t(x)}{t^{1/\alpha(x)}}\right) + R_t (x,y);
    \end{align}
    the remainder term $R_t(x,y)$ satisfies for some $\epsilon_R\in (0,1)$ the estimate
    \begin{align}\label{set-e26}
        \sup_{x\in \rd} \int_\rd |R_t(x,y)|\,dy \leq  C t^{\epsilon_R}, \quad t\in (0,T].
    \end{align}
\end{theorem}

Recall that the transition density $p_t(x,y)$ represents the probability that the process $X$ moves from the starting point $X_0 = x$ within time  $t$ to the point $X_t = y$. Theorem~\ref{t2} shows that the first (principal) part in the decomposition~\eqref{set-e24} for the transition density $p_t(x,y)$  can be represented with the help of a scaled version of an $\alpha(x)$-stable density $g^x(\cdot)$ where we make a drift correction of the starting point $x$ by moving it along the mollified drift vector field to the position $\chi_t(x)$.

\begin{example}\label{set-07}
a)\ \
Let $L$ be an operator with the symbol
\begin{align*}
    |\xi|^{\alpha(x)}
    =  C_{\alpha(x)} \int_{\rd} (1-\cos \xi u) \frac{du}{|\xi|^{d+\alpha(x)}},
    \quad  C_\alpha = \frac{\alpha 2^{\alpha}}{2\pi^{-\frac d2}} \frac{\Gamma\left(\frac 12(\alpha + d)\right)}{\Gamma\left(1-\frac 12\alpha\right)};
\end{align*}
note that $N(x,du) = \mu(x,du)=  C_{\alpha(x)} |u|^{-d-\alpha(x)}\,du$.   In this case $g^x(\cdot)$ is the probability density of an $\alpha(x)$-stable random variable, and it is well known, cf.~\cite{PT69}, that
\begin{align*}
    g^x (w)\asymp \frac{1}{(1+|w|)^{d+\alpha(x)}}, \quad w\in \rd.
\end{align*}

\medskip b)\ \
Assume that $N(x,du) = \mu(x,du)$ has an absolutely continuous spectral measure $\sigma(x,d\ell)=s(x,\ell)\,d\ell$  ($d\ell$ is the  Haar measure on the sphere $\Sph^{d-1}$), and assume that the density  satisfies $0<c_0\leq s(x,\ell) \leq c_1<\infty$.  In this case
\begin{align}\label{set-e28}
    g^x (w) \asymp \frac{1}{(1+|w- \upsilon(x)|)^{d+\alpha(x)}}, \quad w\in \rd,
\end{align}
where the \emph{intrinsic drift} $\upsilon(x)$ -- cf.~\eqref{set-e20} -- vanishes if  $s(x,\cdot)$ is symmetric, i.e.\ $s(x,\ell)=s(x,-\ell)$.

\medskip c)\ \
Assume that the spectral measure $\sigma(x,d\ell)$ is symmetric, i.e.\ $\sigma(x,A)=\sigma(x,-A)$, and satisfies the condition~\eqref{rel-e06}  for some exponent $\gamma(x)$.  Then we have,  cf.~\eqref{rel-e08},
\begin{align}\label{set-e30}
    g^x (w)\leq \frac{C(x)}{(1+|w|)^{\gamma(x)+\alpha(x)}}, \quad w\in \rd.
\end{align}
This follows from  Sztonyk's estimates~\cite{Sz10a,Sz10b,Sz11} for not necessarily rotationally symmetric $\alpha$-stable densities.
\end{example}

The estimate~\eqref{set-e26} of the residual part $R_t(x,y)$ expresses how well the transition density $p_t(x,y)$ can be approximated by a frozen (and shifted) $\alpha(x)$-stable density; this is our principal aim. Estimates for stable densities are well understood as we have seen in Example~\ref{set-07}. For the error bound~\eqref{set-e26}  it is enough to assume only~\MM, \NN\ and \BB;  these assumptions  can be thought of as a  combination of regularity assumptions and an integral domination condition for the residual kernel $\nu(x,du)$. This $L^\infty(dx)\otimes L^1(dy)$-bound for the residual part $R_t(x,y)$ is substantially weaker than the pointwise estimates for the principal part e.g.~\eqref{set-e28}, \eqref{set-e30}. This difference is intrinsic: In Section~\ref{exa} below we give two examples showing that, under just the basic assumptions,  the transition density $p_t(x,y)$ may be unbounded. This means that, in order to derive stronger bounds for the  residual part $R_t(x,y)$, one has to impose additional  assumptions  on the  model -- thus, restricting its generality. The following theorem is one possible result in this direction.

\begin{theorem}[Remainder term]\label{t3}
    Let $L$ be the integro-differential operator given by~\eqref{set-e02} and assume that the conditions~\CC, \MM, \NN\ and \BB\  are satisfied.  Assume, in addition, that the following three estimates hold for $t\in (0, T]$,  for any fixed $T>0$,  and $\alpha\in [\alpha_{\min},\alpha_{\max}]$.

    \noindent
    For every $\mathfrak{q}>0$ there exists some $\epsilon_{\mathfrak{q}}>0$ such that  for all $v\in\rd$
    \begin{align}\label{set-e32}
        \int_\rd  t^{-d/\alpha(x)} N\left(x,\big\{ u: \, |u|\geq t^{1/\alpha(x)-\mathfrak{q}}, \,|v-x-u|\leq t^{1/\alpha(x)}\big\}\right) dx
        \leq
        C t^{-1+\epsilon_{\mathfrak{q}}};
    \end{align}
    moreover,
    \begin{align}\label{set-e34}
        \int_\rd  t^{-d/\alpha} N\left(x,\big\{ u: \, |u|\geq t^{1/\alpha}, \,|v-x-u|\leq t^{1/\alpha}\big\}\right) dx
        \leq C t^{-\alpha_{\max}/\alpha_{\min}},
    \end{align}
    and there exist $\mathfrak{r}>0$ and $\epsilon_{\mathfrak{r}}>0$ such that
    \begin{align}\label{set-e36}
        \int_\rd  t^{-d/\alpha} N\left(x,\big\{ u: \, |u|\geq t^{\mathfrak{r}}, \,|v-x-u|\leq t^{1/\alpha}\big\}\right) dx\leq C t^{-1+\epsilon_{\mathfrak{r}}}.
    \end{align}
   The constants $C$ appearing in these estimates must not depend on $v\in\rd$, $\alpha\in [\alpha_{\min}, \alpha_{\max}]$ or $t\in (0, T]$.

   Under these assumptions, there exists  some $\epsilon_R\in (0,1)$ such that
    \begin{align}\label{set-e38}
        \sup_{x,y\in\rd}|R_t(x,y)|
        \leq C t^{- d/\alpha_{\min}+ \epsilon_R}, \quad t\in (0,T].
    \end{align}
\end{theorem}

\section{Examples}\label{exa}
In this section we give several examples illustrating the role of the additional assumptions~\eqref{set-e32}--\eqref{set-e36}  in  Theorem~\ref{t3}. In the first two examples the basic conditions~\CC, \MM, \NN\ and \BB\  are satisfied,  but both the kernel $p_t(x,y)$ and the residual kernel $R_t(x,y)$ are unbounded. These examples are essentially due to~\cite[Example~3.1]{Ku} and~\cite[Remark~4.23]{KRS18}; for the benefit of the reader and to emphasize the underlying structure, we include a full discussion.

\begin{example}\label{exa-01}
Let  $b\equiv 0$ and $N(x,du) =\mu^{(\alpha)}(du) +\nu(x,du)$ with
\begin{align*}
    \mu^{(\alpha)}(du)= \frac{du}{|u|^{\alpha+d}}
    \et
    \nu(x,du)= \frac{1}{1+|x|}\Big(\delta_x(du)+\delta_{-x}(du)\Big).
\end{align*}
That is, there is no drift part, the principal part of the jump kernel is  (up to a constant)  a rotationally symmetric $\alpha$-stable kernel with constant stability index $\alpha\in (0,2)$ and constant intensity coefficient $\lambda=1$; the  perturbation forces the process $X_t$  either to double its value or  to jump to the  fixed  point $0$. It is not hard to see that the conditions of Theorem~\ref{t1} hold, e.g.
\begin{align*}
    \nu(x,\{|u|>r\})
    = \int_\rd \I_{\{|u|>r\}} \,\delta_x(du)
    = \I_{\{|x|>r\}}
    \leq  r^{-\beta}
\end{align*}
for all $\beta>0$ and $r\in (0,1]$, hence~\eqref{N1} holds.

Similar to Proposition~\ref{exa-07}, see also its proof in Appendix~\ref{appE}, one can show that
\begin{align}\label{exa-e02}
    p_t(x,y)
    \geq \frac{c}{1+|x|}\int_0^tp_{t-s}^{(\alpha)}(0,y)\, ds,
\end{align}
where  $c>0$ and $p_{t}^{(\alpha)}(x,y) = p_{t}^{(\alpha)}(y-x)$ denotes the transition probability density of a rotationally invariant $\alpha$-stable L\'evy process.  Since
\begin{align*}
    p_{t}^{(\alpha)}(x,y)
    \asymp \frac{1}{t^{d/\alpha}} \frac{1}{(1+ |y-x|t^{-1/\alpha})^{d+\alpha}},
    \quad t>0,\; x,y\in \rd,
\end{align*}
we have for  $y\in \rd$ with $|y|<t^{1/\alpha}$,  $\alpha\neq d$,
\begin{align*}
    p_t(x,y)
    &\geq   \frac{c}{1+|x|}\int_0^t(t-s)^{-d/\alpha} \I_{\{|y-0|<(t-s)^{1/\alpha}\}}\, ds\\
    &=   \frac{c}{1+|x|}\int_{|y|^\alpha}^t s ^{-d/\alpha}\, ds\\
    &=   \frac{c}{1+|x|} \left(|y|^{\alpha-d}-t^{1-d/\alpha}\right).
\end{align*}
If $\alpha=d=1$ a similar calculation yields $p_t(x,y)\geq  c (1+|x|)^{-1}   |\ln (t/|y|)|$.
This shows that for $d\geq 2$ or $d=1, \alpha\leq 1$ the density $p_t(x,\cdot)$ is for every given $x\in \rd$, $t>0$ unbounded in any neighbourhood of the point $y=0$.
\end{example}

One might have the impression that the \emph{singular accumulation of mass} effect in the previous example is exotic and that it is caused artificially  by  enforcing the return of the process to a fixed point -- but such a set-up is typical for \emph{systems with resetting}, e.g.~\cite{KMSS14}. Our second example shows that essentially the same effect can be observed in stable-like models without any `artificial' perturbation terms. Note that a common feature of both examples is the fact that the corresponding L\'evy kernels $N(x,du)$ cannot be dominated by a single reference measure $N(du)$; that is, both models are \emph{essentially singular}.

\begin{example}\label{exa-03}
Consider an SDE in $\rd$
\begin{align}\label{exa-e04}
    dX_t=a(X_{t-})\, d Z_t, \quad X_0=x,
\end{align}
driven by a L\'evy process $Z=(Z^1,\dots, Z^d)$ with independent, symmetric $\alpha$-stable $Z^{i}$ marginals, $1\leq i\leq d$; $a(x)\in\real^{d\times d}$ is a matrix with H\"older continuous entries, such that
\begin{itemize}
\item $a(x)$ is invertible with bounded inverse $a(x)^{-1}$;
\item for $|x|\geq 2$: $a(x)$ is a rotation (hence, an isometry) such that $a(x)e_1=x/|x|$;
\item for $|x|\leq 1$: $a(x)=\id$;
\item for $1 < |x| <2$: $x\mapsto a(x)$ is H\"older continuous and otherwise arbitrary.
\end{itemize}

Rather than investigating the SDE~\eqref{exa-e04} directly, we use the martingale problem. By It\^o's formula, any (weak) solution to~\eqref{exa-e04} solves the martingale problem related to the operator $(L, C_\infty^2(\rd))$ where $L$ is as in~\eqref{set-e02} with $b\equiv 0$, and $N(x,du)$ is given by~\eqref{set-e04} with $\nu\equiv 0$ and $\mu(x,du)$ being the image of the L\'evy measure $\mu$ of $Z$ under the mapping $u\mapsto a^{-1}(x) u$:
\begin{align*}
    \mu(x,A)= \mu(a(x)A), \quad A\in \Bscr(\rd\setminus\{0\}).
\end{align*}
For this operator, the conditions~\eqref{M0}, \eqref{N0}, \eqref{N1}, \eqref{B0}, and~\eqref{B1} are clearly satisfied. Let us check \eqref{M1}\  and~\eqref{M2}.  Write $\sigma(d\ell) = \sum_{i=1}^d \delta_{\pm e_i} (d\ell)$ for the spectral measure of $\mu$, and observe that
\begin{align*}
    \int_{\Sph^{d-1}} (\upsilon\cdot \ell)^2 \,\sigma(d\ell)
    =
    \sum_{i=1}^d \left[(\upsilon\cdot e_i)^2 + (\upsilon\cdot (-e_i))^2\right]
    = 2,
    \quad \upsilon\in \Sph^{d-1}.
\end{align*}
Therefore,
\begin{align*}
    \inf_{x\in\rd} \inf_{\upsilon\in \Sph^{d-1}} \int_{\Sph^{d-1}}  (\upsilon\cdot \ell)^2  \,\sigma(x,d\ell)
    &= \inf_{x\in\rd} \inf_{\upsilon\in \Sph^{d-1}} \int_{\Sph^{d-1}}  (\upsilon\cdot a(x)\ell)^2  \,\sigma(d\ell)\\
    &= \inf_{x\in\rd} \inf_{\upsilon\in \Sph^{d-1}} \int_{\Sph^{d-1}}  (a(x)^{\top}\upsilon\cdot \ell)^2  \,\sigma(d\ell)\\
    &\geq \inf_{x\in\rd} \inf_{|w|\geq \|a(x)^{-1}\|^{-1}} \int_{\Sph^{d-1}}  (w\cdot \ell)^2  \,\sigma(d\ell)\\
    &= 2 \inf_{x\in\rd} \|a(x)^{-1}\|^{-2}>0,
\end{align*}
which gives~\eqref{M1}. For~\eqref{M2}, we use~\eqref{appD-e02} and~\eqref{appD-e04} (see Appendix~\ref{appD}) to get  for $|x-y|\leq 1$ and any  Lipschitz function $f$ with Lipschitz constant $C_f$
\begin{align*}
    \left|\int_{\Sph^{d-1}} f(\ell)\,\sigma(x,d\ell) - \int_{\Sph^{d-1}} f(\ell)\,\sigma(y,d\ell)\right|
    &\leq C_f |a(x)-a(y)| \int_{\Sph^{d-1}} |\ell| \,\sigma(d\ell)\\
    &\leq C_f' |x-y|^\eta.
\end{align*}
Thus, Theorem~\ref{t1} applies and shows that the process $X$ is well defined as the unique solution to the martingale problem for the operator $(L, C_\infty^2(\rd))$. With a standard argument, see e.g.~\cite[Section 5]{KK18}, we see that $X$ is the unique weak solution of the SDE~\eqref{exa-e04}.

In order to estimate the transition density $p_t(x,y)$ of the process $X$, we use the parametrix construction and the bounds obtained in the Sections~\ref{para} and~\ref{pfti} below. At this point we explain the key ingredients of the argument, and postpone the proof of the technical part (Proposition~\ref{exa-07}) to Appendix~\ref{appE}. Set
\begin{align*}
    \lambda := \int_{|z|>1}\,\mu(dz)
\end{align*}
and define for all $f\in C_\infty^2(\rd)$ auxiliary operators $\Upsilon^\tail$ and  $L^\tail$ by
\begin{align*}
    \Upsilon^\tail f(x)
    &= \int_{|z|>1}f(x+a(x)z)\,\mu(du)
    = \int_{|a(x)^{-1}u|>1}f(x+u)\,\mu(x,du), \\
    L^\tail f(x)
    &= \Upsilon^\tail f(x)-\lambda f(x),\\
    L^\trunc f(x)
    &= L f(x) - L^\tail f(x).
\end{align*}
$L^\trunc$ is an operator of the form~\eqref{set-e02} with $b=0$, $\mu(x,du)$ as before, and
\begin{align*}
    \nu(x, du)
    = -\I_{\{u:|a(x)^{-1}u|>1\}}\,\mu(x,du).
\end{align*}
Clearly,  \CC, \eqref{N0} and~\eqref{N1} are satisfied, and all other conditions for $L^\trunc$ follow from the same conditions for
$L$ and the symmetry of $\mu(dz)$. That means that $X^\trunc$ is the unique solution to the martingale problem for the operator $(L^\trunc, C_\infty^2(\rd))$ and $X^\trunc$ has a transition probability density $p_t^\trunc(x,y)$.
\end{example}

\begin{proposition}\label{exa-07}
The transition density $p_t(x,y)$ has the representation
\begin{align}\label{exa-e06}
    p_t(x,y)
    = e^{-\lambda t}p_t^\trunc(x,y) + \lambda \int_{0}^{t} e^{-\lambda s } (p_{t-s}\ast \Upsilon^\tail p^\trunc_s)(x,y)\, ds,
\end{align}
and satisfies the following bound: There exists some $t_0>0$ such that
\begin{align}\label{exa-e08}
    \int_{|x-y|\leq t^{1/\alpha}} p_t(x,y)\,dy
    \geq c, \quad t\leq t_0.
\end{align}
For $p_t^\trunc(x,y),$ the following stronger bound holds true near $0$: There exists some $r>0$ such that
\begin{align}\label{exa-e10}
    p_t^\trunc(x,y)
    \geq ct^{-d/\alpha}\I_{\{|x-y|\leq t^{1/\alpha}\}}, \quad t\leq t_0, \; |x|\leq r.
\end{align}
\end{proposition}

\begin{remark}\label{exa-09}
The identity~\eqref{exa-e06} has a natural probabilistic interpretation as renewal equation for the process $X$ at the instant where the driving noise has its first big jump
\begin{align*}
    \tau=\inf \{t: |\Delta Z_t|>1\}.
\end{align*}
A possible proof of~\eqref{exa-e06} could be based on the strong Markov property of $X$. In Appendix~\ref{appE} we give an analytic proof, which is  more natural and easier to apply in the current framework, where the process $X$ is constructed by the parametrix method.
\end{remark}

\begin{proof}[ Continuation of Example~\ref{exa-03}]
Because of our choice of $a(x)$ we have for $z=v e_1$, $v\in \real$ and $|x|\geq 2$,
\begin{align*}
   x + a(x)z
    = x + v|x|^{-1}x.
\end{align*}
If $v\in (-|x|-s^{1/\alpha},-|x|+s^{1/\alpha})$, we see that
\begin{align*}
    |x+a(x)z|
    \leq s^{1/\alpha}, \quad z=ve_1.
\end{align*}
Let $t_1 := t_0\wedge r^\alpha\wedge 1$, then for $s\leq t_1$  and $|x|\geq 2$ we have
\begin{align*}
    \left(-|x|- \tfrac 12 s^{1/\alpha},\: -|x|+ \tfrac 12s^{1/\alpha}\right)\subset(-\infty, -1).
\end{align*}
Then for $|x|\geq 2$ we get from~\eqref{exa-e10}
\begin{align*}
    \Upsilon^\tail p_s^\trunc(x,y)
    &= \int_{|z|>1} {p}_s^\trunc(x+a(x)z,y)\,\mu(dz)\\
    &= \int_{|v|>1}\int_{\Sph^{d-1}} {p}_s^\trunc(x+va(x)\ell,y)\,\sigma(d \ell)\, \frac{dv}{|v|^{\alpha+1}}\\
    &\geq \int_{|v|>1} p_s^\trunc(x+v a(x)e_1,y)\, \frac{dv}{|v|^{\alpha+1}}\\
    &\geq \int_{(-|x|-\frac{1}{2}s^{1/\alpha},-|x|+\frac{1}{2}  s^{1/\alpha})} p_s^\trunc(x+va(x)e_1,y)\,\frac{dv}{|v|^{\alpha+1}}\\
    &\geq c |x|^{-\alpha-1} s^{1/\alpha}\inf_{|x'|\leq \frac{1}{2} s^{1/\alpha}} {p}_s^\trunc(x',y).\\
    &\geq c' |x|^{-\alpha-1} s^{1/\alpha-d/\alpha}  \I_{\{|y|\leq \frac{1}{2} s^{1/\alpha}\}}.
\end{align*}
If we neglect the first (non-negative) term in~\eqref{exa-e06}, we get the following lower bound.
\begin{align*}
    p_t(x,y)
    &\geq c \int_0^t\int_{\rd}p_{t-s}(x,z) \Upsilon^\tail p^\trunc_s (z,y)\,dz\,ds\\
    &\geq c'\int_0^{t\wedge t_1}\int_{2\leq |z|\leq 4} p_{t-s}(x,z)  s^{1/\alpha-d/\alpha} \I_{\{|y|\leq s^{1/\alpha}\}}\,dz\,ds\\
    &= c' \int_0^{t\wedge t_1}\Pp_x(2\leq |X_{t-s}|\leq 4)  s^{1/\alpha-d/\alpha} \I_{\{|y|\leq s^{1/\alpha}\}} \, ds
\end{align*}
Using~\eqref{exa-e08} and the Markov property of $X$ it is easy to show that for any $0<t_1<t_2$ and $x\in \rd$
\begin{align*}
    \inf_{s\in [t_1, t_2]}\Pp_x(2\leq |X_{s}|\leq 4)>0.
\end{align*}
Thus, for arbitrary $t>0$ and $x\in \rd$ we get
\begin{align*}
    p_t(x,y)
    \geq c'' \int_0^{(t/2)\wedge  t_1} s^{1/\alpha-d/\alpha} \I_{\{|y|\leq s^{1/\alpha}\}}\, ds,
\end{align*}
which shows that the transition density $p_t(x,y)$ is unbounded near $y=0$ whenever
\begin{gather*}
    \frac{d}{\alpha}-\frac{1}{\alpha}\geq 1
    \iff \alpha+1\leq d.
\qedhere
\end{gather*}
\end{proof}

In Example~\ref{exa-03} there is no \emph{external} kernel $\nu(x,du)$, but there is still a \emph{singular accumulation of mass} effect which is principally the same as in Example~\ref{exa-01} where the perturbation term resets the process to $0$ after exponential waiting times. A similar kind of reset is achieved in Example~\ref{exa-03} directly by the large jumps of the stable-like kernel; in other words: The tails of an essentially singular  stable-like kernel can behave like a resetting perturbation for the entire kernel.

The following two examples show that, if the L\'evy kernel $N(x,du)$ is dominated by a single reference measure $N(du)$,  the conditions~\eqref{set-e32}--\eqref{set-e36} from Theorem~\ref{t3} typically hold, and the singular accumulation of mass cannot happen.

\begin{example}\label{exa-11}
    Let $\alpha(x)\equiv\alpha$ and $\mu(x, du) = m(x,u)\, \mu(du)$ and $\nu(x, du) = n(x,u)\,\nu(du)$.  We assume that $\mu(du)$ is an $\alpha$-stable L\'evy measure on $\rd$, and the (positive) reference measure $\nu(du)$ satisfies~\eqref{N1}, i.e.\ for some $\beta<\alpha$
    \begin{align*}
        \nu\left(\big\{u : |u|\geq r\big\}\right)
        \leq Cr^{-\beta}, \quad r\in (0, 1].
    \end{align*}
Assume that $0  < c \leq m(x,u)\leq C$ and $0\leq n(x,u)\leq C$. Then we have for  any $\mathfrak{r}>0$  and $v\in\rd$ \begin{align}\label{exa-e12}
\begin{aligned}
   t^{- \frac d\alpha}
   &\int_\rd N \left( x, \big\{u: |u|> t^{\mathfrak{r}}, \,  |v-x-u|\leq t^{\frac 1\alpha} \big\} \right) dx\\
   &\leq C t^{- \frac d\alpha} \int_{\rd}\int_{\rd} \I_{\{|u|> t^{\mathfrak{r}}\}}  \I_{\{|v-x-u|\leq t^{\frac 1\alpha}\}}\,\big(\mu(du)+\nu(du)\big)\, dx  \\
   &= C t^{-{\frac d\alpha}} \int_{|u|> t^{\mathfrak{r}}}\left( \int_{|v-x-u|\leq t^{\frac 1\alpha}}\,  dx \right) \big(\mu(du)  +\nu(du)\big) \\
   &= C \Big( \mu\left(\big\{u: |u|> t^{\mathfrak{r}}\big\}\right)+ \nu\left(\big\{u: |u|> t^{\mathfrak{r}}\big\}\right)\Big) \\
   &\leq C  t^{-\mathfrak{r}\alpha}, \quad t\in (0,1].
\end{aligned}
\end{align}
Taking $\mathfrak{r}=1/\alpha$, we get the  additional assumption~\eqref{set-e34} in Theorem~\ref{t3}, while taking $\mathfrak{r}<1/\alpha$ we get  the additional assumptions~\eqref{set-e32}, \eqref{set-e36}.
\end{example}

In the previous example, the stability index $\alpha$ is constant, i.e.\ there is no real difference between the additional assumptions~\eqref{set-e32}--\eqref{set-e36}. The following example shows that, for a variable index $\alpha(x)$, these additional assumptions still hold true if we assume that the jump kernel is suitably dominated.

\begin{example}\label{exa-13}
Consider a stable-like kernel~\eqref{set-e06} where $\lambda(x)\leq C$ and  the spherical part $\sigma(x, d\ell)$ is dominated by a single measure $\sigma(d\ell)$; that is,  $\sigma(x, d\ell) = s(x,\ell)\,\sigma(d\ell)$ and $s(x,\ell)\leq C$. For the ease of presentation, assume  that $\nu(x, du)\equiv0$. For $\mathfrak{b}>0$ and $\alpha\in [\alpha_{\min},\alpha_{\max}]$ we have
\begin{align}\label{exa-e14}
\begin{aligned}
  \int_\rd
  &t^{-d/\alpha} N\left(x,\big\{ u: \, |u|\geq t^{\mathfrak{b}}, \,|v-x-u|\leq t^{1/\alpha}\big\}\right) dx\\
  &\leq C  t^{-d/\alpha}  \left\{\int_{t^{\mathfrak{b}}\leq r<1}
        + \int_{r\geq 1}\right\}\int_{\mathds{S}^{d-1}} \int_{|v-x-r\ell|\leq t^{1/\alpha}} \frac{dx\,\sigma(d\ell)\,dr}{r^{1+\alpha(x)}}\\
  &\leq C  t^{-d/\alpha} \left\{ \int_{t^{\mathfrak{b}}\leq r<1}\int_{\mathds{S}^{d-1}} \left(\int_{|v-x-r\ell|\leq t^{1/\alpha}} dx\right) \frac{\sigma(d\ell)\,dr}{r^{1+\alpha_{\max}}} \right.\\
  &\quad\qquad\qquad\qquad\qquad\mbox{}+
        \left. \int_{r\geq 1}\int_{\mathds{S}^{d-1}}\left(\int_{|v-x-r\ell|\leq t^{1/\alpha}}dx\right) \frac{\sigma(d\ell)\,dr}{r^{1+\alpha_{\min}}} \right\} \\
    &\leq C\left(t^{-\mathfrak{b} \alpha_{\max}} +  1\right).
\end{aligned}
\end{align}
Taking $\mathfrak{b}=1/\alpha\leq 1/\alpha_{\min}$, we get~\eqref{set-e34}, and taking $\mathfrak{b}<1/\alpha_{\max}$ we get~\eqref{set-e36}.
For~\eqref{set-e32}, we have to modify the estimate~\eqref{exa-e14} since  in~\eqref{set-e32} the exponent $\alpha=\alpha(x)$ depends on $x$, instead of being a free parameter like $\alpha$ in~\eqref{set-e34} and \eqref{set-e36}.

Since $1/\alpha(x)\geq 1/\alpha_{\max}$ and $1/\alpha(x)$ is H\"older continuous, Corollary~\ref{appA-07} shows that
\begin{align*}
    |v-x-u|\leq t^{1/\alpha(x)} \Longrightarrow t^{1/\alpha(x)} \leq C t^{1/\alpha(v+u)}, \quad t^{-1/\alpha(x)} \leq C t^{-1/\alpha(v+u)}.
\end{align*}
For any fixed $\mathfrak{b}>0$, we get similar to~\eqref{exa-e14}
\begin{align}\label{exa-e16}
\begin{aligned}
  \int_\rd&  t^{-d/\alpha(x)}
          N\left(x,\big\{ u: \, |u|\geq t^{\mathfrak{b}}, \,|v-x-u|\leq t^{1/\alpha(x)}\big\}\right) dx\\
    &\leq C   \left\{\int_{t^{\mathfrak{b}}\leq r<1}
        + \int_{r\geq 1}\right\}\int_{\mathds{S}^{d-1}} \int_{|v-x-r\ell|\leq t^{1/\alpha(x)}}  t^{-d/\alpha(x)} \frac{dx\,\sigma(d\ell)\,dr}{r^{1+\alpha(x)}}\\
    &\leq C  \left\{ \int_{t^{\mathfrak{b}}\leq r<1}\int_{\mathds{S}^{d-1}} \left(\int_{|v-x-r\ell|\leq C t^{1/\alpha(v+r\ell)}} t^{-d/\alpha(v+r\ell)} dx\right) \frac{\sigma(d\ell)\,dr}{r^{1+\alpha_{\max}}} \right.\\
    &\qquad\qquad\mbox{}+
        \left. \int_{r\geq 1}\int_{\mathds{S}^{d-1}}\left(\int_{|v-x-r\ell|\leq C t^{1/\alpha(v+r\ell)}} t^{-d/\alpha(v+r\ell)} dx\right) \frac{\sigma(d\ell)\,dr}{r^{1+\alpha_{\min}}} \right\} \\
    &\leq C\left(t^{-\mathfrak{b} \alpha_{\max}} +  1\right).
\end{aligned}
\end{align}
On the other hand, we have  by Corollary~\ref{appA-07}
\begin{align*}
    |u|\leq t^{\mathfrak{b}} \,\&\, |v-x-u|\leq t^{1/\alpha(x)}
    \implies\left\{
    \begin{aligned}
        &t^{1/\alpha(x)} \leq C t^{1/\alpha(v)}, \quad t^{-d/\alpha(x)} \leq C t^{-d/\alpha(v)}\\
        &\text{and\ \ } r^{-1-\alpha(x)}\leq C r^{-1-\alpha(v)}, \quad r\in [t^{1/\alpha(x)-\mathfrak{q}}, t^{\mathfrak{b}}],
    \end{aligned}\right.
\end{align*}
(the interval $[t^{1/\alpha(x)-\mathfrak{q}}, t^{\mathfrak{b}}]$ is non-void if $\mathfrak{b}< 1/\alpha_{\max}$, as we may choose $\mathfrak{q}>0$  sufficiently small).  Then
\begin{align}\label{exa-e18}
\begin{aligned}
  \int_\rd&  t^{-d/\alpha(x)}
          N\left(x,\big\{ u: \, t^{1/\alpha(x)-\mathfrak{q}}\leq |u|< t^{\mathfrak{b}}, \,|v-x-u|\leq t^{1/\alpha(x)}\big\}\right) dx\\
    &\leq C  \int_{C^{-1}t^{1/\alpha(v)-\mathfrak{q}} \leq r <t^{\mathfrak{b}}}
      \int_{\mathds{S}^{d-1}} \int_{|v-x-r\ell|\leq t^{1/\alpha(v)}}  t^{-d/\alpha(v)} \frac{dx\,\sigma(d\ell)\,dr}{r^{1+\alpha(v)}}\\
    &\leq C  t^{-1+\mathfrak{q}\alpha(v)}\\
    &\leq C  t^{-1+\mathfrak{q}\alpha_{\min}}.
\end{aligned}
\end{align}
Combining~\eqref{exa-e16} and~\eqref{exa-e18}, we get~\eqref{set-e32}.
\end{example}

The Fubini argument used in the previous two examples is quite flexible and can be applied in more complicated settings. The last example in this section is also based on this argument, and illustrates the important observation that the conditions~\eqref{set-e32}--\eqref{set-e36} can be verified without $N(x, du)$ being dominated by a single measure; that is, these additional assumptions may hold in the essentially singular setting.
\begin{example}\label{exa-15}
Assume that $\alpha(x)\equiv\alpha$ and $N(x,du)$ possesses the bound
\begin{align*}
    N(x, du)
    \leq N\left(\big\{v: c(x,v)\in du\big\}\right),
\end{align*}
where
\begin{itemize}
\item
    $N(du)$ is a measure such that $N\left(\big\{v: |v|\geq r\}\right)\leq Cr^{-\alpha}$, $r\in (0,1]$;

\item
    the function $c(x, v )$ satisfies $|c(x, v )|\leq C| v |$;
\item
   $\phi(x,v)=x+c(x,v)$, as a function of $x$, is continuously differentiable, invertible, and satisfies $|(\nabla_x \phi(x,v))^{-1}|\leq C$.
\end{itemize}
Let $\psi(\cdot,v)=[\phi]^{-1}(\cdot,v)$  denote the inverse of $x\mapsto\phi(x,v)$;  similar calculations as those in~\eqref{exa-e12} yield
\begin{align*}
   t^{-\frac{d}{\alpha}}
   &\int_\rd N \left( x, \big\{u: |u|> t^{\frac 1\alpha-\sfrak}, \,  |w-x-u|\leq t^{\frac 1\alpha} \big\} \right) dx\\
      &\leq C t^{-\frac{d}{\alpha}}\int_{\rd}\int_{\rd} \I_{\{|v|> C^{-1}t^{\frac 1\alpha-\sfrak}\}}  \I_{\{|w-x-c(x,v)|\leq t^{\frac 1\alpha}\}}\, N(dv)\, dx  \\
   &= C t^{-\frac{d}{\alpha}}\int_{|v|> C^{-1} t^{\frac 1\alpha-\sfrak}}\left( \int_{|w-\phi(x,v)|\leq t^{\frac 1\alpha}}\,  dx \right) N(dv) \\
   &= C t^{-\frac{d}{\alpha}}\int_{|v|> C^{-1} t^{\frac 1\alpha-\sfrak}}\left( \int_{|w-z|\leq t^{\frac 1\alpha}}\,  \frac{dz}{|\mathrm{det}\,\nabla_x \phi\big(\psi(z,v),v\big)| } \right) N(dv) \\
   &=C N\left(\big\{v: |v|> C^{-1} t^{\frac 1\alpha-\sfrak}\big\}\right) \\
   &=C  t^{-1+ \alpha\sfrak}.
\end{align*}
Analogously to Example~\ref{exa-11} we see that the conditions~\eqref{set-e32}--\eqref{set-e36} are satisfied.

\medskip
This set of assumptions is well suited to handle SDEs driven by a truncated $\alpha$-stable noise. Consider, for instance, the SDE~\eqref{exa-e04} with an arbitrary $\alpha$-stable noise $Z$; we do not require, as in Example~\ref{exa-03}, that the coordinates are independent. The corresponding $\alpha$-stable kernel $\mu(x,du)$ is the image of the $\alpha$-stable L\'evy measure $\mu(dv)$ of $Z$ under the linear transformation $u=a(x)v$.  If $A(\cdot)\in C^1_b(\mathds{R}^d, \mathds{R}^{d\times d})$, then there exists some small $q>0$ such that for $|v|\leq q$ the mapping $\phi(x,v)=x+a(x)v$ is a (global) contraction map with $|\nabla_x(\phi(x,v)-x)|\leq \frac 12$. That is, the conditions on the function $c(x,v)=a(x)v$ and $\phi(x,v)$,  formulated above, hold true for $|v|\leq q$. If we consider the SDE~\eqref{exa-e04} with the
truncated noise
\begin{align*}
    Z_t^{\trunc}
    := Z_t-\sum_{s\leq t}\I_{\{|\Delta_s Z|\leq q\}}\Delta_s Z,
\end{align*}
then the corresponding L\'evy kernel $N(x,du)$  has the form~\eqref{set-e04} with a (finite) kernel $\nu(x, du) = -\mu(\{v: |v|\geq q, a(x)v\in du\})$ and satisfy the conditions~\eqref{set-e32}--\eqref{set-e36}. Note that the truncation does not improve the regularity property of the kernel: If $\mu(x,du)$ is essentially singular (e.g.\ as in Example~\ref{exa-03}), then $N(x, du)$ is essentially singular, as well.

This shows that \emph{for an SDE driven by truncated stable noise, the $L^\infty(dx)\otimes L^\infty(dy)$-bounds for the  residual kernel $R_t(x,y)$,  hence for the entire kernel $p_t(x,y)$, can even be obtained in  essentially singular settings. This requires a proper combination of smoothness assumptions on the jump coefficient and smallness assumptions on the truncation level. In~\cite[Section 3]{KRS18}, similar results were obtained for the SDE~\eqref{exa-e04} driven by a  vector of independent one-dimensional \textup{(}truncated\textup{)} $\alpha$-stable processes. The argument we have presented here is free from any structural limitations on the driving noise.}
\end{example}

\section{The parametrix construction}\label{para}

\subsection{An Ansatz.}\label{ans}

We want to construct  the transition density $p_t(x,y)$ of the unknown process $X$ as  a fundamental solution of the following Cauchy problem
\begin{align}\label{para-e02}\begin{aligned}
    \left(\frac{d}{dt} - L_x\right)u(t,x) &= f(t,x), && (t,x)\in (0,T]\times\rd,\\
    u(0,x) &= f_0(x), &&x\in\rd,
\end{aligned}\end{align}
where $L_x = L$ is an integro-differential operator $L$ of the form~\eqref{set-e02}.

\begin{definition}\label{para-03}
    A \emph{fundamental solution} of the Cauchy problem~\eqref{para-e02} is a function $p_t(x,y) := p(0,t;x,y)$ defined for $(t,x,y)\in (0,T]\times\rd\times\rd$ such that the formula
    \begin{align*}
        u(t,x) = \int_{\rd} p_t(x,y) f_0(y)\,dy
    \end{align*}
    is a solution to the homogeneous equation $\left(\frac d{dt} - L_x\right) u(t,x)=0$ with initial condition $u(0,x)=f_0(x)$ and any $t\in (0,T]$ and $f_0\in C_b(\rd)$. Moreover, we assume that $\int_{\rd} p_t(x,y)\,dy = 1$.
\end{definition}
Definition~\ref{para-03} means that $p_t(x,y)$ satisfies the Kolmogorov backward equation
\begin{align}\label{para-e04}
    \left(\frac{d}{dt} - L_x\right)p_t(x,y)=0
\end{align}
such that $\int_\rd p_t(x,y)\,dy = 1$ and $p_t(x,\cdot)\to \delta_x$ as $t\to 0+$ in the sense of vague convergence. In abuse of language, we also call $p_t(x,y)$ a fundamental solution to the Kolmogorov backward equation.

In order to construct a fundamental solution we make the following \emph{Ansatz}. Assume, for a moment, that $p_t(x,y)$ is a fundamental solution and that we can write it  in the form
\begin{align}\label{para-e06}
    p_t(x,y) = p_t^0(x,y) + r_t(x,y)
\end{align}
where $p_t^0(x,y)$ is a suitable \emph{zero-order approximation} of the unknown $p_t(x,y)$, and $r_t(x,y)$ is the remainder term. We assume that $p_t^0(x,\cdot)\to\delta_x$  vaguely as $t\to 0$. Formally applying the operator $\left(\frac{d}{dt} - L_x\right)$ to both sides of~\eqref{para-e06} yields
\begin{align*}
    0 = \left(\frac{d}{dt} - L_x\right) p_t(x,y) = \left(\frac{d}{dt} - L_x\right) p_t^0(x,y) + \left(\frac{d}{dt} - L_x\right) r_t(x,y), \quad t>0,
\end{align*}
and we get  the following equality for the remainder term
\begin{align*}
    \left(\frac{d}{dt} - L_x\right) r_t(x,y) = \Phi_t(x,y) := -\left(\frac{d}{dt} - L_x\right)p_t^0(x,y).
\end{align*}
Now we define\footnote{ We use `$\cstar$' to denote the space-time convolution of two kernels $a_t(x,y)$ and $b_t(x,y)$: $a\cstar b_t(x,y) = \int_0^t\int_{\rd} a_s(s,z)b_{t-s}(z,y)\,dz\,ds$.}
\begin{align}\label{para-e08}
    p\cstar \Phi_t(x,y) := \int_0^t \int_\rd p_s(x,z)\Phi_{t-s}(z,y)\,dz\,ds;
\end{align}
since $p_t(x,y)$ is a fundamental solution, it is not hard to see that
\begin{align*}
    \left(\frac{d}{dt}-L_x\right) p\cstar \Phi_t(x,y) = \Phi_t(x,y).
\end{align*}
This indicates that $p\cstar \Phi_t(x,y)$ is a suitable candidate for the remainder term $r_t(x,y)$. If we plug it into~\eqref{para-e06}, we obtain a fixed-point equation
\begin{align*}
    p_t(x,y) = p_t^0(x,y) + p\cstar \Phi_t(x,y)
\intertext{which can be (formally) solved by iteration:}
    p_t(x,y) = p_t^0(x,y) + \sum_{k=1}^\infty p^0\cstar\Phi^{\cstar k}_t(x,y).
\end{align*}
Obviously, we have to find an  admissible zero-order approximation $p_t^0(x,y)$ and to prove the convergence -- in a suitable function space -- of the formal series expansion $\sum_{k=1}^\infty \Phi^{\cstar k}_t(x,y)$. If we, finally, \emph{define} $p_t(x,y)$ through the above series expansion, a major problem will be the regularity of $p_t(x,y)$ which is needed to make sense of the expression $(\frac d{dt}-L_x)p_t(x,y)$, i.e.\ to verify that $p_t(x,y)$ is indeed a fundamental solution in the sense of Definition~\ref{para-03}. In order to solve this problem, we will introduce the notion of an \emph{approximate fundamental solution}; this will be discussed in Section~\ref{pti}.

\subsection{Functional analytic framework}\label{faf}

Assume, for the moment, that we have found a zero-order approximation $p_t^0(x,y)$, $(t,x,y)\in (0,\infty)\times\rd\times\rd$, which is continuously differentiable in $t$ and of class $C^2_\infty(\rd)$ in $x$; moreover, we assume that
\begin{align}\label{para-e10}
    p_t^0(x, \cdot)\to \delta_x \quad\text{ vaguely as\ \ } t\to 0+.
\end{align}
This allows us to define the function
\begin{align}\label{para-e12}
    \Phi_t(x,y):=-\left(\frac{d}{dt}-L_x\right)p_t^0(x,y),\quad t>0,\; x,y\in\rd.
\end{align}
As we have seen in the previous section, the key to the construction of $p_t(x,y)$ is the following Fredholm integral equation of the second kind
\begin{align}\label{para-e14}
    p_t(x,y)
    = p^0_t(x,y)+p\cstar \Phi_t(x,y).
\end{align}
It is  convenient to treat~\eqref{para-e14} within the following functional analytic framework. Consider the Banach space  $L^\infty(dx)\otimes L^1(dy)$ of kernels $\Upsilon(x,y)$ satisfying
\begin{align*}
    \|\Upsilon\|_{\infty,1}
    := \sup_{x\in \rd}\int_{\rd}|\Upsilon(x,y)|\, dy<\infty.
\end{align*}
Each kernel $\Upsilon\in L^\infty(dx)\otimes L^1(dy)$ generates a bounded linear operator in the space $B_b =  B_b(\rd)$  of bounded measurable functions,
\begin{align*}
    \Upsilon^\op f(x) = \int_{\rd}\Upsilon(x,y) f(y)\, dy, \quad f\in B_b(\rd),
\end{align*}
with the norm $\|\Upsilon\|_{\infty,1}$ which is the same as the operator norm $\|\Upsilon^\op\|_{B_b\to B_b}$. Denote $P_t=p_t^\op, P^0_t=(p_t^0)^\op, t>0$, the operators corresponding to the unknown transition probability kernel $p_t(x,y)$ and its zero-order approximation. Then~\eqref{para-e14} can be equivalently written as
\begin{align}\label{para-e16}
    P_t
    = P^0_t+\int_0^t P_{t-s}\Phi_s^\op\, ds, \quad t>0.
\end{align}
In Section~\ref{dec}, we will choose $p^0_t(x,y)$ in such a way, that the kernel $\Phi_t(x,y)$ satisfies for some  $\epsilon_\Phi>0$ and for a fixed $T>0$
\begin{align}\label{para-e18}
    \sup_{x\in \rd} \int_\rd \left|\Phi_t(x,y)\right| dy\leq  C  t^{-1+\epsilon_\Phi}, \quad t\in (0,T].
\end{align}
The latter inequality can be written as bound for the operator norm
\begin{align}\label{para-e20}
    \left\|\Phi_t^\op\right\|_{B_b\to B_b}\leq  C  t^{-1+\epsilon_\Phi}, \quad t\in (0,T],
\end{align}
which allows us to treat~\eqref{para-e16}, in a standard way, as a Volterra equation with a mild (integrable) singularity. Recall that each kernel $p_t(x,y), t>0$ is supposed to be a probability density, hence it is necessary that
\begin{align}\label{para-e22}
    \|P_t\|_{B_b\to B_b}\leq C.
\end{align}
The unique solution to~\eqref{para-e16} which satisfies~\eqref{para-e22} on a fixed time interval  $t\in (0, T]$ can be interpreted as a classical Neumann series
\begin{align}\label{para-e24}
\begin{aligned}
    P_t
    &=P_t^0+\sum_{k=1}^\infty\;\;\idotsint\limits_{0<s_1<\dots <s_{k}<t} P_{t-s_k}^0\Phi_{s_k-s_{k-1}}^\op\dots\Phi_{s_1}^\op\,ds_1\dots ds_{k}\\
    &= P_t^0+ \int_0^t P_{t-s}^0\Psi_s^\op\, ds,
\end{aligned}
\end{align}
where the operator
\begin{align*}
    \Psi_t^\op
    :=\Phi_t^\op + \sum_{k=2}^\infty\;\;\idotsint\limits_{0<s_1<\dots <s_{k-1}<t} \Phi^\op_{t-s_{k-1}}\dots\Phi_{s_1}^\op\,ds_1\dots ds_{k-1}
\end{align*}
corresponds to the kernel
\begin{align}\label{para-e26}
    \Psi_t(x,y) := \sum_{k=1}^\infty \Phi^{\cstar k}_t(x,y).
\end{align}
The series~\eqref{para-e24}, \eqref{para-e26} converges uniformly in $t\in (0, T]$ in the operator norm $\|\cdot\|_{B_b\to B_b}$ and the norm $\|\cdot\|_{\infty,1}$, respectively. This follows easily from~\eqref{para-e20}, since
\begin{align}\label{para-e28}
\begin{aligned}
    \|\Phi^{\cstar k}_t\|_{\infty,1}
    &= \bigg\|\;\;\idotsint\limits_{0<s_1<\dots <s_{k-1}<t} \Phi^\op_{t-s_{k-1}}\dots\Phi_{s_1}^\op \,ds_1\dots ds_{k-1}\bigg\|_{B_b\to B_b}\\
    &\leq\idotsint\limits_{0<s_1<\dots <s_{k-1}<t} \|\Phi^\op_{t-s_{k-1}}\|_{B_b\to B_b}\cdot\ldots\cdot\|\Phi_{s_1}^\op\|_{B_b\to B_b}\,ds_1\dots \, ds_{k-1}\\
    &\leq  C^k  \idotsint\limits_{0<s_1<\dots <s_{k-1}<t} (t-s_{k-1})^{-1+\epsilon_\Phi}\cdot\ldots\cdot s_1^{-1+\epsilon_\Phi}\,ds_1\dots \, ds_{k-1}\\
    &=  t^{-1+k \epsilon_\Phi} \frac{(C\Gamma(\epsilon_\Phi))^k}{\Gamma(k\epsilon_\Phi)}.
\end{aligned}
\end{align}
The Gamma function $\Gamma(z)$ behaves asymptotically like $\sqrt{2\pi} z^{z-\frac 12} e^{-z}\gg C^z$ as $z\to \infty$. This  asymptotic estimate  yields
\begin{align}\label{para-e30}
    \|\Psi_t\|_{\infty,1}\leq Ct^{-1+\epsilon_\Phi}, \quad t\in (0, T].
\end{align}
Our choice of $p^0_t(x,y)$ will also ensure that
\begin{align}\label{para-e32}
   \sup_{t\in (0, T]}\sup_{x\in \rd} \int_\rd \left|p_t^0(x,y)\right|dy\leq  C \iff \sup_{t\in (0, T]}\|P_t^0\|_{B_b\to B_b}\leq C.
\end{align}
Combining this with~\eqref{para-e30}, we obtain~\eqref{para-e06} with $r=p^0\cstar \Psi$ which satisfies
\begin{align}\label{para-e34}
    \|r_t\|_{\infty,1}
    = \|r_t^\op\|_{B_b\to B_b}
    \leq \int_0^t\|P^0_{t-s}\|_{B_b\to B_b} \|\Psi_{s}^{\op}\|_{B_b\to B_b}\, ds
    \leq Ct^{\epsilon_\Phi},
    \quad t\in (0, T].
\end{align}
The proof of Theorem~\ref{t2} is essentially based on these representations and estimates.

\medskip
This functional-analytic framework can not only be used for the $L^\infty(dx)\otimes L^1(dy)$ estimates, but also for the other bounds mentioned above. In order to prove Theorem~\ref{t3}, we need  $L^\infty(dx)\otimes L^\infty(dy)$-estimates, i.e.\ bounds of the  operator norm $\|\cdot\|_{L^1\to B_b}$. Similar to~\eqref{para-e34}, we can get such an  $L^\infty(dx)\otimes L^\infty(dy)$-estimate for the residual term $r_t(x,y)$ from~\eqref{para-e06}, but this requires further assumptions which we will explain now. Because of our choice of $p^0_t(x,y)$ we have for all $t\in (0,T]$
\begin{align}\label{para-e36}
    \sup_{x, y\in \rd} p_t^0(x,y)
    \leq  C t^{-d/\alpha_{\min}}
    \iff
    \|P_t^0\|_{L^1\to B_b}
    \leq  C t^{-d/\alpha_{\min}}.
\end{align}
This leads to the bound
\begin{align}\label{para-e38}
    \|P_{t-s}^0\Psi_s^\op\|_{L^1\to B_b}
    \leq  C (t-s)^{-d/\alpha_{\min}}s^{-1+\epsilon_\Phi},
    \quad 0< s < t.
\end{align}
This expression cannot be directly integrated because of the strong singularity at the point $s=t$. This difficulty can be resolved in the following way. Assume, for a while, that for some $\epsilon_R\in (0, \epsilon_\Phi]$  and all $t\in (0,T]$
\begin{align}\label{para-e40}
    \sup_{x,y\in \rd} |\Phi_t(x,y)|
    \leq  C t^{-d/\alpha_{\min}-1+\epsilon_R}
    \iff \|\Phi^\op_t\|_{L^1\to B_b}\leq  C t^{-d/\alpha_{\min}-1+\epsilon_R}
\end{align}
and that, in addition,  for all $t\in (0,T]$
\begin{align}\label{para-e42}
    \sup_{y\in \rd} \int_{\rd} |\Phi_t(x,y)|\,dx
    \leq  C t^{-1+\epsilon_R} \iff \|\Phi^\op_t\|_{L^1\to L^1}\leq  C t^{-1+\epsilon_R}.
\end{align}
Let $k\geq2$, $s_0 := 0<s_1<\dots<s_{k-1}<t =: s_k$ be arbitrary. If  $j$ is such that $s_j-s_{j-1}=\max_{i=1, \dots, k}(s_i-s_{i-1})$, then
\begin{align*}
    &\|\Phi^\op_{t-s_{k-1}}\dots\Phi_{s_1}^\op\|_{L^1\to B_b}\\
    &\qquad\leq \|\Phi^\op_{t-s_{k-1}}\dots\Phi_{s_{j+1}-s_j}^\op\|_{B_b\to B_b}\cdot \|\Phi_{s_{j}-s_{j-1}}^\op\|_{L^1\to B_b} \cdot \|\Phi_{s_{j-1}-s_{j-2}}^\op\dots\Phi_{s_1}^\op\|_{L^1\to L^1}\\
    &\qquad\leq C^k\cdot (s_j-s_{j-1})^{-d/\alpha_{\min}}\cdot (t-s_{k-1})^{-1+\epsilon_R}\cdot\ldots \cdot s_1^{-1+\epsilon_R}\\
    &\qquad\leq k^{d/\alpha_{\min}} \cdot C^k\cdot t^{-d/\alpha_{\min}}\cdot (t-s_{k-1})^{-1+\epsilon_R}\cdot\ldots\cdot s_1^{-1+\epsilon_R}.
\end{align*}
 From this we deduce -- in the same way as we got~\eqref{para-e30} from~\eqref{para-e28} -- that
\begin{align}\label{para-e44}
     \|\Psi^\op_t\|_{L^1\to B_b}
     \leq  C t^{-d/\alpha_{\min}-1+\epsilon_R};
\end{align}
 (Note that the extra term $k^{d/\alpha_{\min}}$ is not important for the convergence of the series because of the rapid growth of the Gamma function).
Combining~\eqref{para-e44} with the bound
\begin{align}\label{para-e46}
    \|P_t^0\|_{L^1\to L^1}\leq  C
\end{align}
 -- this is yet to be proved --, we get a  further estimate of type~\eqref{para-e38}:
\begin{align}\label{para-e48}
    \|P_{t-s}^0\Psi_s^\op\|_{L^1\to B_b}\leq  Cs^{-d/\alpha_{\min}-1+\epsilon_R}.
\end{align}
Combining both estimates~\eqref{para-e38} and~\eqref{para-e48} finally yields
\begin{align}
    \sup_{x,y\in \rd}|r_t(x,y)|
    &\notag=\|r_t^\op\|_{L^1\to B_b}\\
    &\label{para-e50}\leq C \int_0^{t/2} (t-s)^{-d/\alpha_{\min}} s^{-1+\epsilon_R}\,ds + C \int_{t/2}^t s^{-d/\alpha_{\min}-1+\epsilon_R}\,ds\\
    &\notag\leq Ct^{-d/\alpha_{\min}+\epsilon_R}.
\end{align}
This is the backbone of the proof of Theorem~\ref{t3}. In Section~\ref{s82} we verify the assumptions~\eqref{para-e40}, \eqref{para-e42}, \eqref{para-e46}, which will give~\eqref{para-e50}, and then only a minor technical issue remains: To compare the (explicit) zero-order approximations in~\eqref{para-e06} and~\eqref{set-e24}. Notice that the  estimates~\eqref{para-e40}, \eqref{para-e46} only require  the basic assumptions  \CC, \MM, \NN\ and \BB\   from Theorem~\ref{t1}; it is the  `dual' $L^1\to L^1$ bound~\eqref{para-e42} which needs the additional non-trivial assumptions~\eqref{set-e32}--\eqref{set-e36}.

\bigskip
\begin{center}\bfseries\label{blind-ref}
    \fbox{From now on we assume that $T=1$.}
\end{center}

\bigskip\noindent
This is only a technical assumption which simplifies our calculations; any other choice of $T > 0$ will only affect constants.

\section{Proof of Theorem~\ref{t1} -- convergence of the parametrix series}\label{pfti}

In this section we will prove the key estimate~\eqref{para-e18} which guarantees the convergence of the (formal) parametrix series, cf.\ Section~\ref{faf}. The main result of this section is Lemma~\ref{pfti-05}.

\subsection{A road map}\label{road}
As we have seen in Section~\ref{para}, a key problem is to choose the function $p^0_t(x,y)$  which is the zero-order approximation for $p_t(x,y)$.  This is a technically difficult problem; therefore we want to give the reader a road map how to proceed.

 The standard approach from the parametrix method for second-order parabolic PDEs~\cite{Fr64} applied to our pseudo-differential operator~\eqref{set-e02} means that we have to freeze the `coefficients', leading  (for the principal part)  to a family of operators of the form
\begin{align}\label{pfti-e02}
    \mathcal{L}^z f(x)
    = \int_{\rd\setminus \{0\}} \left(f(x+u)-f(x)  - \nabla f(x) \cdot u \I_{\{|u|\leq 1\}}\right) \mu(z,du), \quad z\in \rd.
\end{align}
For every $z\in\rd$ this is the generator of a L\`evy process and we denote its transition density by $p^z_t(y-x)$. The `classical' zero-order approximation $p_t^0(x,y)$ is then
\begin{align}\label{pfti-e04}
    p^z_t(y-x)\big|_{z=y}.
\end{align}
This choice of the zero-order approximation means that we neglect all `inessential' parts of the generator while the infinitesimal characteristics of the `principal' part are frozen at the \emph{endpoint} $y$. The reason for this choice is dictated by the necessity to apply the operator $L$ in the variable $x$. For a systematic exposition of this approach for pseudo-differential operators we refer to~\cite{EIK04}.

The classical `frozen at the endpoint' choice~\eqref{pfti-e04}  is often  inappropriate in the essentially singular setting.  Consider, e.g., the simple model from Example~\ref{exa-03}: If  $\alpha+1\leq d$, one can easily show that
\begin{align}\label{pfti-e06}
    \int_{\rd}p^y_t(y-0)\, dy
    = \infty,
\end{align}
which means that $p_t^0(0,y)  = p_t^y(y-0)$ is not integrable, and so a very bad approximation of the probability density $p_t(0, y)$.

Another hidden limitation becomes visible if the `principal part' given by the kernel $\mu(x,du)$ does not dominate the `drift' -- e.g.\ if $b(x)$ is non-trivial and $\alpha(x)<1$. Then  the approach which was developed for diffusions (see above) is bound to fail; more precisely, the error term $\Phi_t(x,y)$ does not admit the bound~\eqref{para-e18} which is crucial for the entire approach. This is not unexpected, since in this case the gradient part of the generator dominates the integral part, and thus it is not `inessential'.  This observation leads to the following natural modification of the method, proposed in~\cite{KK18,Ku18}: Take, instead of~\eqref{pfti-e04}, the following zero-order approximation $p_t^0(x,y)$
\begin{align}\label{pfti-e08}
    p^z_t(\kappa_t(y)-x)|_{z=y}.
\end{align}
The expression $\kappa_t(y)$ is a `flow corrector' which takes into account the deterministic motion caused by the velocity field $-b(y)$; there are some technical difficulties which we will not discuss at this point -- e.g.\ $b(y)$ need not be Lipschitz continuous and one has to consider the dynamically compensated drift $b_t(y)$ which contains the effect of small jumps. The zero-order approximation~\eqref{pfti-e08} is a good choice if $\alpha(\cdot)\equiv \alpha$ is constant and the polarization measure is  comparable with the uniform distribution of the sphere $\Sph^{d-1}$. The latter requirement is crucial: A thorough check of the proofs in~\cite{KK18,Ku18} reveals that they rely on the following property  of the transition density $p^z_t(y-x)$:  For $|y-x|>t^{1/\alpha}$,
\begin{align}\label{pfti-e10}
    |\nabla_x^k p^z_t(y-x)|
    \leq C t^{-d/\alpha-k/\alpha}\left(\frac{|y-x|}{t^{1/\alpha}}\right)^{-d-\alpha-k},
\end{align}
and it is important that the exponent of $|y-x|$  goes down  by $k$ if we differentiate $k$ times. This property \emph{need not hold} for an $\alpha$-stable measure with singular polarization measure. An example is Example~\ref{exa-03} with $d=2$: The transition density of the L\'evy process $Z$ is
\begin{align*}
    p_t(y-x)
    = g_t(y_1-x_1) g_t(y_2-x_2), \quad x=(x_1,x_2), \; y=(y_1,y_2)
\end{align*}
where $g_t(\cdot)$ denotes the density of the one-dimensional components. We have
\begin{align*}
     \partial_{x_1}^k  p_t(y-x)
    = (-1)^k g_t^{(k)}(y_1-x_1)g_t(y_2-x_2),
\end{align*}
i.e.\ taking the derivative in $x_1$ does not lead to a faster decay if $|y_1-x_1|\asymp t^{1/\alpha}$ and $|y_2-x_2|\gg t^{1/\alpha}$.

These two observations can be summarized as follows: For essentially singular L\'evy-type models the choice of $p_t^0(x,y)$ must reflect the non-locality of the operator. If we use the classical `frozen at the endpoint' zero-order approximation, various hardly controllable tail effects may occur.  This explains the main idea of our approach: To offset the non-locality, we do not use the full operator~\eqref{pfti-e02} as principal part, but we use a $t$-dependent operator having a jump kernel $\mu(x,du)$ with $|u| \leq t^{\zeta(z)}$ where $\zeta(z)<1/\alpha(z)$.  This choice will resolve both difficulties which we have mentioned earlier on.

\subsection{Choice of the zero-order approximation}\label{zero}

Let us now proceed to the details of the construction. Fix some $\sfrak\in (0, 1/(2\alpha_{\max}))$  --  the particular value will be specified later on --, and set
\begin{align}\label{pfti-e12}
    \zeta(x)=\frac{1}{\alpha(x)}-\sfrak.
\end{align}
Observe that $\frac12 \frac 1{\alpha_{\max}} < \zeta_{\min} \leq \zeta(x)  \leq \zeta_{\max} \leq \frac 1{\alpha_{\min}}-\sfrak$. For $z\in \rd$ we define
\begin{align}\label{pfti-e14}
    \psi_{t}^{z,\cut}(\xi)
    := \int_{|u|\leq t^{\zeta(z)}} \left(1- e^{i \xi \cdot u} + i \xi \cdot u \I_{\{ |u|\leq t^{1/\alpha(z)}\}}\right) \mu(z,du).
\end{align}
Note that $\int_0^t  \psi_{s}^{z,\cut}(\xi)\, ds$  is the  characteristic exponent of an additive process (in the sense of It\^o,  i.e.\ a process with independent but not necessarily stationary increments, cf.~\cite[p.~3]{sato}); because of  Proposition~\ref{appC-03}, we can calculate the corresponding time-inhomogeneous transition function $p_t^{z,\cut}(x):= p_{0,t}^{z,\cut}(x)$ by the inverse Fourier transform
\begin{align}\label{pfti-e16}
    p_t^{z,\cut}(x)
    := (2\pi)^{-d} \int_\rd e^{- i \xi x -\int_0^t \psi_r^{z,\cut}(\xi)\,dr}\,d\xi, \quad x\in \rd,\; t>0.
\end{align}

The symbol $\psi_{t}^{z,\cut}(\xi)$ is related to the symbol of the integro-differential operator $L$ from~\eqref{set-e02} in the following way: $\psi_{t}^{z,\cut}(\xi)$ does not have a drift and its jump measure contains only the small (depending on time) jumps of the principal component $\mu(z,du)$ of the jump measure $N(z,du)$ of $L$.  In order to construct the zero-order approximation, we will use the function $p_t^{z,\cut}(x)$ instead of $p_t^{z}(x)$.

\medskip
Next, we define the `flow corrector' term $\kappa_t(y)$ in~\eqref{pfti-e08}.
For every $0<\mfrak < 2-\alpha_{\max}$ there exists a $C^1$-function $\theta: \rd \to (0,2)$ such that
\begin{align}\label{pfti-e18}
    \alpha(x) \leq \alpha(x)+\frac 12\mfrak \leq \theta(x) \leq \alpha(x) + \mfrak \leq \alpha_{\max} +\mfrak < 2
    \quad\text{for all $x\in\rd$}.
\end{align}
Fix, for the moment, $\mfrak$ and the corresponding $\theta(x)$; the  particular value of $\mfrak$ will be specified later on. By~\eqref{M0} and~\eqref{B1} we have
\begin{align*}
    \gamma(x)+\alpha(x)=1+\hfrak>1
    \et
    \alpha(x)(\delta(x)+1)=1,
\end{align*}
and, by construction,
\begin{align}\label{pfti-e20}
    \inf_{x\in \rd} (\gamma(x)+\theta(x))>1
    \et
    \inf_{x\in \rd} \theta(x)(\delta(x)+1)>1.
\end{align}
Pick any $C_c^\infty$-function $\phi: \rd \to [0, \infty)$ with $\supp\phi = \overline{B(0,1)}$, $\int_{B(0,1)} \phi(x)\,dx = 1$, set $\phi_t(x):= t^{-d}\phi(t^{-1} x)$, and define
\begin{align}\label{pfti-e22}
    B_t (x):= \int_\rd   b_t(y) \phi_{t^{1/\theta(x)}}(x-y)\,dy.
\end{align}
This approximation enjoys the properties~\eqref{set-e14} and~\eqref{set-e16}, see Proposition~\ref{appA-13}.

Denote by $\kappa_t(y)$  the  solution to the Cauchy  problem
\begin{align}\label{pfti-e24}
\begin{split}
    \frac{d}{dt}\kappa_t (y)&=- B_t (\kappa_t(y)), \quad t> 0, \\
    \kappa_0(y)& =y.
\end{split}
\end{align}
Since the Lipschitz constant satisfies $\Lip(B_t)\leq C t^{-1+\epsilon_B}$,  the Lipschitz constant  is integrable in $t\in (0,1]$. Therefore, the solution $\kappa_t(y)$ is unique; compare this with the solution $\chi_t(x)$ to~\eqref{set-e18}.

\medskip
We can now define the zero-order approximation
\begin{align}\label{pfti-e26}
    p^0_t(x,y):=  p_{t}^{y,\cut}(\kappa_t (y)-x);
\end{align}
this definition combines the original `frozen at the endpoint' parametrix idea, the idea from~\cite{KK18,Ku18} to compensate the gradient part by `flow corrector', and the new idea of the dynamic cut-off of the jump part.
 In the definition of $p^0_t(x,y)$ we use the ``inverse flow'' $\kappa_t(y)$ (acting on $y$) rather than the direct flow $\chi_t(x)$ (acting on $x$) since we want to apply the operator $L$ to the function $x\mapsto p^0_t(x,y)$ -- and, therefore, a simple argument is preferable.  The estimates from Corollary~\ref{appA-21} will enable us to switch between $|\kappa_t(y)-x|$ and $|\chi_t(x)-y|$.

\medskip

The function $p^0_t(x,y)$ possesses the following basic property, see Proposition~\ref{appC-17} below:
\begin{align}\label{pfti-e28}
    \int_\rd \left|p_t^0(x,y) - \frac{1}{t^{d/\alpha(x)}} g^x \left(\frac{y- \chi_t(x)}{t^{1/\alpha(x)}}\right)\right| dy
    \leq C t^{\epsilon_R},
\end{align}
where $g^z$ is the density of a (not necessarily symmetric) $\alpha(z)$-stable random variable with characteristic exponent~\eqref{set-e22} and drift~\eqref{set-e20}. The estimate~\eqref{pfti-e28} implies, in particular,  (cf.~\eqref{para-e32} and the proof of Corollary~\ref{appC-15})
\begin{align}\label{pfti-e30}
    \sup_{t\in (0, 1]}\sup_{x\in \rd}\int_\rd p^0_t(x,y)\,dy\leq C.
\end{align}

\subsection{Decomposition and estimates for $\Phi_t(x,y)$. Proof of~\eqref{para-e20}}\label{dec}

Starting from the zero-order approximation $p_t^0(x,y)$, we  define the kernel $\Phi_t(x,y)$ as in~\eqref{para-e12}. As we have explained in Section~\ref{faf}, we have to check that the corresponding family of operators $\Phi_t^{\op}$, $t\in (0, 1]$,  satisfies~\eqref{para-e20}.

The symbol $\psi_{t}^{z,\cut}(\xi)$ introduced in~\eqref{pfti-e14} defines a pseudo-differential operator which has the following integro-differential representation:
\begin{align}\label{pfti-e32}
    L^{t,z,\cut} f(x)
    :=\int_{|u|\leq t^{\zeta(z)}} \left(f(x+u)- f(x)-  \nabla f(x) \cdot u \I_{\{ |u|\leq t^{1/\alpha(z)}\}}\right) \mu(z,du)
\end{align}
for $f\in C^2_\infty(\rd)$. The operator $L^{t,z,\cut}$ is the time-dependent generator of an  additive process,  and the  transition density~\eqref{pfti-e16} satisfies
\begin{align}\label{pfti-e34}
    \frac{d}{dt} p_t^{z,\cut}(w-x) = L^{t,z,\cut}_x   p_t^{z,\cut}(w-x), \quad z,w\in \rd,
\end{align}
 see Proposition~\ref{appC-05}.  This identity and~\eqref{pfti-e24} give
\begin{align}\label{pfti-e36}
\begin{aligned}
    \frac{d}{dt} p_t^0(x,y)
    &= L^{t,y,\cut}_x p_t^0(x,y)   -  \tfrac{d\kappa_t(y)}{dt}\nabla_x p_t^0(x,y)\\
    &=L^{t,y,\cut}_x p_t^0(x,y) + B_t (\kappa_t(y)) \nabla_x p_t^0(x,y).
\end{aligned}
\end{align}
From the definition of $\Phi_t(x,y)$,
we know
\begin{align*}
    \Phi_t(x,y)
    = -\left(\tfrac{d}{dt} - L^{t,y,\cut}_x\right)p^0_t(x,y) + \left(L_x-L^{t,y,\cut}_x\right) p^0_t(x,y).
\end{align*}
If we use  the integro-differential representations of the operators $L$ and $L^{t,y,\cut}$, cf.~\eqref{set-e02} and~\eqref{pfti-e32}, this  gives
\begin{align}\label{pfti-e38}
\begin{split}
  \Phi_t(x,y) &= -B_t(\kappa_t (y))\cdot \nabla_x p_t^0(x,y)+ \biggl[b_t(x)\cdot \nabla_x p_t^0(x,y)\\
    &\qquad\mbox{}+ \int_\rd \left(p^0_t(x+u,y) -p^0_t (x,y)  - \nabla_x p^0_t(x,y) \cdot u \I_{\{|u|\leq t^{1/\alpha(x)}\}}\right)\mu(x,du)\\
    &\qquad\mbox{}+ \int_\rd \left(p^0_t(x+u,y) -p^0_t (x,y)  - \nabla_x p^0_t(x,y) \cdot u \I_{\{|u|\leq t^{1/\alpha(x)}\}}\right)\nu(x,du)\\
    &\qquad\mbox{}- \int_{|u|\leq t^{\zeta(y)}} \left(p^0_t(x+u,y) -p^0_t (x,y)  - \nabla_x p^0_t(x,y) \cdot u \I_{\{|u|\leq t^{1/\alpha(y)}\}}\right)\mu(y,du)\biggr].
    \end{split}
\end{align}
We will  split $\Phi_t(x,y)$ and separate two groups of its components. The first group, i.e.\ the terms $A_1$--$A_6$ in~\eqref{pfti-e66} below, will admit pointwise bounds, while the second group $B_1$, $B_2$ will have only $L^1(dy)$-integral bounds. In order to formulate the key estimates, we first introduce some auxiliary kernels and basic inequalities.

Consider the family of functions
\begin{align}\label{pfti-e40}
    f_{t,a,c}(x):=t^{-ad} e^{-c |x|t^{-a}},\quad a,c>0,\; t>0,\; x\in\rd.
\end{align}
Using this family we define two types of kernels which will be used in our estimates:
\begin{align}\label{pfti-e42}
    K^{0;c}_t(x,y) &:= f_{t,\zeta(y),c}(\kappa_t(y)-x).
\\\label{pfti-e44}
    K^{1;c}_t(x,y) &:= \I_{\{|y-\chi_t(x)|\leq  t^\delta\}}  f_{t,\zeta(x),c}( y-\chi_t(x))
\\\notag&\qquad\mbox{}  +  t^{-N} \I_{\{|y-\chi_t(x)|> t^\delta\}} f_{t,\zeta_{\min},c} (y-\chi_t(x)),
\end{align}
where $N$  is a sufficiently large integer which will be chosen later on,  and
\begin{align}\label{pfti-e46}
 \delta< \frac{1}{2\alpha_{\max}}.
\end{align}

\begin{remark}[Properties of $K^{0;c}_t(x,y)$ and $K^{1;c}_t(x,y)$]\label{pfti-03}
For the readers' convenience, we collect some properties of the kernels $f_{t,a,c}(x)$, $K^{0;c}_t(x,y)$ and $K^{1;c}_t(x,y)$ which will be used in the sequel; if no further argument is given, the proof is obvious from the definition of the kernel(s).  Throughout we assume that $x,y\in\rd$ and $t\in (0,1]$.
\begin{enumerate}
\item
For any two parameters $c>c'>0$ there is a constant $C>0$ such that
\begin{align}
\label{pfti-e48}
    |x|f_{t,a,c} (x)
    &\leq  C t^a  f_{t,a,c'} (x).
\intertext{In particular, for any $b>0$}
\label{pfti-e50}
    |\kappa_t(y)-x|^{b} K^{0;c}_t(x,y)
    &\leq C t^{b\zeta(y)} K^{0;c'}_t(x,y).
\end{align}

\item
For any $c>0$ and $\ell, k\geq 0$ there exists a constant $C=C_{k,c}$ such that
\begin{align}\label{pfti-e52}
    \left|\partial_t^\ell \nabla^k_x p_t^0(x,y)\right|
    \leq C t^{-\sfrak d -k/\alpha(y)-\ell} K^{0;c}_t(x,y).
\end{align}
\begin{proof}
    Use Proposition~\ref{appC-07} and the definition~\eqref{pfti-e42} of $K^{0;c}_t(x,y)$.
\end{proof}

This is the key estimate, which is the substitute for~\eqref{pfti-e10}. Note that $K^{0;c}_t(x,y)$ decays exponentially as  $|\kappa_t(y)-x| \asymp |y-x| \to \infty$ (see Corollary~\ref{appA-21}); this is even better than the polynomial decay in~\eqref{pfti-e10}. This is due to the dynamic cut-off of the jump measure in the frozen-coefficient operator~\eqref{pfti-e32}.

\item
For any $c>c'$ there exists a constant $C>0$ such that
\begin{align}\label{pfti-e54}
    K^{0;c}_t(x+u, y)
    \leq C K^{0;c'}_t(x,y) \quad \text{for all\ \ } |u|\leq t^{\zeta(y)}.
\end{align}
\begin{proof}
    Use that $K^{0;c}_t(x,y)$ is given by  the exponential family~\eqref{pfti-e40}.
\end{proof}

\item
For any $c>c'$ there exists a  constant $C>0$  and $N>d/\alpha_{\min}$  such that
\begin{align}\label{pfti-e56}
    K^{0;c}_t(x+u, y) \leq C K^{1;c'}_t(x,y) \quad \text{for all\ \ } |u|\leq t^{\zeta(x)}.
\end{align}
\begin{proof}
    See Proposition~\ref{appB-05}.
\end{proof}

\item
For any  $c\geq c'$, any bounded H\"older continuous function $w(\cdot): \rd \to (0,\infty)$, and $v: \rd\to (0,\infty)$  such that $0<v_{\min}\leq v(x)\leq v_{\max}<\infty$, we have
\begin{align}\label{pfti-e58}
    t^{v(y)(w(x)-w(y))}  K^{0;c}_t(x,y)
    \leq  C K^{1;c'}_t(x,y),
\\\label{pfti-e60}
    t^{v(y)(w(x)-w(y))}  p^{0}_t(x,y)
    \leq  C (p^{0}_t(x,y)+  K^{1;c'}_t(x,y)).
\end{align}
The exponent $N$ appearing in the definition of $K^{1;c'}_t(x,y)$ satisfies $N>w_{\max} v_{\max}+d/\alpha_{\min}$.
\begin{proof}
    See Proposition~\ref{appB-03}.
\end{proof}

\item
There is a constant $C<\infty$ such that
\begin{align}\label{pfti-e62}
    \sup_{x\in\rd}\sup_{t\in (0,1]} \int_\rd K^{1;c}_t(x,y) \,dy
    \leq  C.
\end{align}
\begin{proof}
    See Proposition~\ref{appB-07}.
\end{proof}

\item
The following estimate holds true
\begin{align}\label{pfti-e64}
    f_{t,\zeta(x),c}( y-\chi_t(x)) \leq K^{1;c}_t(x,y).
\end{align}
The exponent $N$ appearing in the definition of $K^{1;c}_t(x,y)$ satisfies $N> \frac{d}{\alpha_{\min}}- \frac{d}{\alpha_{\max}}$.
\begin{proof}
    If $|y-\chi_t(x)|> t^\delta$, then the estimate follows from  the inequality
    \begin{align*}
        f_{t,\zeta(x),c}(y-\chi_t(x))
        &\leq t^{-\zeta_{\max} d} e^{-c |y-\chi_t(x)|t^{-\zeta_{\min}}} \\
        &\leq t^{(\zeta_{\min}-\zeta_{\max}) d} f_{t,\zeta_{\min},c} (y-\chi_t(x))
    \end{align*}
    and the definition of $K^{1;c}_t(x,y)$; if $|y-\chi_t(x)|\leq  t^\delta$, then the estimate follows directly from the definition of $K^{1;c}_t(x,y)$.
\end{proof}
\end{enumerate}
\end{remark}

We write the formula~\eqref{pfti-e38} for $\Phi_t(x,y)$ in the following form:
\begin{align}\label{pfti-e66}
    \Phi_t(x,y)
    &=: A_1+ \dots + A_6+ B_1+B_2,
\end{align}
where we use the abbreviations
\begin{align*}
    A_1 &= - p^0_t (x,y)  \int_{|u|> t^{\zeta(y)}}\mu(x,du),\\
    A_2 &= - p^0_t (x,y)  \int_{|u|> t^{1/\alpha(x)}}\nu(x,du), \\
    A_3 &= \left(b_t(x)-B_t(\kappa_t (y))\right)\cdot \nabla_x p_t^0(x,y)\\
    A_4 &= \int_{|u|\leq t^{1/\alpha(x)}} \left(p^0_t(x+u,y) -p^0_t (x,y)  - \nabla_x p^0_t(x,y) \cdot u\right)\nu(x,du),\\
    A_5 &= \int_{|u|\leq t^{\zeta(y)}} \left(p^0_t(x+u,y) -p^0_t (x,y)  - \nabla_x p^0_t(x,y) \cdot u\right)\left(\mu(x,du)- \mu(y,du)\right)\\
    A_6 &=\int_\rd \nabla_x p^0_t(x,y) \cdot u \left(\I_{\{|u|\leq t^{\zeta(y)}\}}  -  \I_{\{|u|\leq t^{1/\alpha(x)}\}}\right)\mu(x,du)\\
    &\qquad \mbox{}- \int_\rd \nabla_x p^0_t(x,y) \cdot u \left(\I_{\{|u|\leq t^{\zeta(y)}\}}  -  \I_{\{|u|\leq t^{1/\alpha(y)}\}}\right)\mu(y,du),\\
    B_1 &= \int_{|u|>t^{\zeta(y)}}p^0_t(x+u,y)\,\mu(x,du),\\
    B_2 &= \int_{|u|> t^{1/\alpha(x)}}  p^0_t(x +u,y)\,\nu(x,du).
\end{align*}
Notice that the terms $A_i=A_i(t;x,y)$ and $B_k = B_k(t;x,y)$ are actually functions depending on $t$, $x$ and $y$. If no confusion is possible, we want to keep notation simple and use the shorthand $A_i$ and $B_k$.

The following lemma is the main result of this section.
\begin{lemma}\label{pfti-05}
There exists  $\epsilon_\Phi>0$ such that for all $x,y\in\rd$ and $t\in (0,1]$
\begin{align}
\label{pfti-e68}
    |A_i(t,x,y)|\leq C t^{-1+ \epsilon_\Phi} (p^0_t(x,y)+ K^{1;c}_t (x,y)), \quad 1\leq i\leq 6,
\\\label{pfti-e70}
    \sup_{x\in\rd} \int_\rd |B_{i}(t,x,y)| \,dy
    \leq C t^{-1 + \epsilon_\Phi}, \quad i=1,2.
\end{align}
The kernel $K^{1;c}_t (x,y)$ is given by~\eqref{pfti-e44}  with $N> d+ \frac{d+3}{\alpha_{\min}}$.
\end{lemma}

Lemma~\ref{pfti-05} guarantees, in particular, the key estimate~\eqref{para-e18} which is needed for the convergence of the parametrix series.
\begin{corollary}\label{pfti-07}
    By~\eqref{pfti-e30} and~\eqref{pfti-e62} the pointwise bounds~\eqref{pfti-e68} yield integral bounds similar to~\eqref{pfti-e70}:
\begin{align}\label{pfti-e72}
    \sup_{x\in\rd} \int_\rd |A_{i}(t,x,y)| \,dy
    \leq C t^{-1 + \epsilon_\Phi}, \quad 1\leq i\leq 6.
\end{align}
\end{corollary}

\begin{proof}[Proof of Lemma~\ref{pfti-05}.]
Set
\begin{align}\label{pfti-e74}
    \sfrak
    = \frac{\eta \min\left\{\epsilon_\nu, \epsilon_B\right\}}{16 d}.
\end{align}
Without loss of generality we may assume that the parameters $\epsilon_\nu$, $\epsilon_B$ and $\eta$ are small enough, so that $\sfrak\in (0, 1/(2\alpha_{\max}))$, see the beginning of Section~\ref{zero}.

We show that~\eqref{pfti-e68} holds with some $\epsilon_i>0$,  $1\leq i\leq 6$,  and that~\eqref{pfti-e70} holds with some $\tilde{\epsilon}_j$, $j=1,2$, respectively. Then we choose $\epsilon_\Phi$ as the minimum of $\epsilon_i$, $1\leq i\leq 6$, and $\tilde{\epsilon}_j$, $j=1,2$.

\medskip\noindent
\emph{Estimate of $A_1$}:
By the scaling property~\eqref{set-e08} of $\mu(x,du)$,  the definition of $\zeta(y)$,  and~\eqref{pfti-e60}  we get
\begin{align*}
    |A_1|
    =  C  t^{-\zeta(y)\alpha(x)} p^0_t(x,y)
    &= C t^{-1+\sfrak\alpha(y)} t^{\zeta(y)(\alpha(y)-\alpha(x))} p^{0}_t(x,y)\\
    &\leq  C t^{-1+\sfrak\alpha_{\min}}  \left(p^0_t(x,y)+K^{1;c}_t(x,y)\right);
\end{align*}
in the in the definition of $K^{1;c}_t(x,y)$ we use $N>\frac{d+2}{\alpha_{\min}}$. This proves~\eqref{pfti-e68} for $i=1$ and $\epsilon_1:= \sfrak\alpha_{\min}$.

\medskip\noindent
\emph{Estimate of $A_2$}:
Using~\eqref{N1}, \eqref{pfti-e52}, and~\eqref{pfti-e58} we get
\begin{align*}
    |A_2|
    \leq  C t^{-\sfrak d} t^{-\beta(x)/\alpha(x)} K^{0;c}_t(x,y)
    &\leq  C t^{-\sfrak d} t^{-\beta(x)/\alpha(x)} K^{1;c}_t(x,y)\\
    &\leq  C t^{-1- \sfrak d+ \epsilon_\nu /\alpha_{\max}}K^{1;c}_t(x,y);
\end{align*}
in the definition of  $K^{1;c}_t(x,y)$ we use $N>\frac{d}{\alpha_{\min}}$, moreover we observe that $\alpha(x)-\beta(x)\leq \epsilon_\nu$ uniformly for all $x$. Note that $\sfrak < \frac{\epsilon_\nu}{\alpha_{\max}d}$ since $\sfrak$ is given by~\eqref{pfti-e74}; thus we have $\epsilon_2:= \frac{\epsilon_\nu}{\alpha_{\max}} - \sfrak d>0$ and we  get~\eqref{pfti-e68} for $i=2$ and $\epsilon_2$ defined above.

\medskip\noindent
\emph{Estimate of $A_3$}:
From~\eqref{set-e14} and~\eqref{set-e16} we infer
\begin{align*}
    |b_t(x) - B_t (\kappa_t(y))|
    &\leq | b_t(x) - B_t (x)| + |B_t (x) - B_t(\kappa_t(y))|\\
    &\leq Ct^{1/\alpha(x)} t^{-1+\epsilon_B} + C  t^{-1+\epsilon_B} | x-\kappa_t(y)|.
\end{align*}
Combining this with the estimate for $\nabla p_t^0(x,y)$, cf.~\eqref{pfti-e52}, we get
\begin{align*}
    \left|A_3\right|
    &\leq C t^{-\sfrak d-1+\epsilon_B}\left(t^{1/\alpha(x)-1/\alpha(y)} + t^{\zeta(y)-1/\alpha(y)} \frac{|\kappa_t(y)-x|}{t^{\zeta(y)}}\right)K^{0;c}_t(x,y)\\
    &\leq  C t^{-\sfrak d-1+\epsilon_B} t^{1/\alpha(x)-1/\alpha(y)}  K^{0;c}_t(x,y) + C t^{-\sfrak(d+1)-1+\epsilon_B} K^{0;c'}_t(x,y)\\
    & \leq  C t^{-1-\sfrak(d+1)+\epsilon_B} K^{1;c''}_t(x,y);
\end{align*}
in the definition of  $K^{1;c}_t(x,y)$ we require $N> \frac{d+1}{\alpha_{\min}}$. In the penultimate line we use~\eqref{pfti-e50} with $a=1$ and $c'<c$, and in the last estimate we use~\eqref{pfti-e58}.  Since  $\sfrak< \frac{\epsilon_B}{d+1}$, we have $\epsilon_3:=\epsilon_B -\sfrak(d+1)>0$, and~\eqref{pfti-e68} holds for $i=3$ and $\epsilon_3$ defined above.

\medskip\noindent
\emph{Estimate of $A_4$}:
We begin with the integrand appearing in $A_4$. Using Taylor's theorem we get
\begin{align}\label{pfti-e76}
\begin{aligned}
     h(t,x,y,u)
     \equiv h(u)
     :=& p_t^0(x+u,y) - p_t^0(x,y)- \nabla_x  p_t^0(x,y)\cdot u\\
    =& \frac{1}{2} \sum_{1\leq i,j\leq d} u_i u_j \int_0^1 \partial_{ij}^2 p_t^0(x+su,y)\,ds.
\end{aligned}
\end{align}
In order to keep notation simple, we will write $h(u)$ instead of $h(t,x,y,u)$ unless this leads to misunderstandings. By~\eqref{pfti-e52} and~\eqref{pfti-e56}  we have for all $|u|\leq  t^{\zeta(x)}$ and any  $N>\frac{d}{\alpha_{\min}}$ in the definition of  $K^{1;c}_t(x,y)$,
\begin{align}\label{pfti-e78}\begin{aligned}
    \left|h(u)\right|
    &\leq C t^{ -\sfrak d -2/\alpha(y)} |u|^2 \int_0^1 K^{0;c}_t(x+us,y)\,ds\\
    &\leq  C t^{ -\sfrak d -2/\alpha(y)} |u|^2  K^{1;c'}_t(x,y).
\end{aligned}\end{align}

In order to estimate the integral w.r.t.\ $du$  in $A_4$, we use~\eqref{N1} to get
\begin{align*}
    \int_{|u|\leq r} |u|^2\, \nu(x,du)
    &= \int_0^\infty \nu\left(x,\{u: r\geq |u|\geq \sqrt{w}\}\right) dw\\
    &= 2 \int_0^r  \rho\,\nu\left(x,\{u: r\geq |u|\geq \rho\}\right)d\rho\\
    &\leq C \int_0^r \rho^{1-\beta(x)} \,d\rho
    \leq  C r^{2-\beta(x)}.
\end{align*}
If we combine this estimate with~\eqref{pfti-e78}, we obtain
\begin{align*}
    |A_4| &\leq  C  t^{- \sfrak d -2/\alpha(y)}  K^{1;c'}_t(x,y) \int_{|u|\leq t^{1/\alpha(x)}} |u|^2 \,\nu(x,du)\\
    &\leq  C  t^{- \sfrak d -2/\alpha(y)+ (2-\beta(x))/\alpha(x)}K^{1;c'}_t(x,y)\\
    &= C  t^{-1+ (1- \sfrak d -\beta(x)/\alpha(x))}K^{1;c'}_t(x,y);
\end{align*}
Arguing as in the estimate of $A_2$, we see that~\eqref{pfti-e68} holds  for $i=4$ with $\epsilon_4:= \epsilon_2+ 1$.

\medskip\noindent
\emph{Estimate of $A_5$}:  Fix $x,y\in \rd$, $t\in (0,1]$ and recall the definition of the function $h(t,x,y,u)$ from~\eqref{pfti-e76}.   We want to apply Proposition~\ref{appD-03}  with $h(u)=h(t,x,y,u)$, $z_1=x$ and $z_2=y$. Let us check that the conditions~\eqref{appD-e06} and~\eqref{appD-e08} are satisfied.
\begin{enumerate}
\item[\eqref{appD-e06}:] The estimate~\eqref{pfti-e78} shows that~\eqref{appD-e06} holds with
\begin{align*}
    C_h = C_h(t,x,y) = C t^{-\sfrak d -2/\alpha(y)}K^{1;c'}_t(x,y).
\end{align*}

\item[\eqref{appD-e08}:] We need to bound $|h(u)-h(v)|$ for   $|u|=|v|\leq t^{\zeta(y)}$.   We  use~\eqref{pfti-e52} to get for any $k\geq 0$
\begin{align}\label{pfti-e80}
\begin{split}
    &\left|\nabla^k_x p^0_t(x+u,y) -  \nabla^k_x p^0_t(x+v,y)\right|\\
    &\qquad\leq   C |u-v| \int_0^1 \left|\nabla^{k+1}_x  p^0_t(x+us+ (1-s)v,y)\right| ds\\
    &\qquad\leq  C |u-v| \cdot t^{-\sfrak d - (k+1)/\alpha(y)} \int_0^1 K^{0;c}_t(x+us+ (1-s)v,y)\,ds\\
    &\qquad\leq  C|u-v| \cdot t^{-\sfrak d - (k+1)/\alpha(y)} K^{0;c'}_t(x,y),
    \end{split}
\end{align}
in the last estimate we use~\eqref{pfti-e54}. Therefore,
\begin{align*}
    \left|h(u) - h(v)\right|
    &\leq  C \sum_{1\leq i,j\leq d} \left|u_i u_j - v_i v_j\right| \int_0^1 \left|\partial^2_{ij} p_t^0(x+su,y)\right| ds\\
    &\qquad\mbox{} + \sum_{1\leq i,j,\leq d}  |v_i v_j |\int_0^1 \left|\partial^2_{ij} p_t^0(x+su,y)- \partial^2_{ij} p_t^0(x+sv,y)\right|ds.
\end{align*}
 Let $u=\rho \ell_1$ and $v=\rho\ell_2$, $\ell_1,\ell_2\in \Sph^{d-1}$,  $\rho\leq t^{\zeta(y)}$.  We get
\begin{align*}
  \left|h(\rho \ell_1) - h(\rho \ell_2)\right|&\leq  C |\ell_1-\ell_2| \rho^2  t^{-\sfrak d - 2/\alpha(y)} K^{0;c'}_t(x,y)\\
    &\qquad\mbox{} + C |\ell_1-\ell_2| \rho^3  t^{-\sfrak d - 3/\alpha(y)} K^{0;c'}_t(x,y),
\end{align*}
where we used that $\rho\leq t^{\zeta(y)}$  along with the definition~\eqref{pfti-e12} of $\zeta(y)$.
Finally,
\begin{align*}
    \left|\rho^{-2}  (h(\rho \ell_1) - h(\rho \ell_2))\right|
    \leq C  |\ell_1-\ell_2|  t^{-\sfrak(d+1) - 2/\alpha(y)} K^{0;c'}_t(x,y).
\end{align*}
This proves~\eqref{appD-e08} with
\begin{align*}
    C_h = C_h(t,x,y) = C  t^{-\sfrak(d+1) - 2/\alpha(y)}K^{0;c'}_t(x,y).
\end{align*}
\end{enumerate}

\medskip\noindent
We can now apply Proposition~\ref{appD-03} with $r= t^{\zeta(y)}$ and get
\begin{align}\label{pfti-e82}
\begin{split}
    |A_5|
    &\leq  C  t^{ -\sfrak(d+1) -2/\alpha(y)}K^{0;c'}_t(x,y) \left(t^{\zeta(y)(2-\alpha(x))} + t^{\zeta(y)(2-\alpha(y))}\right)
    |\log t|(|x-y|^\eta  \wedge 1) \\
    &= C t^{ -\sfrak(d+3)} \left(t^{-\zeta(y)\alpha(x)} + t^{-\zeta(y)\alpha(y)}\right)|\log t| (|x-y|^\eta  \wedge 1)  K^{0;c'}_t(x,y).
 \end{split}
 \end{align}
By Remark~\ref{appA-11}
\begin{align}\label{pfti-e84}
    |y-x|^\eta
    \leq |\kappa_t(y)-x|^\eta + |\kappa_t(y)-y|^\eta
    \leq |\kappa_t(y)-x|^\eta  +  Ct^{\frac 12\eta},
\end{align}
and an application of~\eqref{pfti-e50} with $a=\eta$ yields
\begin{align*}
    |A_5|
    &\leq C t^{ -\sfrak(d+3)}  \left(t^{-\zeta(y)\alpha(x)} + t^{-\zeta(y)\alpha(y)}\right)
    |\log t|\left(|\kappa_t(y)-x|^\eta  + C t^{\frac 12\eta}\right) K^{0;c'}_t(x,y)\\
    &\leq C t^{ -\sfrak (d+3)}  \left(t^{-\zeta(y)\alpha(x)} + t^{-\zeta(y)\alpha(y)}\right)
    |\log t| \left(t^{\zeta(y) \eta} + t^{\frac 12\eta}\right) K^{0;c''}_t(x,y).
\end{align*}
From the definition~\eqref{pfti-e12} of $\zeta(y)$ we have
\begin{align*}
    t^{-\zeta(y)\alpha(x)} + t^{-\zeta(y)\alpha(y)}
    &= t^{-\zeta(y)\alpha(y)} \left(t^{-\zeta(y)(\alpha(x)-\alpha(y))} +1\right)\\
    &= t^{-1+\alpha(y)\sfrak} \left(t^{-\zeta(y)(\alpha(x)-\alpha(y))} +1\right).
\end{align*}
Using~\eqref{pfti-e58} and the monotonicity of the kernel $K_t^{0;c}$ in $c$, we see
\begin{align*}
    |A_5|
    &\leq  C t^{-1 -\sfrak(d+3-\alpha(y))} \left(t^{\zeta(y) \eta} + t^{\eta /2}\right)
    |\log t| \left(t^{-\zeta(y)(\alpha(x)-\alpha(y))} +1\right) K^{0;c''}_t(x,y)\\
    &\leq C  t^{-1 -\sfrak(d+3-\alpha_{\min})} |\log t|  \left(t^{\zeta_{\min} \eta} +  t^{\frac 12\eta}\right)  K^{1;c'''}_t(x,y)\\
    &\leq C  t^{-1 -\sfrak(d+3-\alpha_{\min})} t^{\frac{\eta}{4(\alpha_{\max}\vee 1)} } K^{1;c'''}_t(x,y);
\end{align*}
here we pick $N>\frac{d+2}{\alpha_{\min}}$ in the definition of all kernels of type $K^{1;c}_t(x,y)$, and we use that
\begin{align*}
    \zeta_{\min}= \frac{1}{\alpha_{\max}}-\sfrak = \frac{1}{2\alpha_{\max}} + \frac{1}{2\alpha_{\max}}-\sfrak > \frac{1}{2\alpha_{\max}}.
\end{align*}
Note that for $\sfrak$ as in~\eqref{pfti-e74} we have $\sfrak< \frac{\eta}{2\alpha_{\max} (d+3-\alpha_{\min})}$, thus $\epsilon_5:= \frac{\eta}{4(\alpha_{\max}\vee 1)} -\sfrak(d+3-\alpha_{\min}) >0$, and~\eqref{pfti-e68} holds for $i=5$ and  $\epsilon_5$.

\medskip\noindent
\emph{Estimate of $A_6$}:
We rewrite the expression $A_6$ in the following form
\begin{align*}
    A_6
    & = -\int_\rd \nabla_x p^0_t(x,y) \cdot u \left(\I_{\{|u|\leq t^{\zeta(y)}\}}  -  \I_{\{|u|\leq t^{1/\alpha(y)}\}}\right)\left(\mu(y,du)- \mu(x,du)\right)\\
    &\qquad\mbox{}+ \int_\rd \nabla_x p^0_t(x,y) \cdot u\left(\I_{\{|u|\leq t^{1/\alpha(y)}\}}- \I_{\{|u|\leq t^{1/\alpha(x)}\}}\right)\mu(x,du)\\
    &= A_{61}+ A_{62},
\end{align*}
and estimate the terms $A_{61}$ and $A_{62}$ separately. For the first term we apply Proposition~\ref{appD-03} with the function
\begin{align*}
    h(u) \equiv h(t,x,y,u)
    := u \I_{\{ t^{1/\alpha(y)} <|u|\leq t^{\zeta(y)}\}}
    \et u= \rho \ell,\; \ell\in\Sph^{d-1},
\end{align*}
(recall that $t<1$ and $\zeta(y) < 1/\alpha(y)$).
Let $t^{1/\alpha(y)} \leq \rho\leq t^{\zeta(y)}$ and $\ell,\ell_1,\ell_2 \in \Sph^{d-1}$; the following inequalities show that~\eqref{appD-e06} and~\eqref{appD-e08} hold true with the constant $C_h=t^{1/\alpha(y)}$:
\begin{align*}
    |\rho^{-2} h(\rho \ell)| \leq  \rho^{-1}\leq t^{-1/\alpha(y)},\\
    \left|\rho^{-2} \left(h(\rho \ell_1)- h(\rho \ell_2)\right)\right|
    \leq \rho^{-1} |\ell_1-\ell_2|
    \leq t^{-1/\alpha(y)}\left|\ell_1-\ell_2\right|.
\end{align*}
Thus, we can apply Proposition~\ref{appD-03}   with $z_1=x$, $z_2=y$. Using~\eqref{pfti-e52} we get
\begin{align*}
    |A_{61}|
    &\leq C \left|\nabla_x p^0_t(x,y)\right|  t^{-1/\alpha(y)} \left(t^{\zeta(y)(2-\alpha(y))} +  t^{\zeta(y)(2-\alpha(x))}\right)
    |\log t| \left(|x-y|^\eta\wedge 1\right)\\
    &\leq  C  t^{-\sfrak d-2/\alpha(y)}  K^{0;c}_t(x,y)\left(t^{\zeta(y)(2-\alpha(y))} +  t^{\zeta(y)(2-\alpha(x))}\right)|\log t| \left(|x-y|^\eta\wedge 1\right).
\end{align*}
Up to the factor $t^{\sfrak}$ this estimate coincides with the estimate~\eqref{pfti-e82} for $A_5$. We may, therefore, follow from this point onwards literally the arguments for the estimate of $A_5$, and get
\begin{align*}
    |A_{61}|\leq C   t^{-1-\sfrak (d+2-\alpha(y))} \left(t^{\zeta_{\min} \eta} + t^{\eta /2}\right) |\log t|  K^{1;c'}_t(x,y),
\end{align*}
for the kernel $K^{1;c}_t(x,y)$ with  $N>\frac{d+2}{\alpha_{\min}}$. Proceeding in the same way as for $A_5$, we deduce that $A_{61}$ satisfies~\eqref{pfti-e68}  with $i=6$ and $\epsilon_{61}:=\epsilon_5+ \sfrak$.

The estimate of the second term $A_{62}$ is easier. Let $\delta= \frac{1}{8}\min\{\epsilon_\nu, \epsilon_B\}$; any such $\delta$ satisfies the condition  $\delta<\frac{1}{4}\min\{\frac{\epsilon_\nu}{\alpha_{\max}}, \epsilon_B\}$ required in Proposition~\ref{appA-17}; this $\delta$ is later used in the key inequalities~\eqref{appA-e48}. We consider two cases:

\medskip\noindent
\emph{Case 1}:  $|\kappa_t(y)-x|\leq t^{\delta}$. The gradient $|\nabla_x p^0_t(x,y)|$ can be estimated by~\eqref{pfti-e52}.  Using the $\eta$-H\"older continuity of $\alpha(x)$ we have because of the upper estimate in~\eqref{appA-e48} with $w(x) = 1/\alpha(x)$
\begin{align*}
    |A_{62}|
    &\leq C  t^{-\sfrak d-\frac 1{\alpha(y)}} K^{0;c}_t(x,y) \int_\rd |u| \left|\I_{\{|u|\leq t^{1/{\alpha(x)}}\}} - \I_{\{|u|\leq t^{1/{\alpha(y)}}\}}\right| \mu(x,du)\\
    &\leq C  t^{-\sfrak d-\frac 1{\alpha(y)}} K^{0;c}_t(x,y) \int_{e^{-c t^{\eta\delta/2}} t^{1/{\alpha(x)}} \leq |u|\leq e^{c t^{\eta\delta/2}} t^{1/{\alpha(x)}}} |u|\, \mu(x,du)\\
\intertext{and using the representation for the kernel $\mu(x,du)$, cf.~\eqref{set-e06}, we have in the case $\alpha(x)\neq 1$}
    |A_{62}|
    &\leq  C t^{-\sfrak d-\frac 1{\alpha(y)}} K^{0;c}_t(x,y) \int_{e^{-c t^{\eta\delta/2}} t^{1/{\alpha(x)}}}^{e^{c t^{\eta \delta/2}}t^{1/{\alpha(x)}}} \rho^{-\alpha(x)}\,d\rho \\
    &= C t^{-1-\sfrak d + \frac 1{\alpha(x)} -\frac 1{\alpha(y)}} K^{0;c}_t(x,y)  \frac{1}{|1-\alpha(x)|}\left|e^{c t^{\eta \delta/2}\frac{1-\alpha(x)}{\alpha(x)}} - e^{-c t^{\eta \delta/2} \frac{1-\alpha(x)}{\alpha(x)}}\right| \\
    &\leq C t^{-1-\sfrak d + \frac 1{\alpha(x)} -\frac 1{\alpha(y)}+ \frac 12 \eta \delta} K^{0;c}_t(x,y) \\
    &\leq C t^{-1-\sfrak d + \frac 12\eta \delta} K^{0;c}_t(x,y) \\
    &\leq C t^{-1-\sfrak d + \frac 12\eta \delta} K^{1;c'}_t(x,y),
\end{align*}
for  $K^{1;c'}_t(x,y)$ with $N>\frac{d}{\alpha_{\min}}$.
In the last line  we use~\eqref{appA-e48}   with $w(x) = 1/\alpha(x)$, and then~\eqref{pfti-e56} with $u=0$. It is easily seen that this estimate is still valid if $\alpha(x)=1$. Note that for $\delta$ as above  the parameter   $\sfrak$ given by~\eqref{pfti-e74} satisfies $\sfrak<\frac{\delta\eta}{2d}$.

\medskip\noindent
\emph{Case 2}:  $|\kappa_t(y)-x|> t^{\delta}$.  Observe that
\begin{align*}
    |u|\left|\I_{\{|u|\leq t^{1/\alpha(x)}\}}-\I_{\{|u|\leq t^{1/\alpha(y)}\}}\right|
    &\leq t^{1/\alpha_{\max}}\left|\I_{\{|u|\leq t^{1/\alpha(x)}\}}-\I_{\{|u|\leq t^{1/\alpha(y)}\}}\right|\\
    &= t^{1/\alpha_{\max}}\left|\I_{\{|u|> t^{1/\alpha(x)}\}} - \I_{\{|u|> t^{1/\alpha(y)}\}}\right|\\
    &\leq 2 t^{1/\alpha_{\max}} \I_{\{|u|> t^{1/\alpha_{\min}}\}}.
\end{align*}
Using~\eqref{pfti-e52}, the definition of $K_t^{0;c}(x,y)$ and the tail behaviour of $\mu(x,du)$, yield
\begin{align*}
    |A_{62}|
    &\leq C t^{-\sfrak d-1/\alpha(y)} f_{t,\zeta(y),c}(\kappa_t(y)-x) t^{1/\alpha_{\max}} \mu\left((x,\{{|u|>t^{1/\alpha_{\min}}}\}\right)\\
    &\leq C t^{-\sfrak d-1/\alpha(y)} f_{t,\zeta_{\min},c'}(y-\chi_t(x)) t^{1/\alpha_{\max}} t^{-\alpha_{\max}/\alpha_{\min}}\\
    &\leq  C t^{-N} f_{t,\zeta_{\min},c'}(y-\chi_t(x))\\
    &\leq C K_t^{1;c'}(x,y),
\end{align*}
for  $K^{1;c'}_t(x,y)$ with  $N>d +\frac{3}{\alpha_{\min}}$. In the second line we use Corollary~\ref{appA-21}. Thus, $A_{62}$ satisfies~\eqref{pfti-e68} with $i=6$ and $\epsilon_{62}:=\frac{\eta}{8} \min\{\epsilon_\nu, \epsilon_B\}- \sfrak d$.

 Combining the estimates for $A_{61}$ and $A_{62}$  in their different regions of validity  yields~\eqref{pfti-e68} with $i=6$  and $\epsilon_6:= \min\{\epsilon_{61},\epsilon_{62}\}$.

\medskip\noindent
\emph{Estimate of $B_1$}:
Denote by $\mathrm{I}$ the integral  $\int_\rd B_1\,dy$. We have
\begin{align*}
    \mathrm{I}
    = \int_\rd \left(\int_{t^{\zeta(y)}< |u|\leq t^{\zeta(x)}} + \int_{|u|\geq t^{\zeta(x)}}\right) p^0_t(x+u,y)\,\mu(x,du)\,dy
    =:\mathrm{I}_1+\mathrm{I}_2.
\end{align*}
(If $\zeta(x)\leq \zeta(y)$, the inner integral of $\mathrm{I}_1$ ranges over the empty set).

For $\mathrm{I}_2$ get with~\eqref{pfti-e30}, the assumptions on $\mu(x,du)$, cf.~\eqref{set-e08}, and the definition~\eqref{pfti-e12} of $\zeta(y)$
\begin{align*}
\mathrm{I}_2
    \leq C \int_{|u|\geq t^{\zeta(x)}} \mu(x,du)
    \leq C t^{-\zeta(x)\alpha(x)}
    \leq  C t^{-1+\sfrak \alpha_{\min}}.
\end{align*}

We will now estimate $\mathrm{I}_1$. Take $\delta<\zeta_{\min}$. We have
\begin{align*}
\mathrm{I}_1
    &= \int_\rd \int_{t^{\zeta(y)}< |u|\leq t^{\zeta(x)}}  p^0_t(x+u,y)\,\mu(x,du)\,dy\\
    &= \left(\int_{|\kappa_t(y)-x|\leq t^\delta}+ \int_{|\kappa_t(y)-x|> t^\delta}\right)\int_{t^{\zeta(y)}< |u|\leq t^{\zeta(x)}} p^0_t(x+u,y)\,\mu(x,du)\,dy\\
    &=:\mathrm{I}_{11}+\mathrm{I}_{12}.
\end{align*}
For $\mathrm{I}_{11}$  we  have by~\eqref{appA-e08}, \eqref{pfti-e30}  and~\eqref{set-e08}
\begin{align*}
\mathrm{I}_{11}
    &\leq \int_{|\kappa_t(y)-x|\leq t^\delta} \int_{ |u|> t^{\zeta(y)}}  p^0_t(x+u,y)\,\mu(x,du)\,dy\\
    &\leq \int_{ |u|>ct^{\zeta(x)}} \left(\int_\rd p^0_t(x+u,y)\,dy\right)\mu(x,du)
    \leq C t^{-1+\sfrak \alpha_{\min}}.
\end{align*}
Now we consider $\mathrm{I}_{12}$. Note that in $\mathrm{I}_{12}$ we only have to integrate over those $y$ such that $t^{\zeta(y)}<t^{\zeta(x)}$. In the inner  integral we have $|u|\leq t^{\zeta(x)}$, so
\begin{align}\label{pfti-e86}
    \frac{|\kappa_t(y)-x-u |}{t^{\zeta(y)}}
    \geq \frac{|\kappa_t(y)-x|- |u|}{t^{\zeta(y)}}
    \geq \frac{|\kappa_t(y)-x|-t^{\zeta(x)}}{t^{\zeta(x)}}
    = \frac{|\kappa_t(y)-x|}{t^{\zeta(x)}}-1.
\end{align}
Thus, using~\eqref{pfti-e52} with $k=0$ and the definition~\eqref{pfti-e42} of $K_t^{0;c}$  yields
\begin{align}\label{pfti-e88}\begin{aligned}
    &\int_{t^{\zeta(y)}<|u|<t^{\zeta(x)}} p^0_t(x+u,y) \, \mu(x,du)\\
    &\qquad\leq C t^{-\sfrak d}  \int_{t^{\zeta(y)}<|u|<t^{\zeta(x)}} f_{t,\zeta(y),c}(\kappa_t(y)-x-u)\,\mu(x,du)\\
    &\qquad\leq C t^{-\sfrak d-\zeta(y)d} e^{- {c|\kappa_t(y)-x|}{t^{-\zeta(x)}}} \mu(x,\{u:\, |u|\geq t^{\zeta(y)}\})\\
    &\qquad= C t^{-\sfrak d-\zeta(y)d-\alpha(x)\zeta(y)} e^{- {c|\kappa_t(y)-x|}{t^{-\zeta(x)}}}\\
    &\qquad\leq C t^{-Q}  e^{-c'{|y-\chi_t(x)|}{t^{-\zeta(x)}}}
\end{aligned}\end{align}
for some $Q>0$. For the last estimate we use Corollary~\ref{appA-21}. Integrating in $y$, we get
\begin{align*}
\mathrm{I}_{12}
    &\leq C\int_{|y-\chi_t(x)|> t^\delta}   t^{-Q}    e^{-c'{|y-\chi_t(x)|}{t^{-\zeta(x)}}}\,dy\\
    &\leq C t^{-Q} e^{-\frac12c't^{\delta-\zeta_{\min}}} \int_{|z|>1}   e^{-\frac12c' |z|}\,dz\\
    &\leq C t^{-Q} e^{-\frac12c't^{\delta-\zeta_{\min}}}
    \leq C t^{-1+\sfrak \alpha_{\min}},
\end{align*}
where we use that $\delta<\zeta_{\min}$, see the comment after~\eqref{pfti-e12}. Combining the above estimates proves~\eqref{pfti-e70} for $i=1$.

\medskip\noindent
\emph{Estimate of $B_2$}:
For $B_2$ we have by~\eqref{pfti-e30} and~\eqref{N1}
\begin{align*}
    \int_\rd \int_{|u|\geq t^{1/\alpha(x)}}  p^0_t(x+u,y)\,|\nu|(x,du)\,dy
    &\leq C \int_{|u|\geq t^{1/\alpha(x)}}\,|\nu|(x,du)\\
    &\leq C t^{-\beta(x)/\alpha(x)}
    \leq C t^{-1+\epsilon_\nu/\alpha_{\max}};
\end{align*}
this  proves~\eqref{pfti-e70} for $i=2$.
\end{proof}

We have already mentioned that Lemma~\ref{pfti-05} yields  the key estimate~\eqref{para-e18},  see Corollary~\ref{pfti-07}. Essentially the same argument can be used to establish additionally the tail behavior of the kernel $\Phi_t(x,y)$. Namely, we have the following.

\begin{lemma}\label{pfti-09}
We have
    \begin{align}\label{pfti-e90}
    \lim_{R\to \infty}\sup_{x\in \rd}\left( t^{1-\epsilon_\Phi} \int_{|y-x|\geq R} |\Phi_t(x,y)|\,dy\right)= 0.
    \end{align}
\end{lemma}
The proof, with minor changes, repeats that of  Lemma~\ref{pfti-05}; we omit the details.

\subsection{Further properties of $\Phi_t^\op$: Continuity and decay as $|x|\to \infty$}\label{phiop}

\begin{lemma}\label{pfti-11}
    For any $f\in B_b(\rd)$ and $t>0$  we have $\Phi_t^\op f \in C_b(\rd)$.
\end{lemma}
\begin{proof}
Using the definition of $b_t(x)$ we can rewrite~\eqref{pfti-e38} in the following way:
\begin{align}\label{pfti-e92}
\begin{split}
  \Phi_t(x,y) &= \left(b(x)-B_t(\kappa_t (y))\right)\cdot \nabla_x p_t^0(x,y)\\
    &\quad\mbox{}+ \int_\rd \left(p^0_t(x+u,y) -p^0_t (x,y)  - \nabla_x p^0_t(x,y) \cdot u \I_{\{|u|\leq 1\}}\right)N(x,du)\\
    &\quad\mbox{}- \int_{|u|\leq t^{\zeta(y)}} \left(p^0_t(x+u,y) -p^0_t (x,y)  - \nabla_x p^0_t(x,y) \cdot u \I_{\{|u|\leq t^{1/\alpha(y)}\}}\right)\mu(y,du).
    \end{split}
\end{align}
Decompose, accordingly,
\begin{align}\label{pfti-e94}
    \Phi^\op_t f (x):= \int_\rd \Phi_t(x,y)f(y)\, dy= \Phi^{\op,1}_t f (x)+\Phi^{\op,2}_t f (x)+\Phi^{\op,3}_t f (x).
\end{align}
Since the function $x\mapsto \nabla_x p_t^0(x,y)$  is continuous, the continuity of  $\Phi^{\op,1}_t f (x)$  follows from~\eqref{pfti-e52}, \eqref{C0} and the dominated convergence theorem.  Indeed, assume that  $x, x_0\in B(0,R)$, $x\to x_0$. Then we use \eqref{pfti-e52} and bound the right-hand side of~\eqref{pfti-e52} by $C(t,x)e^{-c t^{-\zeta_{\min}} |y|} $; now the continuity follows from the dominated convergence theorem. With the same argument, we see that  $\Phi^{\op,3}_t f \in C_b(\rd)$.

Consider now  $\Phi^{\op,2}_t f (x)$.  Denote by $h_f(t,x,u)$ the expression under the integral in $\Phi^{\op,2}_t f (x)$, i.e.
\begin{align*}
    h_f(t,x,u)= \int_\rd \left(p^0_t(x+u,y) -p^0_t (x,y)  - \nabla_x p^0_t(x,y) \cdot u \I_{\{|u|\leq 1\}}\right)f(y)\, dy.
\end{align*}
Note that for $x,x_0\in B(0,R)$, $R>0$,
\begin{align}\label{pfti-e96}
    |h_f(t,x,u)-h_f(t,x_0,u)|\leq C_{f,t,R} (|u|^2\wedge 1)|x-x_0|.
\end{align}
Indeed,  using~\eqref{pfti-e76}  and the second line in~\eqref{pfti-e80} with $u\rightsquigarrow su$ and $v\rightsquigarrow x_0-x+ su$, we get for $|u|\leq 1$
\begin{align*}
    |h_f(t,x,u)&-h_f(t,x_0,u)|\leq  C |u|^2   \int_0^1 \left|\nabla^2 P^{0,\op}_tf (x+su)-  \nabla^2 P^{0,\op}_tf (x_0+su)\right| ds\\
    &\leq C |u|^2  |x-x_0| \int_0^1 \int_0^1  \left|\nabla^3 P^{0,\op}_tf \big(x+sru+ (1-r)  (x-x_0+su)\big) \right|dr\, ds\\
    &\leq C_{f,t,R} |u|^2  |x-x_0|,
\end{align*}
where $P^{0,\op}_tf(x)= \int_\rd p^0_t(x,y) f(y)dy$. In the last line we used that the function under the integral is continuous in $x$ for  any $t>0$; this follows from~\eqref{pfti-e52}.  A similar estimate holds for $|u|\geq 1$, but with $1$ instead of $|u|^2$ on the right-hand side. Therefore,
\begin{align}\label{pfti-e98}
|h_f(t,x,u)-h_f(t,x_0,u)|\leq
 C_{f,t,R} (|u|^2 \wedge 1) |x-x_0|,
\end{align}

Rewrite $\Phi^{\op,2}_t f (x)$ as
\begin{align}\label{pfti-e100}
    \Phi^{\op,2}_t f (x) = \int_\rd \left( h_f(t,x,u) - h_f(t,x_0,u)\right)N(x,du)+\int_\rd  h_f(t,x_0,u)N(x,du).
\end{align}
Using the estimate~\eqref{pfti-e98}   we get for the first term
\begin{align*}
    \left| \int_\rd \left( h_f(t,x,u) - h_f(t,x_0,u)\right)N(x,du) \right|
    \leq   C_{f,t} |x-x_0| \sup_x \int_\rd (|u|^2\wedge 1) N(x,du).
\end{align*}
This proves the continuity in $x$ of the first term in~\eqref{pfti-e100}.

In order to handle the second term in~\eqref{pfti-e100} observe that by~\eqref{C1} the family probability measures $(\tilde{N}(x,du))_{x\in \rd}$,
\begin{align*}
    \tilde{N} (x,du) =  \frac{(|u|^2\wedge 1) N(x,du)}{\int_\rd (|u|^2\wedge 1) N(x,du)},
\end{align*}
converges weakly as $x\to x_0$. Note that the function $h_f(t,x_0,u) (|u|^2\wedge 1)^{-1}$ is continuous and bounded in $u$ on $\rd\setminus(\{0\}\cup\{ u: \, |u|=1\})$. The condition~\eqref{C2} guarantees that the discontinuity set of $h_f$ has $\tilde{N}(x,\cdot)$-measure $0$ for every $x\in \rd$. Therefore, the continuous mapping theorem of weak convergence (cf.~\cite[Theorem 25.7]{Bil95}) applies, and we get
\begin{align*}
    \lim_{x\to x_0} \int_\rd  h_f(t,x_0,u)N(x,du)
    &=\int_\rd  h_f(t,x_0,u)N(x_0,du).
\qedhere
\end{align*}
\end{proof}

Essentially the same argument  as in Lemma~\ref{pfti-11} gives the continuity in $t$.
\begin{lemma}\label{pfti-13}
    For any $f\in B_b(\rd)$ and $t>0$ we have $\lim_{s\to t}\|\Phi_s^\op f -\Phi_t^\op f\|_\infty=0$.
\end{lemma}
\begin{proof}[Sketch of the proof]
    The proof mainly repeats the previous one, hence we just outline it. By condition~\eqref{B0} and the dominated convergence theorem, the function $B_t(x)$ defined by~\eqref{pfti-e22} is continuous in $(t,x)\in (0,\infty)\times \rd$. Next, $p_t^0(x,y)$ and its derivatives in $x$ are continuous in $(t,x)\in (0,\infty)\times\rd$; this follows from the definition~\eqref{pfti-e26} of  $p_t^0(x,y)$ and the properties of the function $p_t^{z,\cut}$, see Proposition~\ref{appC-07}. With very few changes in the arguments of the previous proof, we can show that
    \begin{align}\label{pfti-e102}
        \|\Phi_s^{\op,i} f -\Phi_t^{\op,i} f\|_\infty\to 0, \quad s\to t,\; i=1,2.
    \end{align}

    Now we consider the third term. Without loss of generality let $s<t$. We have
    \begin{align*}
        \Phi_s^{\op,3} f(x) &-\Phi_t^{\op,3} f(x)\\
        &= \int_\rd \left(\int_{|u|\leq s^{\zeta(y)} } \big(h(t,u,x,y)- h(s,u,x,y)\big) \mu(y,du)\right) f(y)\,dy\\
        &\quad \mbox{}+\int_\rd \left(\int_{s^{\zeta(y)}< |u|\leq t^{\zeta(y)}} h(t,u,x,y) \mu(y,du)\right) f(y)\,dy\\
        &=: I_1(t,s,x) +  I_2(t,s,x);
    \end{align*}
    the integrand $h$ is defined similar to~\eqref{pfti-e76}, i.e.
    \begin{align*}
        h(t,u,x,y):= p^0_t(x+u,y) -p^0_t (x,y)  - \nabla_x p^0_t(x,y) \cdot u \I_{\{|u|\leq t^{1/\alpha(y)}\}}.
    \end{align*}
    Thus, $\|I_1(t,s,\cdot)\|_\infty\to 0$, $s\to t$, by the above mentioned continuity properties of $t\mapsto p_t^{z,\cut}$; moreover, $\|I_2(t,s,\cdot)\|_\infty\to 0$, $s\to t$, since the radial part $\lambda(y)r^{-1-\alpha(y)}\, dr$ of  $\mu(y,du)$ is uniformly continuous in $y$.
\end{proof}

In order to establish the decay of $\Phi^\op_t f(x)$ as $|x|\to \infty$, we require further properties of $f$.
\begin{lemma}\label{pfti-15}
    If $f\in B_b(\rd)$ and $f(x)\to 0$ as $|x|\to \infty$, then $\Phi_t^\op f \in C_\infty(\rd)$.
\end{lemma}
The proof follows from Lemma~\ref{pfti-09}, Lemma~\ref{pfti-11}, and an argument similar to that in the proof of~\eqref{appC-e42}  where Lemma~\ref{pfti-09} takes over the role of Corollary~\ref{appC-15}.

The above results will eventually allow us to establish the (strong) Feller continuity of the family $(P_t)_{t\geq 0}$.  We can summarise these results as follows.
\begin{corollary}\label{pfti-17}
Let $t>0$. The following properties hold
\begin{enumerate}
    \item\label{pfti-17-a} $\Phi^\op_t (B_b)\subset C_b$;
    \item\label{pfti-17-b} $\Phi^\op_t (C_\infty)\subset C_\infty$;
    \item\label{pfti-17-c} $\lim_{s\to t}\|\Phi^\op_sf-\Phi^\op_tf\|_\infty=0$ for all $f\in B_b$.
\end{enumerate}
\end{corollary}

Corollary~\ref{pfti-17} and  the representation~\eqref{para-e26} of  (the kernel of)  $\Psi_t^\op$   imply  similar properties for the family $(\Psi_t)_{t\geq 0}$.
\begin{corollary}\label{pfti-19}
Let $t>0$. The following properties hold
\begin{enumerate}
    \item\label{pfti-19-a} $\Psi^\op_t (B_b)\subset C_b$;
    \item\label{pfti-19-b} $\Psi^\op_t (C_\infty)\subset C_\infty$;
    \item\label{pfti-19-c} $\lim_{s\to t}\|\Psi^\op_sf-\Psi^\op_tf\|_\infty=0$ for all $f\in B_b$.
\end{enumerate}
\end{corollary}

\section{Proof of Theorem~\ref{t1}  -- Mapping properties}\label{pti}

Let $(P_t)_{t>0}$ be the family of linear operators on $B_b(\rd)$ defined by~\eqref{para-e24}; for $t=0$ we set $P_0=\mathrm{id}$, the identity operator. For the proof of Theorem~\ref{t1} we have to show that the restriction of  $(P_t)_{t\geq 0}$ to $C_\infty(\rd)$ is a Feller semigroup -- i.e.\ a strongly continuous, positivity preserving and contractive semigroup -- whose generator is an extension of $L$; we also have to verify the uniqueness stated in Theorem~\ref{t1}. For that,  we  adapt the strategy developed in~\cite{KK18,Ku}: In particular, we show that $p_t(x,y)$, which was constructed by means of the parametrix approach in Section~\ref{para} as a candidate for the fundamental solution, is an approximate fundamental solution (in the sense of Section~\ref{feller}).  For the readers' convenience, and in order to have a  self-contained presentation, we give full proofs below.

\subsection{(Strong) Feller continuity. Approximate fundamental solutions}\label{feller}

Consider the family of the operators $(P_t^0)_{t\geq 0}$, where $P_t^0=(p^0_t)^\op$, $t>0$, and $P_0^0=\mathrm{id}$. The following lemma
shows for this family the properties similar to those listed in Corollary~\ref{pfti-17} for $\Phi_t^{\op}$, $t>0$.

\begin{lemma}\label{pti-03}
    For any $t\geq 0$ and $f\in C_\infty$ one has $P_t^0 f\subset C_\infty$ and $\|P^0_sf-P_t^0f\|_\infty\to 0$ as $s\to t$.
\end{lemma}

\begin{proof}
The following properties are shown in Proposition~\ref{appC-19}, compare also~\eqref{pfti-e28}: For every $f\in C_\infty(\rd)$
\begin{align}\label{pti-e02}
\lim_{|x|\to\infty} \left| \int_\rd p_t^0(x,y)f(y)dy\right| =0,
\\\label{pti-e04}
\lim_{t\to 0}\sup_{x\in \rd} \Big|\int_\rd p_t^0(x,y)f(y)dy -f(x)\Big|\to 0.
\end{align}
Observe that $P_t^0 (C_\infty)\subset C_\infty$ follows from~\eqref{pti-e02} every fixed $t>0$. Since the function  $p_t^0(x,y)$ is smooth in $t$, and since the derivative has for any $t\geq t_0>0$ an integrable upper bound (cf. Proposition~\ref{appC-07}), we obtain strong continuity: $\lim_{s\to t}\|P^0_sf-P_t^0f\|_\infty = 0$ for any $f\in C_\infty$ and $t> t_0$.
Finally, \eqref{pti-e04} shows that this family is strongly continuous at $t=0$.
\end{proof}

Combining the above statements, we get the following.

\begin{lemma}\label{pti-05}
\begin{enumerate}
\item\label{pti-05-a}
    \textup{(}strong Feller and Feller continuity\textup{)}  One has $P_t (B_b)\subset C_b$ and $P_t (C_\infty)\subset C_\infty$ for all $t>0$.
\item\label{pti-05-b}
    \textup{(}strong continuity\textup{)} For any $t\geq 0$, $f\in C_\infty$ one has
    $\lim_{s\to t}\|P_sf-P_tf\|_\infty = 0$.
\end{enumerate}
\end{lemma}

\begin{proof}
The required statements follow directly from the corresponding properties of $P^0_t$, $\Phi^\op_t$, see Corollary~\ref{pfti-17} and Lemma~\ref{pti-03}, and the formula~\eqref{para-e24}, which contains the series representation of $P_t$,  combined with the bounds~\eqref{para-e28}, which ensures the strong convergence of the series.
\end{proof}

We have
\begin{align}\label{pti-e06}
    p_t(x,y)=p_t^0(x,y)+ \int_0^t\int_{\rd} p^0_{t-s}(x,z)\Psi_s(z,y) \,dz\,ds,
\end{align}
see~\eqref{para-e06} and Section~\ref{faf}. One might expect that~\eqref{para-e04} can be verified by applying $\left(\frac{d}{dt} - L_x\right)$ to~\eqref{pti-e06}; however, such an application causes considerable difficulties. The function $p_t^0(x,y)$ is of class $C^1$ in $t$ and of class $C^2$ in $x$, however a major issue is the strong singularity of $\partial_t p_t^0(x,y), L_xp_t^0(x,y)$ as $t\to 0$, which causes difficulties when we want to interchange $\left(\frac{d}{dt} - L_x\right)$ with the integral w.r.t.\ $s$ in the right hand side of~\eqref{pti-e06}. The following approximations of $P_t$ and $p_t(x,y)$ take care of this problem. For $\epsilon>0$ we define
\begin{align*}
\begin{aligned}
    P_{t,\epsilon}
    &= P_{t+\epsilon}^0 + \int_0^tP_{t-s+\epsilon}^0\Psi_s^\op\,ds,
    \\
 p_{t,\epsilon}(x,y)&=p_{t+\epsilon}^0(x,y)+ \int_0^t\int_{\rd} p^0_{t-s+\epsilon}(x,z)\Psi_s(z,y) \,dz\,ds.
\end{aligned}
\end{align*}
The following lemma shows that $p_{t,\epsilon}(x,y)$ approximates $p_t(x,y)$ and solves, approximatively, the equation $(\partial_t - L_x)p_t(x,y)=0$. Therefore, we call $\{p_{t,\epsilon}(x,y), \epsilon >0\}$ an \emph{approximate fundamental solution}.

\begin{lemma}\label{pti-07}
Let $f\in C_\infty(\rd)$.
\begin{enumerate}
\item\label{pti-07-a}
    $\lim_{\epsilon\to 0} \|P_{t,\epsilon}f- P_t f\|_\infty = 0$ and $\lim_{|x|\to\infty} P_{t,\epsilon}f(x) = 0$ exist uniformly in $t\in (0,1]$, resp., uniformly in $(t,\epsilon)\in (0,1]\times (0,1]$.

\item\label{pti-07-b}
    $\lim_{t,\epsilon\to 0+} \|P_{t,\epsilon} f-f\|_\infty =0$.

\item\label{pti-07-c}
    For every  $\epsilon >0$,  $P_{t,\epsilon}f(x)$ belongs to $C^1(0,\infty)$ as a function of $t$, and to $C^2_\infty(\rd)$ as a function of $x$; moreover, $\partial_tP_{t,\epsilon}f(x), L_x P_{t,\epsilon}f(x)$ are continuous as functions of $(t,x)$.

\item\label{pti-07-d}
    For every $\tau\in (0,1)$ the following limit exists uniformly for all $t\in [\tau,1]$ and $x\in \rd$
    \begin{align*}
    \Delta_{t,\epsilon} f(x) := \left(\partial_t-L_x\right) P_{t,\epsilon} f (x)\to 0 \quad \text{as $\epsilon\to 0$}.
    \end{align*}
\item\label{pti-07-e}
    We have
    \begin{align*} 
        \lim_{\epsilon\to 0} \int_0^1\sup_{x\in \rd}|\Delta_{t,\epsilon} f(x)|\, dt = 0.
    \end{align*}
\end{enumerate}
\end{lemma}
\begin{proof} Statements~\ref{pti-07-a}, \ref{pti-07-b} follow from Corollary~\ref{pfti-17}, Lemma~\ref{pti-03}, and the formula~\eqref{para-e24} combined with the bounds~\eqref{para-e28}. The first statement in~\ref{pti-07-c} follows from the definition of the functions  $p^0_t(x,y)$ and, respectively,  $P_{t,\epsilon}f(x)$,  because  in the latter we removed the singularity in $t$ by adding $\epsilon>0$. The second part of the statement follows by the same argument combined with assumption \CC.

Let us proceed to the proof of~\ref{pti-07-d}. We have
\begin{align}\label{pti-e12}\begin{aligned}
    LP_{t,\epsilon }f(x)
    &= L_x \int_\rd  p_{t+\epsilon}^0(x,y)f(y)\,dy \\
    &\qquad\mbox{}    + L_x\int_0^t \int_\rd\int_\rd p_{t-s+\epsilon}^0(x,z)\Psi_s(z,y)f(y)\,dy\,dz\,ds;
\end{aligned}\end{align}
note that by~\ref{pti-07-c}  $P_{t,\epsilon}f$  and both integrals in the right hand side are $C_\infty^2$-functions in $x$.  We would like to interchange $L_x$ and  the integrals in~\eqref{pti-e12}, i.e.\ write
\begin{align}\label{pti-e14}\begin{aligned}
    LP_{t,\epsilon}f(x)
    &= \int_\rd  L_x p_{t+\epsilon}^0(x,y)f(y)\,dy \\
    &\qquad\mbox{} + \int_0^t \int_\rd\int_\rd L_x p_{t-s+\epsilon}^0(x,z)\Psi_s(z,y)f(y)\,dy\,dz\,ds.
\end{aligned}\end{align}
Because of~\eqref{pfti-e52}  and the dominated convergence theorem we can bring the gradient part of $L$ inside the integrals. Denote, for a moment, the integral part of the operator $L$ by $L^{\textrm{int}}$. Observe that  we have for $f\in C^2_\infty  (\rd)$
\begin{align*}
    L^{\textrm{int}}f(x)
    = \lim_{\delta \to 0+}\int_{|u|>\delta}\Big(f(x+u)-f(x)- \nabla f(x) \cdot u \I_{\{|u|\leq 1\}} \Big) \,N(x,du).
\end{align*}
Since adding $\epsilon$ removes the singularity in time, \eqref{pfti-e52} implies
\begin{align*}
  \sup_{t\in (0,1],x,y\in \rd} \big| p^0_{t+\epsilon}(x+u,y)& -p^0_{t+\epsilon} (x,y)  - \nabla_x p^0_{t+\epsilon}(x,y) \cdot u \I_{\{|u|\leq t^{1/\alpha(y)}\}}\big|
  \leq C(\epsilon)(|u|^2\wedge 1),
\end{align*}
and Fubini's theorem justifies the interchange of $L^{\textrm{int}}$ with the integrals in~\eqref{pti-e14}.

Similarly, using the differentiability of $p_{t}^0(x,y)$ in $t$ and the upper estimate of the derivatives (cf.~\eqref{pfti-e52}), we get
\begin{align}\label{pti-e16}
\begin{split}
    \partial_t P_{t, \epsilon}f(x)
    &=\int_\rd \partial_tp_{t+\epsilon}^0(x,y)f(y)\,dy
        + \int_0^t \int_\rd \partial_t p_{t-s+\epsilon}^0(x,z)\Psi_s^\op f(z)\,dz\,ds\\
    &\qquad\mbox{} + \int_\rd p_{\epsilon}^0(x,z)\Psi_t^\op f(z)\,dz.
\end{split}
\end{align}
Using~\eqref{pti-e14} and~\eqref{pti-e16}  we get
\begin{align}\label{pti-e18}
\begin{split}
    \Delta_{t,\epsilon}f(x)
    &=\int_\rd  p_{\epsilon}^0(x,z)\Psi_t^\op  f (z)\,dz
        - \Phi_{t+\epsilon}^\op f (x)\\
    &\qquad\mbox{} -\int_0^t \int_\rd  \Phi_{t-s+\epsilon}(x,z) \Psi_s^\op f(z)\,dz\,ds.
\end{split}
\end{align}
Since the  function $\Psi^\op f$ satisfies the equation
\begin{align}\label{pti-e20}
    \Psi_t^\op f(x)
    = \Phi_t^\op f (x) + \int_0^t \int_\rd \Phi_{t-s}(x,z) \Psi_s^\op f(z)\,dz\,ds,
\end{align}
we can rewrite $\Delta_{t,\epsilon}f(x)$ as follows:
\begin{align*}
    \Delta_{t,\epsilon}f(x)
    &= \left(\int_\rd p_{\epsilon}^0(x,z)\Psi_t^\op f(z)\,dz
        -\Psi_{t+\epsilon}^\op f (x)\right)
        + \int_t^{t+\epsilon} \int_\rd \Phi_{t-s+\epsilon} (x,z)\Psi_s^\op f(z)\,dz\,ds\\
    &=:\Delta_{t,\epsilon}^1f(x) + \Delta_{t,\epsilon}^2f(x).
\end{align*}
By the strong  continuity of the operator family $\Psi^\op_t, t>0$ (see Corollary~\ref{pfti-19}.\ref{pfti-19-c})  we have
\begin{align*}
    \sup_{t\in [\tau, 1], x\in \rd}\left|\Psi_{t+\epsilon}^\op f (x)-\Psi_{t}^\op f (x)\right|\to 0, \quad \epsilon\to 0.
\end{align*}
By~\eqref{pti-e04}, we know that $P^0_\epsilon$ strongly converges to the identity operator as $\epsilon\to 0$. Because of the strong continuity of the operator family $\Psi^\op_t$, $t>0$, see Corollary~\ref{pfti-19}.\ref{pfti-19-c},
\begin{align*}
    \sup_{t\in [\tau, 1], x\in \rd} \left|\int_{\rd}p_\epsilon^0(x,z)\Psi_{t}^\op f (z)dz-\Psi_{t}^\op f(x)\right|\to 0, \quad \epsilon\to 0.
\end{align*}
This proves that  $\Delta_{t,\epsilon}^{1}f(x) \to 0$ as $\epsilon\to 0$.  By the  strong  continuity of the operator family $\Phi^\op_t, t>0$ (see Corollary~\ref{pfti-17}.\ref{pfti-17-c}) and the estimate~\eqref{para-e30} we get $\Delta_{t,\epsilon}^{2}f(x)\to 0$, and~\ref{pti-07-d} follows.

Statement~\ref{pti-07-e} now  follows from the statement~\ref{pti-07-d} and the estimate~\eqref{para-e30}.
\end{proof}

Following~\cite[Def.\ 5.1]{Ku} we call a function $h(t,x)$ an \emph{approximate harmonic} function for the operator $\partial_t-L$, if there exists a family $\{h_\epsilon(t,x)\}_{\epsilon\in (0,1]}\subset C([0,\infty)\times \rd)$ such that
\begin{enumerate}
\item
    for any $R>0$ and $T>0$
      \begin{align*}
         \lim_{\epsilon\to 0}\sup_{|x|\leq R,\, t\in [0,T]}|h_\epsilon(t,x)-h(t,x)|= 0,\\
         \lim_{|x|\to\infty} \sup_{t\in [0,T], \epsilon\in (0,1]}|h_\epsilon(t,x)|=0;
      \end{align*}
\item
    each function $h_\epsilon(t,x)$ is of class $C^1(0,\infty)$ w.r.t.\ $t$ and $C^2_\infty(\rd)$ w.r.t.\ $x$, and for  any $R>0$, $T>0$, and $\tau\in (0,T)$
    \begin{align*}
        \lim_{\epsilon\to 0} \sup_{|x|\leq R,\,t\in [\tau, T]}|(\partial_t-L_x)h_\epsilon(t,x)|=0.
            \end{align*}
\end{enumerate}
Lemma~\ref{pti-07} actually shows that the function $h(t,x)=P_tf(x)$ is for any $f\in C_\infty(\rd)$ approximately harmonic for $\partial_t-L$.  The corresponding approximating family is given by
\begin{align*}
    h_\epsilon(t,x) = P_{t,\epsilon}f(x), \quad \epsilon >0.
\end{align*}
 Let us point out, that we prove in Lemma~\ref{pti-07} the required properties for $T=1$; this is for notational convenience, in the case of a general fixed $0<T<\infty$, only the constants will be affected, see also the comment on p.~\pageref{blind-ref}.

\subsection{The Positive Maximum Principle: Positivity and semigroup properties}\label{pmp}
An operator $L$ is said to satisfy the \emph{positive maximum principle} (PMP), if
\begin{align*}
    \text{for any $f\in D(L)$ such that $f(x_0)=\sup_{x\in\rd} f(x)\geq 0$ we have $Lf(x_0)\leq 0$}.
\end{align*}
The PMP is, essentially, a structural property of the operator $L$, see~\cite[Section~2.3]{BSW};
it is not difficult to see that operators of the form~\eqref{set-e02} satisfy the PMP on $D(L)=C^2_\infty(\rd)$.  It is well known that the PMP property of $L$ yields non-negativity of a function which is harmonic for $\partial_t-L$ and is non-negative for $t=0$. The following statement extends the range of applications of this principle to approximately harmonic functions, cf.~\cite[Propsition~5.5]{Ku}.

\begin{lemma}\label{pti-09}
    Let $h(t,x)$ be an approximate harmonic function for $\partial_t-L$ where $L$ satisfies the PMP.
    If $h(0, x)\geq 0$ for all $x\in \rd$, then $h(t,x)\geq 0$ for all $t> 0$ and  $x\in \rd$.
\end{lemma}

With the help of Lemma~\ref{pti-09} we can derive, in a standard way,  the positivity and semigroup properties of the family $(P_t)_{t\geq 0}$, cf.~\cite[Corollary~5.1]{Ku}.
\begin{corollary}\label{pti-11}
\begin{enumerate}
\item\label{pti-11-a}
    Each operator $P_t$, $t\geq 0$ is positivity preserving: $P_t f \geq 0$ for all $f\geq 0$.

\item\label{pti-11-b}
    The family $(P_t)_{t\geq 0}$ is an operator semigroup: $P_{t+s}f=P_tP_sf$ for all $f\in C_\infty(\rd)$, $s,t\geq 0$.

\item\label{pti-11-c}
    For any $f\in C_\infty^2(\rd)$,
    \begin{align*}
        P_tf(x)-f(x)=\int_0^tP_sLf(x)\, ds,  \quad t\geq 0.
    \end{align*}
    In particular, $\int_\rd p_t(x,y)\, dy=1$ for all $t>0$ and $x\in \rd$.
\end{enumerate}
\end{corollary}

Lemmas~\ref{pti-05} and~\ref{pti-07} together with Corollary~\ref{pti-11}.\ref{pti-11-a}, \ref{pti-11-b} imply that $(P_t)_{t\geq 0}$ is  a \emph{strongly continuous and positivity preserving semigroup} in $C_\infty(\rd)$.
Conservativeness follows from Corollary~\ref{pti-11}.\ref{pti-11-c} and the boundedness of the coefficients in $L$,  see~\cite{Sch98}, also~\cite{KK18} which, in turn, implies the contractivity.  This means that $(P_t)_{t\geq 0}$ is a Feller and a strong Feller semigroup. Note that Corollary~\ref{pti-11}.\ref{pti-11-c}  identifies $L$ as the semigroup's generator -- at least on the set $C_\infty^2(\rd)\subset D(L)$.

Using Kolmogorov's standard construction for stochastic processes, for every probability measure $\pi$ on $\rd$ there exists a Markov process $(X_t)_{t\geq 0}$ with the transition semigroup $(P_t)_{t\geq 0}$ and transition function $p_t(x,y)$, c\`adl\`ag trajectories, and initial distribution $X_0\sim\pi$, see e.g.~\cite[Ch.~4, Th.~2.7]{EK86}. By Lemma~\ref{pti-05}, the process $(X_t)_{t\geq 0}$ is also strong Feller.

\subsection{The martingale problem: Uniqueness}\label{mart}
Let $Y$ be any Feller process and denote its generator by $(A,D(A))$. If $C_\infty^2(\rd)\subset D(A)$ and $A|_{C_\infty^2(\rd)} = L$, then $Y$ is a c\`adl\`ag solution to the martingale problem for $(L, C_\infty^2(\rd))$; this follows from the strong Markov and semigroup properties of a Feller process, see~\cite[Corollary~4.1.7]{EK86}.  In particular, the Markov process $X$ which we have constructed in the previous section, is a solution to the  $(L, C_\infty^2(\rd))$-martingale  problem. In this section, we sketch the proof that the c\`adl\`ag-solution to the $(L, C_\infty^2(\rd))$-martingale problem with a given initial distribution $\pi$ is unique; this will complete the proof of Theorem~\ref{t1}. The argument here is almost the same as in~\cite[Section~5.3]{Ku18}, so we provide only the key points and omit details.

By Corollary 4.4.3 in~\cite{EK86}, uniqueness holds if for any two c\`adl\`ag-solutions to $(L, C_\infty^2(\rd))$ with the same initial distribution $\pi$ the corresponding one-dimensional distributions coincide. In what follows, we fix \emph{some} solution $Y$ and prove that
\begin{align}\label{pti-e22}
    \Ee^\pi f(Y_T)
    = \int_{\rd}P_Tf(x)\,\pi(dx), \quad f\in C_\infty(\rd),\; T>0
\end{align}
($\Ee^\pi$ indicates that the initial distribution is $\pi$).

Since $Y$ has c\`adl\`ag paths, it is stochastically continuous.  By~\cite[Lemma~4.3.4]{EK86}, the process
\begin{align*}
    h(t,Y_t) - \int_0^t\left(\partial_s h(s, Y_s) + L_xh(s, Y_s)\right) ds,\quad t\geq 0,
\end{align*}
is a martingale for any function $h(t, x)$ such that
\begin{align*}
    h(\cdot,x)\in C^1(0,\infty),\quad
    h(t,\cdot)\in C_\infty^{ 2}(\rd),\quad
    \partial_t h(\cdot,\cdot),\; L_x h(\cdot,\cdot)\in C_b((0,\infty)\times\rd).
\end{align*}
If we apply this to the function
\begin{align*}
    h(t,x) := h_\epsilon^{T,f}(t,x) := P_{T-t, \epsilon}f(x), \quad t\in [0,T],\; x\in \rd
\end{align*}
with fixed $f\in C_\infty(\rd)$, $T>0$, and arbitrary $\epsilon >0$, we get
\begin{align*}
    \Ee^\pi h^{T,f}_\epsilon(T,Y_T) - \Ee^\pi h^{T,f}_\epsilon(0,Y_0)
    &= \Ee^\pi\int_0^T(\partial_s+L_x) h^{T,f}_\epsilon(s,Y_s)\, ds\\
    &= -\Ee^\pi \int_0^T\Delta_{T-s, \epsilon}f(Y_s)\, ds,
\end{align*}
using the notation of Lemma~\ref{pti-07}. We have
\begin{align*}
    \lim_{\epsilon\to 0}h^{T,f}_\epsilon(T,x) = f(x)
    \et
    \lim_{\epsilon\to 0} h^{T,f}_\epsilon(0,x) = P_Tf(x)
\end{align*}
uniformly in $x$. With the help of Lemma~\ref{pti-07}.\ref{pti-07-d} we can let $\epsilon\to 0$ and get
\begin{align*}
    \Ee^\pi f(Y_T)
    = \Ee^\pi P_Tf(Y_0)
    = \int_{\rd}P_Tf(x)\pi(dx),
\end{align*}
which proves~\eqref{pti-e22}, finishing the proof of Theorem~\ref{t1}.

\section{Proof of Theorems~\ref{t2} and~\ref{t3}}\label{pftii}

We have constructed the transition density $p_t(x,y)$ in the form
\begin{align*}
    p_t(x,y) = p_t^0(x,y) + \left(p^0 \cstar \Psi\right)_t(x,y),
\end{align*}
see Sections~\ref{ans} and~\ref{faf}; the operation `$\cstar$' is defined in~\eqref{para-e08}.  In this section, we show that $p_t(x,y)$ can be written in the following form
\begin{align*}
    p_t(x,y)
    = \frac{1}{t^{d/\alpha(x)}} g^x\left(\frac{y- \chi_t(x)}{t^{1/\alpha(x)}}\right) + R_t(x,y)
\end{align*}
as it is claimed  in Theorem~\ref{t2} and Theorem~\ref{t3}. Here $g^x(z)$ is the density of an $\alpha(x)$-stable random variable with drift $\upsilon=\upsilon(x)$ and characteristic exponent $\psi^{x,\upsilon}(\xi)$, see~\eqref{set-e20} and~\eqref{set-e22}.  We have to show that the $L^\infty(dx)\otimes L^1(dy)$-norm and the $L^\infty(dx)\otimes L^\infty(dy)$-norm of the remainder term $R_t(x,y)$ are  bounded by $Ct^{\epsilon_R}$ and $Ct^{\epsilon_R-d/\alpha_{\min}}$, respectively, as claimed in Theorem~\ref{t2} and Theorem~\ref{t3}.

\subsection{Proof of Theorem~\ref{t2}}\label{s81}

\begin{proof}
Our proof is based on   Proposition~\ref{appC-17}. We write $p_t(x,y)$ in the form
\begin{align}\label{pftii-e02}
    \frac{1}{t^{d/\alpha(x)}} g^x \left(\frac{y- \chi_t(x)}{t^{1/\alpha(x)}}\right) + \left(p^0_t(x,y)- \frac{1}{t^{d/\alpha(x)}} g^x \left(\frac{y- \chi_t(x)}{t^{1/\alpha(x)}}\right)\right) +  (p^0 \cstar \Psi)_t(x,y),
\end{align}
 which shows that the remainder $R_t(x,y)$ consists of the last two terms.   The estimate for the middle term is already contained in~\eqref{appC-e32}.  From~\eqref{para-e30}  we know that
\begin{align*}
    \sup_{x\in\rd}  \int_\rd |\Psi_t(x,y)| \, dy
    \leq C t^{-1+\epsilon_\Phi}.
\end{align*}
This estimate, together with Corollary~\ref{appC-15}, finally gives
\begin{gather*}
    \sup_{x\in\rd} \int_\rd |(p^0 \cstar \Psi)_t(x,y)|\,dy
    \leq C  \sup_{x\in\rd}  \int_0^t \left(\int_\rd p_{t-s}^0(x,z)\,dz\right)  s^{-1+\epsilon_\Phi}\,ds
    \leq C t^{\epsilon_\Phi}.
\qedhere
\end{gather*}
\end{proof}

\subsection{Proof of Theorem~\ref{t3}}\label{s82}

Compared with Theorem~\ref{t2}, Theorem~\ref{t3} requires the additional conditions~\eqref{set-e32}--\eqref{set-e36}  which we will assume from now on.

The proof of Theorem~\ref{t3}  relies on the decomposition~\eqref{pftii-e02} and  estimates of the second and the third term which are uniform
in $x,y$. We know from Section~\ref{faf} that the estimate of  the third term follows from the operator bounds~\eqref{para-e40}, \eqref{para-e42}, and~\eqref{para-e46}.

We begin with~\eqref{para-e46}. From~\eqref{pfti-e26}, \eqref{pfti-e52} and the definition~\eqref{pfti-e42} of the kernel $K_t^{0;c}(x,y)$, we have that
\begin{align*}
    p_t^0(x,y)\leq Ct^{-d\zeta(y)-\mathfrak{s}d}=Ct^{-d/\alpha(y)};
\end{align*}
the last equality holds as $\zeta(y)= 1/\alpha(y)-\mathfrak{s}$, see~\eqref{pfti-e12}. Therefore, we get
\begin{align}\label{pftii-e04}
    \sup_{x,y\in\rd} p^0_t(x,y) \leq C t^{-d/\alpha_{\min}},
\end{align}
and~\eqref{para-e46} follows.

The  bounds~\eqref{para-e40}, \eqref{para-e42} are contained in the following lemma.
\begin{lemma}\label{pftii-03}
    There exist constants $C>0$ and $\epsilon_R\in (0,\epsilon_\Phi]$  such that for all $t\in (0,1]$
    \begin{align}\label{pftii-e06}
        \sup_{x,y\in\rd} |\Phi_t(x,y)|\leq C t^{-d/\alpha_{\min} -1+\epsilon_R},
    \end{align}
     \begin{align}\label{pftii-e08}
        \sup_{y\in\rd} \int_\rd |\Phi_t(x,y)|\,dx   \leq C t^{ -1+\epsilon_R}.
    \end{align}
\end{lemma}
\begin{proof}  Recall that  the kernels $K^{0;c}_t(x,y)$ and $K^{1;c}_t(x,y)$ were introduced in~\eqref{pfti-e42} and~\eqref{pfti-e44}, respectively.
In Section~\ref{dec}, cf.\ Lemma~\ref{pfti-05}, we have seen that
\begin{align}\label{pftii-e10}
    |\Phi_t(x,y)|
    \leq C t^{-1+\epsilon_\Phi} \left(p^0_t(x,y)+ K^{1;c}_t(x,y)\right)+  \left| B_1(t,x,y)+ B_2(t,x,y)\right|
\end{align}
where
\begin{align*}
     B_1 (t,x,y) = \int_{|u|>t^{\zeta(y)}}p^0_t(x+u,y)\,\mu(x,du),
\end{align*}
and
\begin{align*}
    B_2 (t,x,y) = \int_{|u|> t^{1/\alpha(x)}}  p^0_t(x +u,y)\,\nu(x,du)
\end{align*}
are from the decomposition~\eqref{pfti-e66} of the kernel $\Phi_t(x,y)$.  We use~\eqref{pftii-e10} in order to get~\eqref{pftii-e06} and~\eqref{pftii-e08}.

\medskip\noindent
\emph{Verification of~\eqref{pftii-e06}.}
We estimate the terms appearing on the right-hand side of~\eqref{pftii-e10} separately.
Further, using the definition~\eqref{pfti-e44} of $K_t^{1;c}(x,y)$ and Corollary~\ref{appA-21} -- this allows us to switch from $y-\chi_t(x)$ to $\kappa_t(y)-x$ -- together with Remark~\ref{appA-23} we get
\begin{align}\label{pftii-e12}
    K_t^{1;c}(x,y)
    &\leq C  \I_{\{|\kappa_t(y)-x|\leq c t^\delta\}} f_{t, \zeta(y),c'}(\kappa_t(y)-x) \\
    &\notag\qquad\qquad\mbox{}        +  C t^{-N}  \I_{\{|\kappa_t(y)-x| >c t^\delta\}} f_{t, \zeta_{\min},c'}(\kappa_t(y)-x),
\end{align}
where $\delta< \frac 12 \alpha_{\max}^{-1}  < \zeta_{\min}$,  see~\eqref{pfti-e12}. This give
\begin{align}\label{pftii-e14}
\begin{aligned}
    K_t^{1;c}(x,y)
    &\leq C t^{-d \zeta(y)}   +   C t^{-d\zeta_{\min}-N} e^{- c' t^{\delta-\zeta_{\min}}} \\
    &\leq C'  t^{-d \zeta(y)} \\
    &= C' t^{-d/\alpha(y)+ \sfrak d}\\
    &\leq C t^{-d/\alpha_{\min}+  \sfrak d}.
\end{aligned}
\end{align}
It remains to show that we have for suitable $C,\omega>0$ and $i=1,2$
\begin{align}\label{pftii-e16}
    \sup_{x,y\in\rd}  |B_i(t,x,y)| \leq C t^{-d/\alpha_{\min}-1+ \omega}.
\end{align}
Let us first estimate $B_1(t,x,y)$.  Suppose that $t^{\zeta(x)}< t^{\zeta(y)}$. Using again the estimate $p^0_t(x,y) \leq C t^{-d/\alpha(y)}$  we get from the definition~\eqref{pfti-e12} of $\zeta(x)$ and the scaling property~\eqref{set-e08} that
\begin{align}\label{pftii-e18}
    |B_1(t,x,y)|
    \leq C t^{-d/\alpha(y)} \mu\left(x, \{u: |u|\geq t^{\zeta(x)}\}\right)
    \leq C' t^{-d/\alpha_{\min}- 1+ \sfrak \alpha_{\min}}.
\end{align}
If  $t^{\zeta(y)}\leq t^{\zeta(x)}$, \eqref{pfti-e52} with $k=0$ and the definition of $K_t^{0;c}$ yields
\begin{align*}
    |B_1(t,x,y)|
    &\leq C t^{-d/\alpha(y)} \left(\int_{t^{\zeta(y)}< |u|\leq t^{\zeta(x)}}+ \int_{ |u|> t^{\zeta(x)}}\right) e^{- c|\kappa_t(y)-x-u|t^{-\zeta(y)}} \, \mu(x,du)\\
    &=: \mathrm{I}_1(t,x,y)+\mathrm{I}_2(t,x,y).
\end{align*}
As in~\eqref{pftii-e18}, we have
\begin{align*}
    \mathrm{I}_2(t,x,y)
    \leq  C t^{-d/\alpha(y)} \mu\left(x, \{u: |u|\geq t^{\zeta(x)}\}\right)
    \leq C' t^{-d/\alpha_{\min}- 1+ \sfrak \alpha_{\min}},
\end{align*}
and it remains to estimate $\mathrm{I}_1(t,x,y)$. Using the last  three lines   in~\eqref{pfti-e88},  we get
\begin{align*}
    \mathrm{I}_1(t,x,y)
    &\leq C  t^{-d/\alpha(y)-\alpha(x)\zeta(y)} e^{-c |\kappa_t(y)-x| t^{-\zeta(x)}}\\
    &= C t^{-d/\alpha(y)-\alpha(x)\zeta(x)} t^{- \alpha(x)(\zeta(y)-\zeta(x))} e^{- c |\kappa_t(y)-x|t^{-\zeta(x)}}\\
    &\leq C  t^{-d/\alpha(y)-1+\sfrak\alpha_{\min}} t^{- \alpha(x)(\zeta(y)-\zeta(x))} e^{-c |\kappa_t(y)-x|t^{-\zeta_{\min}}}\\
    &\leq
    \begin{cases}
        C t^{-d/\alpha_{\min}-1+\sfrak\alpha_{\min}}, &|\kappa_t(y)-x|\leq t^\delta, \\[10pt]
        C t^K , & |\kappa_t(y)-x|> t^\delta.
    \end{cases}
\end{align*}
The top line in the last estimate follows from Lemma~\ref{appA-05} since
$|y- x|\leq |\kappa_t(y)-x|  + |y-\kappa_t(y)|\leq t^\delta + C t^{\frac 12} \leq C' t^\delta$, cf.\ Remark~\ref{appA-11}; the bottom line holds for any $K\geq 1$, since $\delta < \zeta_{\min}$.

Now we estimate  $B_2(t,x,y)$. Using $p^0_t(x,y) \leq C t^{-d/\alpha(y)}$, cf.~\eqref{pfti-e52}, and~\eqref{N1} we get
\begin{align*}
    |B_2(t,x,y)|
    &\leq C t^{-d/\alpha(y)} |\nu|(x, \{u: |u|> t^{1/\alpha(x)}\})\\
    &\leq C' t^{-d/\alpha_{\min}} t^{-\beta(x)/\alpha(x)}\\
    &\leq C' t^{-d/\alpha_{\min}-1+ \epsilon_\nu/\alpha_{\max}}.
\end{align*}
Combining~\eqref{pftii-e10}, \eqref{pftii-e04}, \eqref{pftii-e14} and~\eqref{pftii-e16}, we get~\eqref{pftii-e06}.

\medskip\noindent
\emph{Verification of~\eqref{pftii-e08}.}
We begin with the estimates of the integrals in $x$ of $p^0_t(x,y)$ and $K_t^{1;c}(x,y)$. From~\eqref{pfti-e26} we get
\begin{align*}
   \int_\rd p^0_t(x,y)\,dx=1.
\end{align*}
Integrating~\eqref{pftii-e12} in $x$ we get
\begin{align*}
   \sup_{y\in\rd}  \int_\rd K_t^{1;c}(x,y)\,dx \leq C.
\end{align*}
We will now prove that there exist $C,\omega>0$ such that
\begin{align}\label{pftii-e20}
    \sup_{y\in \rd}  \int_\rd  | B_1(t,x,y)+ B_2(t,x,y)| \,dx
    \leq C t^{-1+\omega}.
\end{align}
Split
\begin{align*}
    B_2(t,x,y)
    &=\left(\int_{t^{1/\alpha(x)}<|u|\leq t^{\zeta(y)}} + \int_{|u| > t^{\zeta(y)}}\right)  p^0_t(x+u,y)\,\nu(x,du)\\
    &=: \mathrm{J}_1(t,x,y)+ \mathrm{J}_2(t,x,y);
\end{align*}
if  $t^{\zeta(y)} < t^{1/\alpha(x)}$, we have $\mathrm{J}_1(t,x,y)\equiv 0$ and there is nothing to show in this case. So we assume that $t^{\zeta(y)} \geq t^{1/\alpha(x)}$.  We begin with $\int_\rd |J_1(t,x,y)|\,dx$ and then we estimate $\int_\rd |J_2(t,x,y) +  B_1(t,x,y)|\,dx$.  We know from~\eqref{pfti-e52} with $k=0$ that
\begin{align*}
    \int_\rd | \mathrm{J}_1&(t,x,y)|\,dx\\
    & \leq  C t^{-\sfrak d}
        \int_\rd \int_{t^{\frac 1{\alpha(x)}}<|u|\leq t^{\zeta(y)}} t^{-\zeta(y)d}  e^{-c |\kappa_t(y)-x-u|t^{-\zeta(y)}} \,|\nu|(x,du)\,dx\\
    &\leq C t^{-\sfrak d} \int_\rd t^{-\zeta(y)d}  e^{-c |\kappa_t(y)-x|t^{-\zeta(y)}} |\nu|(x, \{ |u|\geq t^{1/\alpha(x)}\})\,dx.
\end{align*}
Now we use~\eqref{N1} and get
\begin{align*}
  \int_\rd | \mathrm{J}_1(t,x,y)|\,dx
    & \leq  C t^{-\sfrak d} \int_\rd t^{-\zeta(y)d}  e^{-c |\kappa_t(y)-x|t^{-\zeta(y)}} t^{-\beta(x)/\alpha(x)}\,dx.
\intertext{Since $t^{-\beta(x)/\alpha(x)} \leq t^{-1+ \epsilon_\nu /\alpha_{\max}}$, we finally see}
   \int_\rd | \mathrm{J}_1(t,x,y)|\,dx
    &\leq  C t^{-\sfrak d}t^{-1+ \epsilon_\nu /\alpha_{\max}} \int_\rd t^{-\zeta(y)d}  e^{-c |\kappa_t(y)-x|t^{-\zeta(y)}} \,dx\\
    &\leq  C t^{-\sfrak d}t^{-1+ \epsilon_\nu /\alpha_{\max}}.
\end{align*}
Without loss of  generality we may assume that  $\sfrak$ is small, hence
\begin{align}\label{pftii-e22}
    \sup_{y\in\rd} \int_\rd | \mathrm{J}_1(t,x,y)|\,dx
    \leq C t^{-1+ \epsilon_N}.
\end{align}

Now we  estimate  the integral of $|J_2(t,x,y)+  B_1(t,x,y)|$.
For that, we need the following improved estimate of $p^0_t(x,y)$. As a simple consequence of its definition as $p^0_t(x,y)= p_t^{y,\cut}(\kappa_t(y)-x)$, and the estimate proved in Proposition~\ref{appC-21} for $p_t^{z,\cut}(w-x)$, we have
\begin{align}\label{pftii-e24}
    p^0_t(x,y)
    \leq C \int_\rd   t^{-d/\alpha(y)} e^{-c |\kappa_t(y)-x- k(t,y)-z-u| t^{-1/\alpha(y)}}\,|P_t|(y,dz);
\end{align}
by $|P_t|(y,du)$ we denote the total variation measure of the kernel $P_t(y,du)$ from~\eqref{appC-e50}. Note that $\sup_{y\in \rd} |P_t|(y,\rd) \leq C$,  and the function  $k(t,z)$ from~\eqref{appC-e54} satisfies~\eqref{appC-e56}.

In the estimates below we will need the  inequality
\begin{align}\label{pftii-e26}
    e^{-|x|t^{-a}} \leq C t^{-ad} \int_{|v|\leq t^a} e^{-|x-v|t^{-a}}\,dv,\quad a>0,
\end{align}
which follows from
\begin{align*}
    \int_{|v|\leq t^a} e^{-|x-v|t^{-a}}\,dv
    =  t^{ad} \int_{|v|\leq 1} e^{-|xt^{-a}-v|}\,dv
    \geq C t^{ad} e^{-|x|t^{-a}}.
\end{align*}
Below, in the steps marked by ($*$) we apply first~\eqref{pftii-e24} and then~\eqref{pftii-e26}:
\begin{align*}
    &\int_\rd |\mathrm{J}_2(t,x,y)+  B_1(t,x,y)| \,dx \\
    &\leq  \int_\rd \int_{|u|\geq t^{\zeta(y)}} p^0_t(x+u,y)\,N(x,du)\,dx\\
    &\stackrel{(*)}{\leq} C  \int_\rd  \int_{|u|\geq t^{\zeta(y)}}
        \left[\int_\rd  t^{-d/\alpha(y)}  e^{-c |\kappa_t(y)-x- k(t,y)-z-u| t^{-1/\alpha(y)} }\, |P_t|(y,dz)  \right] N(x,du)\,dx\\
    &\stackrel{(*)}{\leq}  C t^{- d/\alpha(y)} \int_\rd \int_{|u|\geq t^{\zeta(y)}}\int_{|v|\leq t^{1/\alpha(y)}} \\
    &\qquad\qquad \times \left[ \int_\rd  t^{-d/\alpha(y)} e^{-c |\kappa_t(y)-x- k(t,y)-z-u-v| t^{-1/\alpha(y)} } \, |P_t|(y,dz)  \right] dv\,N(x,du)\, dx \\
    &= C t^{-d/\alpha(y)}  \int_\rd \int_\rd   \left[\int_{|u|\geq t^{\zeta(y)}}  \I_{\{|\kappa_t(y)-x-k(t,y)-u-w|\leq t^{1/\alpha(y)}\}}\,N(x,du)\right]\times\\
    &\qquad\qquad\mbox{} \times \left[  \int_\rd  t^{-d/\alpha(y)}  e^{-c |w-z| t^{-1/\alpha(y)} } \, |P_t|(y,dz)  \right] dw\,dx\\
    &\leq C \int_\rd \int_\rd  t^{-d/\alpha(y)}
          N\left(x,\big\{ u: \, |u|\geq t^{\zeta(y)}, \,|\kappa_t(y)-x-k(t,y)-u-w|\leq t^{1/\alpha(y)}\big\}\right) dx\\
    &\qquad\qquad \times \left[  \int_\rd  t^{-d/\alpha(y)} e^{-c |w-z| t^{-1/\alpha(y)}} \, |P_t|(y,dz)  \right] dw
    =:\mathbf{I}(t,y).
\end{align*}
In order to estimate the last integral, we consider two cases. Fix some $t_*\in (0,1]$, the value of $t_*$ will be chosen later on.

\medskip\noindent
\emph{Case 1: $t>t_*$.}  Since $t^{\zeta(y)}=t^{1/\alpha(y)-\sfrak}\geq t^{1/\alpha(y)}$, we can use~\eqref{set-e34} with $v=\kappa_t(y)-k(t,y)-w$ and $\alpha=\alpha(y)$ to get for all $t>t_*$
\begin{align*}
    \int_\rd \int_\rd  t^{-d/\alpha(y)}&
          N\left(x,\big\{ u: \, |u|\geq t^{\zeta(y)}, \,|\kappa_t(y)-x-k(t,y)-u-w|\leq t^{1/\alpha(y)}\big\}\right) dx\\
    &\qquad\qquad \times \left[  \int_\rd  t^{-d/\alpha(y)} e^{-c |w-z| t^{-1/\alpha(y)}} \, |P_t|(y,dz)  \right] dw \\
    &\leq Ct^{-\alpha_{\max}/\alpha_{\min}} \int_\rd  \int_\rd  t^{-d/\alpha(y)} e^{-c |w-z| t^{-1/\alpha(y)}} \, |P_t|(y,dz)  dw\\
    &\leq Ct^{-\alpha_{\max}/\alpha_{\min}}\\
    &\leq C(t_*);
\end{align*}
in the penultimate inequality we use that $|P_t|(y,\rd)$ is bounded, see~\eqref{appC-e52}.

\emph{Case 2: $t\leq t_*$.} We split the integral $\mathbf{I}(t,y)$ in two parts, $\mathbf{I}_{1}(t,y) + \mathbf{I}_{2}(t,y)$, in the following way:
\begin{align*}
    &N\left(x,\big\{ u: \, |u|\geq t^{\zeta(y)}, \,|\kappa_t(y)-x-k(t,y)-u-w|\leq t^{1/\alpha(y)}\big\}\right)\\
    &\quad = N\left(x,\big\{ u: \, |u|\geq t^{\mathfrak{r}}, \,|\kappa_t(y)-x-k(t,y)-u-w|\leq t^{1/\alpha(y)}\big\}\right)\\
    &\qquad\mbox{}+N\left(x,\big\{ u: \, t^{\zeta(y)}\leq |u|<t^{\mathfrak{r}}, \,|\kappa_t(y)-x-k(t,y)-u-w|\leq t^{1/\alpha(y)}\big\}\right),
\end{align*}
where $\mathfrak{r}$ is taken from~\eqref{set-e36} (this includes the possibility that the second set on the left is actually empty). Applying~\eqref{set-e36} with $v=\kappa_t(y)-k(t,y)-w$ and $\alpha=\alpha(y)$, yields
\begin{align*}
    &\mathbf{I}_{1}(t,y)\\
    &=\int_\rd \int_\rd  t^{-d/\alpha(y)}
          N\left(x,\big\{ u: \, |u|\geq t^{\mathfrak{r}}, \,|\kappa_t(y)-x-k(t,y)-u-w|\leq t^{1/\alpha(y)}\big\}\right) dx\\
    &\qquad\qquad \times \left[  \int_\rd  t^{-d/\alpha(y)} e^{-c |w-z| t^{-1/\alpha(y)}} \, |P_t|(y,dz)  \right] dw\\
    &\leq C t^{ -1+ \epsilon_{\mathfrak{r}}} \int_\rd\int_\rd  t^{-d/\alpha(y)} e^{-c |w-z| t^{-1/\alpha(y)} } \, dw\, |P_t|(y,dz)\\
    &\leq C t^{ -1+ \epsilon_{\mathfrak{r}}}.
\end{align*}
In order to estimate the second part, we make yet another split. Note that the signed kernel $\Lambda(t,y,dz)$ defined by~\eqref{appC-e46} is supported by $\{|z|\leq t^{\zeta(y)}\}$ and has total variation which is bounded by $Ct^{\sfrak\alpha_{\min}}$, see~\eqref{appC-e48}. Since $P_t(y,du)$ is  the convolution-exponential of $\Lambda(t,y,dz)$, we find for any $Q>0$ some $n_Q\in \mathds{N}$ such that
\begin{align*}
|P_t|(y,\{|z|\geq n_Qt^{\zeta(y)}\})\leq Ct^{Q}.
\end{align*}
Pick $Q$ large enough such that $-\alpha_{\max}/\alpha_{\min}+Q>-1$, and split $\mathbf{I}_{2}(t,y)$ into two integrals $\mathbf{I}_{21}(t,y)$ and $\mathbf{I}_{22}(t,y)$ using the decomposition
\begin{align*}
    \rd\times\rd
    = & \left\{(w,z)\,:\, |z|\geq n_Qt^{\zeta(y)} \text{\ or\ } |w-z|\geq t^{1/(2\alpha_{\max})}\right\}\\
      &\quad\mbox{}\cup \left\{(w,z)\,:\,|z|< n_Qt^{\zeta(y)} \text{\ and\ } |w-z|<t^{1/(2\alpha_{\max})}\right\}
      =:\mathbf{D}_{21}\cup \mathbf{D}_{22}.
\end{align*}

Since $t^{\zeta(y)}=t^{1/\alpha(y)-\sfrak}\geq t^{1/\alpha(y)}$, we see using~\eqref{set-e34} with $v=\kappa_t(y)-k(t,y)-w$ and $\alpha=\alpha(y)$
\begin{align*}
    &\mathbf{I}_{21}(t,y) \\
    &\leq\iint_{\mathbf{D}_{21}} t^{-d/\alpha(y)} e^{-c |w-z| t^{-1/\alpha(y)}} \, |P_t|(y,dz)\,  dw\\
    &\quad\mbox{}\times\int_\rd  t^{-d/\alpha(y)}
          N\left(x,\big\{ u: \, |u|\geq t^{\zeta(y)}, \,|\kappa_t(y)-x-k(t,y)-u-w|\leq t^{1/\alpha(y)}\big\}\right) dx\\
    &\leq Ct^{-\alpha_{\max}/\alpha_{\min}}\iint_{\substack{\{|z|\geq n_Qt^{\zeta(y)}\} \\ \cup \{|w-z|\geq t^{1/(2\alpha_{\max})}\}}} t^{-d/\alpha(y)} e^{-c |w-z| t^{-1/\alpha(y)}} \, |P_t|(y,dz)\,  dw\\
    &\leq Ct^{-\alpha_{\max}/\alpha_{\min}}\Big(t^Q+e^{-c t^{-1/(2\alpha_{\max})} }\Big)\\
    &\leq Ct^{-\alpha_{\max}/\alpha_{\min}+Q}\\
    &=Ct^{-1+\epsilon_Q}
\end{align*}
for $\epsilon_Q=1-\alpha_{\max}/\alpha_{\min}+Q>0$.

For the estimate of $\mathbf{I}_{22}(t,y)$, we observe that
\begin{align*}
    |y-x|\leq |\kappa_t(y)-x-k(t,y)-u-w|+|\kappa_t(y)-y|+ |k(t,y)|+|u|+ |z|+|w-z|.
\end{align*}
Thus, if $|z|<n_Q t^{\zeta(y)}$ and $|w-z|<t^{1/(2\alpha_{\max})}$, then the bounds~\eqref{appA-e14} and~\eqref{appC-e56} yield
\begin{align*}
    |u|\leq t^{\mathfrak{r}} \;\&\; |\kappa_t(y)-x-k(t,y)-u-w|\leq  t^{1/\alpha(y)}
    \implies
    |y-x|\leq Ct^{\mathfrak{r}'}
\end{align*}
if we set $\mathfrak{r}'=\min\left\{\mathfrak{r}, \frac{1}{2\alpha_{\max}}, \frac12\right\}>0$.

As in Corollary~\ref{appA-07}, the condition $|x-y| \leq Ct^{\mathfrak{r}'}$ implies that for any $\delta>0$ there is some $t_\delta>0$ such that
\begin{align*}
    t^{1/\alpha(x)+\delta}<t^{1/\alpha(y)}<t^{1/\alpha(x)-\delta}, \quad  t\in (0, t_\delta].
\end{align*}
That is, for $t\in (0, t_\delta]$ and $z,w$ as above we have
\begin{align}\label{pftii-e28}
\begin{split}
        &\int_\rd  t^{-d/\alpha(y)}
          N\left(x,\big\{ u: \, t^{\zeta(y)}\leq |u|\leq t^{\mathfrak{r}}, \,|\kappa_t(y)-x-k(t,y)-u-w|\leq t^{1/\alpha(y)}\big\}\right) dx\\
          &\leq \int_\rd  t^{-d/\alpha(x)-\delta d}
          N\left(x,\big\{ u: \, |u|\geq t^{\zeta(x)+\delta}, \,|\kappa_t(y)-x-k(t,y)-u-w|\leq t^{1/\alpha(x)-\delta}\big\}\right) dx
\end{split}
\end{align}
We want to use assumption~\eqref{set-e32}; for this, we rearrange the above integral in order to get rid of the terms $\pm \delta$ in the integrand. Take $t'=t^{1-\delta\alpha_{\max}}$, then
\begin{align*}
    t^{1/\alpha(x)-\delta}
    = (t')^{(1/\alpha(x)-\delta)/(1-\delta\alpha_{\max})}
    = (t')^{(1-\alpha(x)\delta)/((1-\delta\alpha_{\max})\alpha(x))}
    \leq (t')^{1/\alpha(x)},
\end{align*}
and
\begin{align*}
    t^{\zeta(x)+\delta}
    &= t^{1/\alpha(x)+\delta-\sfrak}\\
    &= (t')^{1/((1-\delta\alpha_{\max})\alpha(x))-\sfrak/(1-\delta\alpha_{\max})+\delta/(1-\delta\alpha_{\max})}\\
    &\geq (t')^{1/\alpha(x)-\sfrak/(1-\delta\alpha_{\max})+\delta(1+\alpha_{\max}/\alpha_{\min})/(1-\delta\alpha_{\max})},
\end{align*}
where we use, in the last inequality, the fact that
\begin{align*}
    \frac1{(1-\delta\alpha_{\max})\alpha(x)} - \frac1{\alpha(x)}
    = \frac{\delta \alpha_{\max}}{\alpha(x)(1-\delta\alpha_{\max})}
    \leq \frac{\delta \alpha_{\max}/\alpha_{\min}}{1-\delta\alpha_{\max}}.
\end{align*}
may assume that $\delta$ is so small that
\begin{align}\label{pftii-e30}
    \frac\sfrak{1-\delta\alpha_{\max}}-\frac{\delta(1+\alpha_{\max}/\alpha_{\min})}{1-\delta\alpha_{\max}}>\frac{\sfrak}{2}.
\end{align}
Thus, we have
\begin{align*}
    t^{\zeta(x)+\delta} \geq (t')^{1/\alpha(x)- \sfrak/2}.
\end{align*}
From the definition of $t'$ we get
\begin{align*}
    t^{-1/\alpha(x)}
    \leq (t')^{-1/\alpha(x)} t^{-\delta\alpha_{\max}/\alpha_{\min}}.
\end{align*}
We can now continue our estimation of~\eqref{pftii-e28}. Using~\eqref{set-e32} with $\mathfrak{q}:=\frac{\sfrak}{2}$ and $t'$ instead of $t$ we obtain
\begin{align*}
    &\int_\rd t^{-d/\alpha(y)}
        N\left(x,\big\{ u: \, t^{\zeta(y)}\leq |u|\leq t^{\mathfrak{r}}, \,|\kappa_t(y)-x-k(t,y)-u-w|\leq t^{1/\alpha(y)}\big\}\right) dx\\
    &\leq  t^{-\delta \alpha_{\max} d/\alpha_{\min}-\delta d}\times\mbox{} \\
    &\quad \mbox{}\times \int_\rd  (t')^{-d/\alpha(x)}
          N\left(x,\big\{ u: \, |u|>(t')^{1/\alpha(x)-\sfrak/2}, \,|\kappa_t(y)-x-k(t,y)-u-w|\leq (t')^{1/\alpha(x)}\big\}\right) dx\\
    &\leq C(t')^{-1+\epsilon_{{\mathfrak{q}}}} t^{-\delta d(\alpha_{\max}/\alpha_{\min})-\delta d}\\
    &\leq Ct^{-1+\epsilon_{{\mathfrak{q}}}} t^{-\delta d(\alpha_{\max}/\alpha_{\min})-\delta d}
\end{align*}
(recall that $t\leq t'$). Making $\delta$ even smaller, if necessary, we can achieve that
\begin{align}\label{pftii-e32}
    \delta d \left(\frac{\alpha_{\max}}{\alpha_{\min}}+1\right)<\frac{\epsilon_{{\mathfrak{q}}}}{2}.
\end{align}
Thus, we finally get the bound
\begin{align*}
    \int_\rd  t^{-d/\alpha(y)}
          N\left(x,\big\{ u: \, t^{\zeta(y)}\leq |u|\leq t^{\mathfrak{r}}, \,|\kappa_t(y)-x-k(t,y)-u-w|\leq t^{1/\alpha(y)}\big\}\right) dx
    \leq Ct^{-1+\epsilon_{{\mathfrak{q}}}/2},
\end{align*}
which holds for all $t\in (0, t_\delta]$ and all $w$ such that $(w,z)\in \mathbf{D}_{22}$  for some $z$.  Using \eqref{set-e32},  we get for all $t<t_\delta$
\begin{align*}
    &\mathbf{I}_{22}(t,y)
    =\iint_{ \mathbf{D}_{22}} t^{-d/\alpha(y)} e^{-c |w-z| t^{-1/\alpha(y)}} \, |P_t|(y,dz)\,  dw\\
    &\qquad\qquad\mbox{} \times \int_\rd  t^{-d/\alpha(y)}
          N\left(x,\big\{ u: \, |u|\geq t^{\mathfrak{r}}, \,|\kappa_t(y)-x-k(t,y)-u-w|\leq t^{1/\alpha(y)}\big\}\right) dx\\
    &\leq Ct^{-1+\epsilon_{{\mathfrak{q}}}/2}\iint_{ \mathbf{D}_{22}} t^{-d/\alpha(y)} e^{-c |w-z| t^{-1/\alpha(y)}} \, |P_t|(y,dz)\,  dw\\
    &\leq Ct^{-1+\epsilon_{{\mathfrak{q}}}/2}\int_{\rd}\int_{\rd} t^{-d/\alpha(y)} e^{-c |w-z| t^{-1/\alpha(y)}} \, |P_t|(y,dz)\,dw\\
    &\leq Ct^{-1+\epsilon_{{\mathfrak{q}}}/2}.
\end{align*}

Let us finally combine all estimates and choose the parameters: First we take $\delta>0$ so small that~\eqref{pftii-e30} and~\eqref{pftii-e32} hold, then we set $t_*=t_\delta$. Thus,
\begin{align*}
    \mathbf{I}(t,y)
    \leq
    \begin{cases}
        C, & t\geq t_*; \\[10pt]
        Ct^{-1+\epsilon_{{\mathfrak{r}}}}+Ct^{-1+\epsilon_Q}+Ct^{-1+\epsilon_{{\mathfrak{q}}}/2}, & t\in (0, t_*),
    \end{cases}
\end{align*}
which gives
\begin{align*}
    \int_\rd |\mathrm{J}_2(t,x,y)+  B_1(t,x,y)| \,dx\leq Ct^{-1+\epsilon}
\end{align*}
with $\epsilon=\min\{\epsilon_{{\mathfrak{r}}}, \epsilon_Q, \epsilon_{{\mathfrak{q}}}/2\}$. Combining this and~\eqref{pftii-e22} gives~\eqref{pftii-e20}.
\end{proof}

We can finally complete the proof of Theorem~\ref{t3}.
\begin{proof}[Proof of Theorem~\ref{t3}]
    Write $p_t(x,y)$ in the form~\eqref{pftii-e02}.  We have proved the operator bounds~\eqref{para-e40}, \eqref{para-e42}, and~\eqref{para-e46}, which yield~\eqref{para-e50}; that is, the required estimate for the third term in~\eqref{pftii-e02} holds. The estimate for the second term follows from~\eqref{appC-e30}, see Proposition~\ref{appC-17}.
\end{proof}

\appendix
\section{Properties of $b_t$, $B_t$ and the flows $\chi_t$, $\kappa_t$}\label{appA}

In order to construct the fundamental solution of the non-local operator~\eqref{set-e02}, we consider its dynamically compensated drift~\eqref{set-e12}
\begin{align*}
    b_t(x)
    =  b(x)- \int_{(1\wedge t)^{1/\alpha(x)} <|u|\leq 1} u \,N(x,du).
\end{align*}
By $B_t$ we denote the mollified version of $b_t$, see~\eqref{pfti-e22},
\begin{align*}
    B_t (x)
    = \int_\rd  b_t(y) \phi_{t^{1/\theta(x)}}(x-y)\,dy
\end{align*}
($\phi_s$ is the usual Friedrichs mollifier kernel). Since $B_t(\cdot)$ is Lipschitz continuous, we can study the  unique  deterministic (backward and forward) flows  which are induced by the mollified drift
\begin{align}\label{appA-e02}
    \left\{\begin{aligned}
    \frac{d}{dt}\kappa_t(y) &= - B_t (\kappa_t(y)), \quad t> 0,\\
    \kappa_0(y) &= y,
    \end{aligned}\right.
    \quad\text{and}\quad
    \left\{\begin{aligned}
    \frac{d}{dt}\chi_t(x) &= B_t (\chi_t(x)), \quad t>0,\\
    \chi_0(x) &= x.
    \end{aligned}\right.
\end{align}
These flows enable us to deal with the anisotropic nature of the operator~\eqref{set-e02} when freezing its coefficients, see Sections~\ref{road}, \ref{zero}.

In this appendix we study the properties of the flows $\chi_t$ and $\kappa_t$.  Unless otherwise mentioned, we assume throughout this appendix~\eqref{M0}--\eqref{M2}, \eqref{N1} and~\eqref{B0}, \eqref{B1}.  We begin with three technical results.
The first lemma is a generalization of~\cite[Proposition~A.1]{Ku}.

\begin{lemma}\label{appA-03}
    Let $M(du)$ be a signed measure on $\rd$ satisfying
    \begin{align*}
        |M|(\{|u|>r\})
        \leq C_M r^{-\beta}, \quad r\in (0, 1]
    \end{align*}
    for some $\beta\in (0,2)$. There exists a constant $C>0$ such that for every $\epsilon\in (0,1)$
    \begin{align}\label{appA-e04}
        \left|\int_{|u|\leq \epsilon} u\,M(du)\right|
        \leq  \int_{|u|\leq \epsilon} |u|\,|M|(du)
        &\leq C \epsilon^{1-\beta},\qquad\text{if $\beta < 1$};
    \\\label{appA-e06}
        \left|\int_{\epsilon<|u|\leq 1} u\, M(du)\right|
        \leq  \int_{\epsilon < |u|\leq 1} |u|\,|M|(du)
        &\leq
        \begin{cases}
        C \epsilon ^{1-\beta}, &\text{if $\beta\neq 1$},\\[10pt]
        C(1+|\log \epsilon|), &\text{if $\beta = 1$}.
        \end{cases}
    \end{align}
\end{lemma}
\begin{proof}
We will only prove~\eqref{appA-e04}, the proof of~\eqref{appA-e06} is similar.  Using  integration by parts, we get
\begin{align*}
    \left|\int_{|u|\leq \epsilon} u\,M(du)\right|
    &\leq \int_0^\epsilon |M|\left(\left\{|u|>r\right\}\right) dr
    \leq C_M \int_0^\epsilon r^{-\beta}\, dr
    =\frac{C_M}{1-\beta}\epsilon^{1-\beta}.
\qedhere
\end{align*}
\end{proof}

\begin{lemma}\label{appA-05}
    Let $w(x): \rd\to (0,\infty)$ be some bounded $\gamma$-H\"older continuous function and set $w_{\max} := \sup_{x\in \rd} w(x)$. For every $\delta>0$ there exists some $C>0$ such that for all $x,y\in\rd$ and $t\in (0,1]$
    \begin{align}\label{appA-e08}
        e^{-C  t^{\gamma\delta/2}}  \I_{\{|x-y|\leq t^\delta\}}
        \leq t^{w(x)-w(y)}
        \leq e^{Ct^{\gamma\delta/2}}  \I_{\{|x-y|\leq t^\delta\}} + t^{-w_{\max}}  \I_{\{|x-y|> t^\delta\}}.
    \end{align}
\end{lemma}

\begin{proof}
Because of the H\"older continuity of $w$ we have for $|x-y|\leq t^\delta$ and all $t\in (0,1]$,
\begin{align*}
    -c t^{\delta \gamma}
    \leq -c |x-y|^\gamma
    \leq w(x)-w(y)
    \leq c |x-y|^\gamma
    \leq c t^{\delta \gamma}
\end{align*}
for some $c>0$.  Hence,
\begin{align*}
    t^{c t^{\gamma\delta}}  \I_{\{|x-y|\leq t^\delta\}}
    \leq t^{w(x)-w(y)}
    \leq t^{-c t^{\gamma\delta}}  \I_{\{|x-y|\leq t^\delta\}} + t^{-w_{\max}}  \I_{\{|x-y|> t^\delta\}}.
\end{align*}
Since $\lim_{t\to 0} t^{\gamma\delta/2} |\log t| = 0$, there exists a constant $C>0$ such that $e^{-C  t^{\gamma\delta/2}}\leq t^{  ct^{\gamma\delta}}\leq e^{C  t^{\gamma\delta/2}}$ for all $t\in (0,1]$, and the claim follows.
\end{proof}

\begin{corollary}\label{appA-07}
    Let $w(x): \rd\to (0,\infty)$ be a bounded $\gamma$-H\"older continuous function.  If $|x-y|\leq t^\delta$ for some $\delta>0$, then $t^{w(x)}\asymp t^{w(y)}$.
\end{corollary}

Lemmas~\ref{appA-03} and~\ref{appA-05} allow us to establish the following key estimates on $b_t$ and $B_t$.
\begin{proposition}\label{appA-09}
    There exists a constant $C>0$ such that for all $x\in \rd$ and $t\in (0,1]$
    \begin{align}\label{appA-e10}
        |b_t(x)|
        \leq
        \begin{cases}
            C, & 0< \alpha(x)<1; \\[10pt]
            C |\log t|, & \alpha(x)=1;\\[10pt]
            C t^{-1+1/\alpha(x)}, & 1<\alpha(x)<2.
            \end{cases}
    \end{align}
    In particular, there exists a constant $C>0$ such that for all $t\in (0,1]$
    \begin{align}\label{appA-e12}
        \sup_{x\in \rd} |b_t(x)|\leq C t^{-\frac 12}
        \quad\text{and}\quad
        \sup_{x\in \rd} |B_t(x)|\leq C t^{-\frac 12}.
    \end{align}
\end{proposition}

\begin{proof}
It follows from the conditions~\eqref{B0}, \eqref{N1}  and Lemma~\ref{appA-03} that
\begin{align*}
    |b_t(x)|
    \leq
    \begin{cases}
    C + C t^{1/\alpha(x)-1}, & \alpha(x) \neq 1,\\[10pt]
    C + C |\log t|, & \alpha(x) = 1.
    \end{cases}
\end{align*}
This proves~\eqref{appA-e10} and the first part of~\eqref{appA-e12}. The second part of~\eqref{appA-e12} follows from the fact that $B_t$ is the  Friedrichs  mollification of $b_t$.
\end{proof}

\begin{remark}\label{appA-11}
    If we apply the estimate~\eqref{appA-e12} to the very definition of the flow $\chi_t$, we see that
    \begin{align*}
        |\chi_t(x)- x|
        \leq \int_0^t |B_s(\chi_s(x))|\,ds
        \leq C \int_0^t s^{-\frac 12}\,ds
        \leq 2 C t^{\frac 12}, \quad x\in\rd,\; t\in (0,1].
    \end{align*}
    A similar estimate holds for $\kappa_t$:
    \begin{align}\label{appA-e14}
        |\kappa_t(y)-y|
        \leq 2C t^{\frac 12},  \quad y\in\rd,\; t\in (0,1].
    \end{align}
\end{remark}

Recall that $\hfrak$ and $\epsilon$ were defined in~\eqref{B1}. The following lemma provides an estimate of the approximation of $b_t$ by $B_t$, and an estimate of the ($t$-dependent) Lipschitz constant for $B_t(\cdot)$.

\begin{proposition}\label{appA-13}
    There exists a constant $C>0$, such that for $\epsilon_B := \frac 14\min\left\{\hfrak, \epsilon\right\}$ and all $t\in (0,1]$ and $x\in\rd$
    \begin{align}\label{appA-e16}
     |b_t(x)- B_t(x)|
     &\leq C t^{1/\alpha(x)} t^{-1+\epsilon_B},
    \\\label{appA-e18}
    \Lip(B_t(\cdot))
    &\leq  C t^{-1+\epsilon_B}.
    \end{align}
\end{proposition}
\begin{proof}
   Recall that $B_t$ is defined by~\eqref{pfti-e22} and  $\phi_{s}(x) = s^{-d}\phi(s^{-1}x)$ is a test function with  $\supp \phi_s = \overline{B(0,s)}$ and $\int_{\rd} \phi_s(x)\,dx = 1$. Using~\eqref{B1} we see
    \begin{align*}
    &|b_t(x)-B_t(x)|\\
    &=\int_\rd |b_t(x)-b_t(y)| \, \phi_{t^{\frac{1}{\theta(x)}}}(y-x)\, dy   \\
    &\leq  C\hspace{-6mm} \int\limits_{|x-y|\leq t^{1/\theta(x)}} \hspace{-6mm} \left[|x-y|^{\gamma(x)} + |x-y|^{\gamma(y)}  + (t^{\delta(x)} +t^{\delta(y)})|x-y|^\epsilon\right]  \phi_{t^{\frac{1}{\theta(x)}}}(y-x)\, dy.
    \end{align*}
    The exponents $\theta(x)$ and $\gamma(x)$, $\delta(x)$ were defined in~\eqref{pfti-e18} and~\eqref{B1}, respectively. Note that $\gamma(x)$ and $\delta(x)$ inherit the H\"older continuity from $\alpha(x)$. Therefore, we can use Lemma~\ref{appA-05} and conclude that for $|x-y|\leq t^{1/\theta(x)}$
    \begin{align}\label{appA-e20}
        |x-y|^{\gamma(y)}\leq t^{\gamma(x)/\theta(x)} (t^{1/\theta(x)})^{\gamma(y)-\gamma(x)} \leq C t^{\gamma(x)/\theta(x)}
        \quad\text{and}\quad
        t^{\delta(y)}\asymp t^{\delta(x)}.
     \end{align}
    This implies,  as $\theta(x)\leq 2$,
    \begin{align*}
        |b_t(x)-B_t(x)|
        &\leq C \left(t^{\gamma(x)/\theta(x)} + t^{\delta(x)+ \epsilon/\theta(x)}\right)\\
        &\leq  C t^{1/\alpha(x)} \left(t^{\gamma(x)/\theta(x)-1/\alpha(x)} +  t^{-1+  \epsilon/2}\right).
    \end{align*}
    Moreover, see~\eqref{B1} for the definition of $\hfrak$,
    \begin{align*}
        \frac{\gamma(x)}{\theta(x)}-\frac{1}{\alpha(x)}
        = \frac{1-\alpha(x)+\hfrak}{\theta(x)}-\frac{1}{\alpha(x)}
        =(1-\alpha(x))\left(\frac{1}{\theta(x)}-\frac{1}{\alpha(x)}\right)-1+\frac{\hfrak}{\theta(x)}.
    \end{align*}
    By the definition of $\theta(x)$ we have
    \begin{align*}
        0
        > \frac{1}{\theta(x)}-\frac{1}{\alpha(x)}
        \geq - \frac{\mathfrak{m}}{\alpha_{\min}^2}.
    \end{align*}
    Since $\theta(x)\leq 2$, we get
    \begin{align*}
        \frac{\gamma(x)}{\theta(x)}-\frac{1}{\alpha(x)}
        &\geq - 1+ \frac{\hfrak}{2} + (1-\alpha(x))\left(\frac{1}{\theta(x)}-\frac{1}{\alpha(x)}\right) \\
        &\geq
        \begin{cases}
        - 1+ \frac{\hfrak}{2}, & \text{if\ \ } \alpha(x)\in (1,2),\\[10pt]
        - 1+ \frac{\hfrak}{2}- \frac{\mathfrak{m}}{\alpha_{\min}^2}, &  \text{if\ \ } \alpha(x)\in (0,1].
        \end{cases}
    \end{align*}
    If we pick $\mathfrak{m} \leq \frac 14 \hfrak \alpha_{\min}^2$, we finally see that
    $
        \frac{\gamma(x)}{\theta(x)}-\frac{1}{\alpha(x)}\geq -1 +  \frac{\hfrak}{4}
    $,
    and we get~\eqref{appA-e16}.

    In order to show~\eqref{appA-e18} we use the fact that $\supp\phi_s$ is compact. We have
    \begin{align*}
        \partial_{x_i} B_t(x)
        &= \int_{\rd} b_t (y)\partial_{x_i}\left(\frac{1}{t^{d/\theta(x)}}  \phi_1\left(\frac{y-x}{t^{1/\theta(x)}}\right)\right) dy \\
        &= \int_{\rd}(b_t (y)-b_t(x))\partial_{x_i}\left(\frac{1}{t^{d/\theta(x)}} \phi_1\left(\frac{y-x}{t^{1/\theta(x)}}\right)\right) dy.
    \end{align*}
    The derivative inside the integral leads to the following three terms
    \begin{align*}
       \mathrm{I}_1 &:= \log (t^d) \, \frac{\partial_{x_i}\theta(x)}{\theta^2(x)} \int_{\rd} (b_t (y)-b_t(x)) \,\frac{1}{t^{d/\theta(x)}} \phi_1\left(\frac{y-x}{t^{1/\theta(x)}}\right) dy\\
       \mathrm{I}_2 &:= -\int_{\rd} (b_t (y)-b_t(x)) \,\frac{1}{t^{(d+1)/\theta(x)}} \left(\partial_{x_i} \phi_1\right)\left( \frac{y-x}{t^{1/\theta(x)}}\right) dy\\
       \mathrm{I}_3 &:= \log t \frac{\partial_{x_i}\theta(x)}{\theta^2(x)}  \int_\rd ( b_t (y)-b_t(x))\left[\sum_{j=1}^d \frac{(y_j-x_j)}{t^{1/\theta(x)}} \frac{1}{t^{d/\theta(x)}} \left(\partial_{x_j} \phi_1\right)\left( \frac{y-x}{t^{1/\theta(x)}}\right)\right] dy
    \end{align*}
    which we estimate separately.

    Because of the definition of $\phi_1$, we integrate over the ball  $\{y: |x-y|\leq t^{1/\theta(x)}\}$.
    As in the first half of the proof, we combine the estimate for $|b_t(x)-b_t(y)|$ from~\eqref{B1} with~\eqref{appA-e20} to get
    \begin{align*}
    |\mathrm{I}_1|
    &\leq C |\log t| \int_{\rd} \left(t^{\gamma(x)/\theta(x)} + t^{\delta(x)+ \epsilon/\theta(x)}\right) \frac{1}{t^{d/\theta(x)}}  \phi_1\left(\frac{y-x}{t^{1/\theta(x)}}\right) dy\\
    &\leq  C'|\log t|  \left(t^{\gamma(x)/\theta(x)} +  t^{\delta(x)+  \epsilon/2}\right)\\
    &\leq  C't^{-1/\alpha(x)}\left(t^{\gamma(x)/\theta(x)} +  t^{\delta(x)+  \epsilon/2}\right).
    \end{align*}
    Almost the same calculation for $\mathrm{I}_2$ yields
    \begin{align*}
    |\mathrm{I}_2|
    &\leq C t^{-1/ \theta(x)} \left(t^{\gamma(x)/\theta(x)} +  t^{\delta(x)+ \epsilon/4}\right) \\
    &\leq C t^{-1/ \alpha(x)} \left(t^{\gamma(x)/\theta(x)} +  t^{\delta(x)+ \epsilon/4}\right).
    \end{align*}
    For the second estimate we use that $\alpha(x)\leq \theta(x)$, cf.~\eqref{pfti-e18}.
    In order to deal with $\mathrm{I}_3$ we note that $\left|\frac{(y_j-x_j)}{t^{1/\theta(x)}}\right|\leq 1$. This means that we can estimate $\mathrm{I}_3$ in the same way as $\mathrm{I}_1$, and we obtain
    \begin{align*}
        |\mathrm{I}_3|
        &\leq  C |\log t| \left(t^{\gamma(x)/\theta(x)} +  t^{\delta(x)+ \epsilon/4}\right)\\
        &\leq  C t^{-1/\alpha(x)}\left(t^{\gamma(x)/\theta(x)} +  t^{\delta(x)+ \epsilon/4}\right).
    \end{align*}
    We can now use the arguments from the first part of the proof to get~\eqref{appA-e18}.
\end{proof}

In order to relate the flows $\chi_s(x)$, $\kappa_s(y)$  with each other,  we introduce an auxiliary family $\chi_s^t(x)$ which is the solution to the following Cauchy problem
\begin{align}\label{appA-e22}
    \frac{d}{ds}\chi_s^t(x)=B_{t-s}(\chi_s^t(x)), \quad s\in [0, t], \quad \chi_0^t(x)=x.
\end{align}
The results below  generalize the results from~\cite[Prop.~A4, A5, Cor.~A1]{Ku} in two aspects: (i) we consider the multi-dimensional case $d\geq 1$ and (ii) variable-order exponents $0<\alpha(x)<2$.
\begin{proposition}\label{appA-15}
    There exists a constant $C>0$ such that for all $x,y\in\rd$ and $0<s\leq t\leq 1$
    \begin{align}\label{appA-e24}
        e^{-Ct^{\epsilon_{B}}}|\kappa_{t}(y)-x|
        \leq  |\kappa_{t-s}(y)-\chi_s^t(x)|
        \leq e^{Ct^{\epsilon_B}}|\kappa_{t}(y)-x|, \quad 0<s\leq t\leq T.
    \end{align}
\end{proposition}
A particularly interesting case of~\eqref{appA-e24} is $s=t$. This yields that $|y-\chi_t^t(x)|$ and $|\kappa_{t}(y)-x|$ are comparable.
\begin{proof}
In this proof we write $\widetilde{x}_s:=\chi^t_s(x)$ and $y_s:=\kappa_{t-s}(y)$. We have
\begin{align*}
    \frac{d}{ds}(\widetilde{x}_s-y_s)
    = (\widetilde{x}_s-y_s)q_{t,s},
    \quad
    q_{t,s} = \frac{B_{t-s}(\widetilde{x}_s)-B_{t-s}(y_s)}{\widetilde{x}_s-y_s}
\end{align*}
with the convention $\frac{0}{0}:=1$. Observe that
\begin{align*}
    \chi^t_s(x)-\kappa_{t-s}(y)
    = \widetilde{x}_s-y_s
    &= (x_0-y_0)\exp\left(\int_0^s q_{t,r}\, dr\right)\\
    &= (x-\kappa_t(y))\exp\left(\int_0^s q_{t,r}\, dr\right).
\end{align*}
Since we have $|q_{t,r}| \leq \Lip(B_{t-r}) \leq C (t-r)^{-1+\epsilon_B}$, see~\eqref{appA-e18}, the claim follows.
\end{proof}

We set  for $0\leq s < t\leq 1$
\begin{align}\label{appA-e26}
    W(t,s,x)
    := t^{-1/\alpha(x)}\int_{s^{1/\alpha(x)}< |u|\leq t^{1/\alpha(x)}} u \, \mu(x,du).
\end{align}
A direct calculation using the spherical decomposition~\eqref{set-e06} of the stable-like kernel $\mu(x,dy)$ shows
\begin{align}\label{appA-e28}
    t^{1/\alpha(x)} W(t,s,x)=\frac{\upsilon(x)}{\alpha(x)}\int_s^t r^{1/\alpha(x)-2}\,dr,
\end{align}
 see \eqref{set-e20} for the definition of $\upsilon(x)$. Indeed,
\begin{align*}
    t^{1/\alpha(x)} W(t,s,x)&=\int_{s^{1/\alpha(x)}< |u|\leq t^{1/\alpha(x)}} u \, \mu(x,du)\\
    &=\int_{s^{1/\alpha(x)}}^{t^{1/\alpha(x)}} \int_{\Sph^{d-1}} (r\ell) r^{-1-\alpha(x)} \,\sigma(x,d\ell) \,dr\\
    &=\upsilon(x)\int_{s^{1/\alpha(x)}}^{t^{1/\alpha(x)}}  r^{-\alpha(x)}  \,dr\\
    &=\left\{
          \begin{alignedat}{2}
            &\upsilon(x)(1-\alpha(x))^{-1}(t^{1/\alpha(x)-1}-s^{1/\alpha(x)-1}), &\quad& \alpha(x)\not=1; \\[\medskipamount]
            &\upsilon(x)(\log t-\log s), &\quad& \alpha(x)=1;
          \end{alignedat}
        \right.\\
    &=\frac{\upsilon(x)}{\alpha(x)}\int_s^t r^{1/\alpha(x)-2}\,dr.
\end{align*}
This yields, in particular, for all $t\in (0,1]$
\begin{align}
\label{appA-e30}
\begin{aligned}
    \int_0^t W(t,s,x)\,ds&=t^{-1/\alpha(x)}  \frac{\upsilon(x)}{\alpha(x)}\int_0^t\int_s^t r^{1/\alpha(x)-2}\,dr\, ds\\
    &= t^{-1/\alpha(x)}  \frac{\upsilon(x)}{\alpha(x)}\int_0^t r^{1/\alpha(x)-1}\,dr\\
    &= \upsilon(x).
\end{aligned}
\end{align}

\begin{proposition}\label{appA-17}
For any $\delta< \frac 12 \min\left\{\epsilon_\nu \alpha_{\max}^{-1}, \epsilon_B\right\}$ there exist a constant $C>0$ such that for all $0<s\leq t\leq 1$
\begin{align}\label{appA-e32}
    \chi_s^t(x)
    = \chi_s(x)+t^{1/\alpha(x)}\int_0^s \left[W(t, r, \chi_r(x))-W(t, t-r, \chi_r(x))\right] dr + Q_{s,t}(x),
\end{align}
with an error term satisfying
\begin{align}\label{appA-e34}
    |Q_{s,t}(x)|
    \leq Ct^{1/\alpha(x)+\delta/2}.
\end{align}
\end{proposition}

\begin{proof}
Fix $\delta< \frac 12 \min\left\{\epsilon_\nu \alpha_{\max}^{-1}, \epsilon_B\right\}$, $x\in \rd$  and write $\widetilde{x}_s:=\chi^t_s(x)$, $x_s:=\chi_s(x)$. We have
\begin{align*}
    \frac{d}{ds}(\widetilde{x}_s-x_s)
    = (\widetilde{x}_s-x_s)\widetilde{q}_{t,s} + \widetilde{Q}_{t,s}
\intertext{where}
    \widetilde{q}_{t,s} := \frac{B_{t-s}(\widetilde{x}_s)-B_{t-s}(x_s)}{\widetilde{x}_s-x_s}
    \quad\text{and}\quad \widetilde{Q}_{t,s} := B_{t-s}(x_s)-B_{s}(x_s).
\end{align*}
The solution to this ODE is given by
\begin{align*}
    \widetilde{x}_s-x_s
    &= \int_0^s \widetilde{Q}_{t,r} \exp\left(\int_r^s \widetilde{q}_{t,w}\, dw\right) dr
    =: \mathrm{I}_1+\mathrm{I}_2+\mathrm{I}_3,\\
    \mathrm{I}_1
    &= \int_0^s\left(\widetilde{Q}_{t,r} -  \left(b_{t-r}(x_r)-b_{r}(x_r)\right)\right)  \exp\left(\int_r^s \widetilde{q}_{t,w}\, dw\right) dr\\
    \mathrm{I}_2
    &= \int_0^s   (b_{t-r}(x_r)-b_{r}(x_r))\, dr \\
    \mathrm{I}_3
    &= \int_0^s   (b_{t-r}(x_r)-b_{r}(x_r))  \left(\exp\left(\int_r^s \widetilde{q}_{t,w}\, dw\right)-1\right) dr.
\end{align*}
We estimate the expressions $\mathrm{I}_1$, $\mathrm{I}_2$, $\mathrm{I}_3$ separately. Since $\delta < \frac 12 \epsilon_B$ we see because of~\eqref{appA-e18}
\begin{align}\label{appA-e36}
    \left|\exp\left(\int_r^s \widetilde{q}_{t,w}\, dw\right)-1\right|
    \leq Ct^{\epsilon_B}
    \leq C t^\delta.
\end{align}
Moreover, by~\eqref{appA-e16},
\begin{align*}
    \left|\widetilde{Q}_{t,r}- (b_{t-r}(x_r)-b_{r}(x_r))\right|
    &\leq \left|B_{t-r}(x_r) - b_{t-r}(x_r)\right| + \left|B_r(x_r) - b_r(x_r)\right| \\
    &\leq C \left((t-r)^{-1+1/\alpha(x_r)+ \epsilon_B}  +  r^{-1+1/\alpha(x_r)+ \epsilon_B}\right).
\end{align*}
The H\"older continuity of $\alpha(x)$, its boundedness from above and below, and the differential equation for $x_r=
\chi_r(x)$ combined with~\eqref{appA-e12} give
\begin{align*}
    \left|\frac{1}{\alpha(x_r)} - \frac{1}{\alpha(x)}\right|
    = \left|\frac{\alpha(x_r) - \alpha(x)}{\alpha(x_r)\alpha(x)}\right|
    \leq C |x_r -x|^\eta
    \leq C' t^{\frac 12\eta}.
\end{align*}
Now we can use Corollary~\ref{appA-07} and get
\begin{align*}
    \left|\widetilde{Q}_{t,r} - (b_{t-r}(x_r)-b_{r}(x_r))\right|
    &\leq C \left((t-r)^{-1+1/\alpha(x) +\epsilon_B}+r^{-1+1/\alpha(x) +\epsilon_B}\right).
\end{align*}
Since $\delta<\frac 12\epsilon_B$, we see
\begin{align*}
    |\mathrm{I}_1|
    &\leq C \int_0^s  C \left((t-r)^{-1+1/\alpha(x) +\epsilon_B}+r^{-1+1/\alpha(x) +\epsilon_B}\right) dr\\
    &\leq C t^{1/\alpha(x) + \epsilon_B}
    \leq C t^{1/\alpha(x) + \delta}.
\end{align*}

We will now estimate $\mathrm{I}_2$.  In view of the definition~\eqref{set-e12} of $b_t$ we have
\begin{align*}
    b_{t-r}(x)-b_{r}(x)
    =\left(m_r^\mu(x)-m^\mu_{t-r}(x)\right) + \left(m_{r}^{\nu}(x)-m_{t-r}^{\nu}(x)\right),
\end{align*}
where we set
\begin{align*}
    m_r^{\mu}(x)
    := \int_{r^{1/\alpha(x)}<|u|\leq 1} u\,\mu(x,du)
    \quad\text{and}\quad
    m_r^{\nu}(x)
    := \int_{r^{1/\alpha(x)}<|u|\leq 1} u\,\nu(x,du).
\end{align*}
By~\eqref{appA-e06} and~\eqref{N1}
\begin{align*}
    |m_{r}^{\nu}(x)|
    \leq \int_{r^{1/\alpha(x)} <|u|\leq 1}|u| \,|\nu|(x,du)
    \leq
    \begin{cases}
    \displaystyle
    C r^{\frac{1-\beta(x)}{\alpha(x)}}, &\text{if\ \ } \beta(x)\neq 1,\\[10pt]
    \displaystyle
    C |\log r|, &\text{if\ \ } \beta(x)= 1.
    \end{cases}
\end{align*}
If $\beta(x)\neq 1$, the above estimate for $m_r^\nu$ and $m_{t-r}^\nu$ combined with~\eqref{N1} yields
\begin{align*}
    |m_{r}^{\nu}(x)-m_{t-r}^{\nu}(x)|
    &\leq C\left(r^{\frac{1-\beta(x)}{\alpha(x)}} +(t-r)^{\frac{1-\beta(x)}{\alpha(x)}}\right) \\
    &\leq C\left(r^{-1 + \frac{\epsilon_\nu +1}{\alpha(x)}} +(t-r)^{-1 + \frac{\epsilon_\nu +1}{\alpha(x)}}\right).
\end{align*}
If $\beta(x)=1$, we have by~\eqref{N1} $\alpha(x)\geq 1+ \epsilon_\nu$, hence
\begin{align*}
    |m_{r}^{\nu}(x)-m_{t-r}^{\nu}(x)|&   \leq  C  ( | \log r| +|\log (t-r)|) \\
    &\leq C  \left(r^{-1 + \frac{1+\epsilon_\nu/2}{\alpha(x)}} + (t-r)^{-1 + \frac{1+\epsilon_\nu/2}{\alpha(x)}}\right).
\end{align*}
Since $\delta < \frac{\epsilon_\nu}{2 \alpha_{\max}}$,
\begin{align}\label{appA-e38}
    |m_{r}^{\nu}(x)-m_{t-r}^{\nu}(x)|
    \leq C  \left(r^{-1 + \frac 1{\alpha(x)}+\delta}   + (t-r)^{-1 + \frac 1{\alpha(x)} + \delta}\right).
\end{align}
For $m_r^{\mu}$, $r\leq t$, we have
\begin{align*}
    m_r^{\mu}(x)
    = m_t^{\mu}(x) + \int_{r^{1/\alpha(x)}<|u|\leq t^{1/\alpha(x)}}u\,\mu(x,du)
    = m_t^{\mu}(x) + t^{\frac 1{\alpha(x)}}W(t,r,x).
\end{align*}
This shows that
\begin{align*}
    m_r^{\mu}(x) - m_{t-r}^{\mu}(x)
    = t^{\frac 1{\alpha(x)}}W(t,r,x) - t^{\frac 1{\alpha(x)}}W(t,t-r,x),
\end{align*}
and so
\begin{align*}
    \mathrm{I}_2
    = t^{\frac 1{\alpha(x)}} \int_0^s \left(W(t,r,x_r) - W(t,t-r,x_r)\right) dr
        + \int_0^s \left(m_{r}^{\nu}(x_r)-m_{t-r}^{\nu}(x_r)\right) dr.
\end{align*}
Using Corollary~\ref{appA-07}  and~\eqref{appA-e38} we get
\begin{align}\label{appA-e40}
    \left|\int_0^s \left(m_{r}^{\nu}(x_r)-m_{t-r}^{\nu}(x_r)\right) dr\right|
    \leq C t^{\frac 1{\alpha(x)} + \delta}.
\end{align}

In order to estimate $\mathrm{I}_3$ we observe that
\begin{align*}
    m_r^{\mu}(x)
    =
    \begin{cases}
    \displaystyle
    \frac{\alpha(x)}{\alpha(x)-1} [1-   r^{1/\alpha(x)-1}], &\text{if\ \ } \alpha(x)\neq 1, \\[10pt]
    \displaystyle
    |\log r|,  &\text{if\ \ } \alpha(x)=1.
     \end{cases}
\end{align*}
Note that $1- r^{1/\alpha(x)-1}$ is of order $1-\alpha(x)$ if $\alpha(x)\to 1$. Since $|\log r| \leq C r^{-\delta/2}$, $t\in (0,1]$,
we get
\begin{align*}
    \left|m_r^{\mu}(x) - m_{t-r}^{\mu}(x)\right|
    \leq  C \left(r^{1/\alpha(x)-1-\delta/2}  + (t-r)^{1/\alpha(x)-1-\delta/2}\right),
    \quad 0<r\leq t\leq 1.
\end{align*}
If we combine this estimate with Corollary~\ref{appA-07}, we get
\begin{align}\label{appA-e42}
    \int_0^t \left|m_r^{\mu}(x_r) - m_{t-r}^{\mu}(x_r)\right|
    \leq C t^{1/\alpha(x) - \delta/2},
\end{align}
and~\eqref{appA-e36}, \eqref{appA-e40} and~\eqref{appA-e42} finally give
\begin{align*}
    |\mathrm{I}_3|
    \leq C  t^{\delta}  \left(t^{ 1/\alpha(x)-\delta/2} + t^{1/\alpha(x) + \delta}\right)
    \leq C t^{1/\alpha(x) +\delta/2}.
\end{align*}

These estimates show that~\eqref{appA-e32} and~\eqref{appA-e34} hold true with
\begin{align*}
    Q_{s,t}(x)
    &:= \mathrm{I}_1 + \int_0^s  \left(m_{r}^{\nu}(x_r)-m_{t-r}^{\nu}(x_r)\right) dr + \mathrm{I}_3.
\qedhere
\end{align*}
\end{proof}

In the next proposition we estimate the integral term in appearing in~\eqref{appA-e32}. Recall that $\eta$ is the H\"older index from~\eqref{M2}.
\begin{proposition}\label{appA-19}
    For $0 < \epsilon_\kappa < \frac 14 \min\left\{\epsilon_\nu \alpha_{\max}^{-1}, \epsilon_B, \eta\right\}$  there is a constant $C>0$ such that for all $x\in\rd$ and $t\in (0,1]$
    \begin{align}\label{appA-e44}
        |\chi_t^t(x)-\chi_t(x)|
        \leq C t^{1/\alpha(x)+\epsilon_\kappa}.
    \end{align}
\end{proposition}

\begin{proof}
 Fix $\epsilon_\kappa < \frac 14 \min\left\{\epsilon_\nu \alpha_{\max}^{-1}, \epsilon_B, \eta\right\}$.  Our starting point is the formula~\eqref{appA-e32}, where we set $s=t$. Since the estimate~\eqref{appA-e34} for $Q_{t,t}(x)$ is already of the correct form, we only have to estimate the integral term. Recall that $x_s = \chi_s(x)$. We write
\begin{align*}
    \mathrm{I}
    := \int_0^t \left[W(t,s,x_s) - W(t,t-s,x_s )\right] ds
    =  \int_0^t \left[W(t,s,x_s) - W(t,s,x_{t-s})\right] ds
\end{align*}
for the integral term appearing in~\eqref{appA-e32}. In order to estimate this integral, we first analyze the integrand.

From  \eqref{appA-e28}  we get
\begin{align*}
    &|W(t,s,x_s)- W(t,s,x_{t-s})| \\
    &\leq \left|\frac{\upsilon(x_s)}{\alpha(x_s)} - \frac{\upsilon(x_{t-s})}{\alpha(x_{t-s})}\right|
        t^{-\frac{1}{\alpha(x_s)}}\int_s^t r^{\frac{1}{\alpha(x_s)}-2}\,dr \\
    &\qquad \mbox{}+\frac{\upsilon(x_{t-s})}{\alpha(x_{t-s})} \left|1 -t^{\frac{1}{\alpha(x_s)}-\frac{1}{\alpha(x_{t-s})}}\right|
        t^{-\frac{1}{\alpha(x_s)}}  \int_s^t r^{\frac{1}{\alpha(x_s)}-2}\,dr\\
    &\qquad \mbox{}+\frac{\upsilon(x_{t-s})}{\alpha(x_{t-s})} t^{-\frac{1}{\alpha(x_{t-s})}}
        \left|\int_s^t r^{\frac{1}{\alpha(x_{t-s})}-2} \left( r^{\frac{1}{\alpha(x_s)}-\frac{1}{\alpha(x_{t-s})}}- 1\right) dr\right|\\
    &= \left(\left|\frac{\upsilon(x_s)}{\alpha(x_s)}- \frac{\upsilon(x_{t-s})}{\alpha(x_{t-s})}\right|+ \frac{\upsilon(x_{t-s})}{\alpha(x_{t-s})} \left|1 -t^{\frac{1}{\alpha(x_s)}-\frac{1}{\alpha(x_{t-s})}}\right|\right)
        t^{-\frac{1}{\alpha(x_s)}}\int_s^t r^{\frac{1}{\alpha(x_s)}-2}\,dr\\
    &\qquad \mbox{}+\frac{\upsilon(x_{t-s})}{\alpha(x_{t-s})} t^{-\frac{1}{\alpha(x_{t-s})}}
        \left|\int_s^t r^{\frac{1}{\alpha(x_{t-s})}-2} \left( r^{\frac{1}{\alpha(x_s)}-\frac{1}{\alpha(x_{t-s})}}- 1\right) dr\right|.
\end{align*}
Consider first the coefficients depending on $\upsilon$ and $\alpha$ in front of  the integrals.  The coefficient appearing before the second integral is bounded. Since the functions $\upsilon(\cdot)$ and $\alpha(\cdot)$ are $\eta$-H\"older continuous, and since $1/\alpha(\cdot)$ inherits its $\eta$-H\"older continuity from $\alpha(\cdot)$, we have
\begin{align*}
    &\left|\frac{\upsilon(x_s)}{\alpha(x_s)}- \frac{\upsilon(x_{t-s})}{\alpha(x_{t-s})}\right|
    + \frac{\upsilon(x_{t-s})}{\alpha(x_{t-s})} \left|1 -t^{1/\alpha(x_s)-1/\alpha(x_{t-s})}\right|   \\
    &\qquad\leq C \left(|x_s-x_{t-s}|^\eta\wedge 1\right) \left(1+ t^{-|1/\alpha(x_s)-1/\alpha(x_{t-s})|}   |\log t|\right).
\end{align*}
in the last line we use the inequality $|1-e^x|\leq |x|e^{|x|}$. Because of~\eqref{set-e18} and~\eqref{appA-e12}we have
\begin{align*}
    \left|x_s -x_{t-s}\right|\leq C \left|\int_s^{t-s}r^{-\frac 12}dr\right|
    \leq C t^{\frac 12},
    \quad s\leq t.
\end{align*}
Therefore, we may use Lemma~\ref{appA-05} to get  $t^{-|1/\alpha(x_s)-1/\alpha(x_{t-s})|}  \leq C$.  Combining the above inequalities yields
\begin{align*}
    \left|\frac{\upsilon(x_s)}{\alpha(x_s)}- \frac{\upsilon(x_{t-s})}{\alpha(x_{t-s})}\right|
    + \frac{\upsilon(x_{t-s})}{\alpha(x_{t-s})} \left|1 -t^{1/\alpha(x_s)-1/\alpha(x_{t-s})}\right|
    \leq
    C t^{\frac 12\eta}  \left(1+  |\log t|\right).
\end{align*}
A similar estimate applies to the integral in the second term:
\begin{align*}
    \int_s^t &r^{1/\alpha(x_{t-s})-2} \left|r^{1/\alpha(x_s)-1/\alpha(x_{t-s})}- 1\right| dr \\
    &\leq  C (|x_s-x_{t-s}|^\eta\wedge 1)
        \int_s^t r^{1/\alpha(x_{t-s})-2}  r^{- |1/\alpha(x_s)-1/\alpha(x_{t-s})|}|\log r| \, dr\\
    &\leq C t^{\frac 12\eta}  \int_s^t r^{1/\alpha(x_{t-s})-2} |\log r| \, dr.
\end{align*}
For every  $\omega < \min\left\{\frac 12\eta - \epsilon_\kappa, \alpha_{\max}^{-1}\right\}$ we get
\begin{align*}
    &|W(t,s,x_s)- W(t,s,x_{t-s})|\\
    &\qquad\leq C t^{\frac 12\eta} \left(t^{-1/\alpha(x_s)} \int_s^t r^{1/\alpha(x_s)-2}\, dr
    + t^{-1/\alpha(x_{t-s})}\int_s^t r^{1/\alpha(x_{t-s})-2} |\log r| \, dr\right)\\
    &\qquad\leq C t^{\frac 12\eta-\omega} \left(t^{-1/\alpha(x_s)} \int_s^t r^{1/\alpha(x_s)-2} \, dr
    + t^{-1/\alpha(x_{t-s})+\omega}\int_s^t r^{1/\alpha(x_{t-s})-\omega-2} \,dr\right).
\end{align*}
Using the elementary inequality
\begin{align*}
    \int_s^t r^{\theta-2}dr =\int_s^t r^{\theta-1}\frac{dr}{r}
    \leq
    (t^{\theta-1}+ s^{\theta-1}) (\log t-\log s),
    \quad 0<s<t<1,
\end{align*}
we get for $\theta>\theta_0 :=  \alpha_{\max}^{-1}-\omega$ and  $s\in (0,t]$
\begin{align*}
    t^{-\theta} \int_s^t r^{\theta-2}\,dr
    &\leq
    (t^{-1}+ s^{-1}(s/t)^\theta) (\log t-\log s)\\
    &\leq
    (t^{-1}+ s^{-1}(s/t)^{\theta_0}) (\log t-\log s).
\end{align*}
If we combine the last three inequalities, we arrive at
\begin{align*}
    |W(t,s,x_s)- W(t,s,x_{t-s})|
    \leq
    C t^{\frac 12\eta - \omega}  (t^{-1}+ s^{-1}(s/t)^{\theta_0}) (\log t-\log s).
\end{align*}
Since $\int_0^t s^{\theta-1}  \log s \,ds = \theta^{-1} t^\theta\log t - \theta^{-2} t^\theta$  and $\omega < \min\left\{\frac 12\eta - \epsilon_\kappa, \alpha_{\max}^{-1}\right\}$,  we get
\begin{align*}
    &\int_0^t |W(t,s,x_s)- W(t,s,x_{t-s})| \,ds\\
    &\qquad\leq C t^{\frac 12\eta-\omega}  \left(\frac{t \log t - t \log t + t}{t} +  t^{-\theta_0} \frac{t^{\theta_0} \log t}{\theta_0}-  t^{-\theta_0} \left( \frac{t^{\theta_0}\log t}{\theta_0}- \frac{t^{\theta_0}}{\theta_0^2}\right)\right) \\
    &\qquad= C t^{\frac 12\eta-\omega}  \left(1+ \theta_0^{-2}\right)
   \leq C' t^{\epsilon_\kappa}.
    \qedhere
\end{align*}
\end{proof}

The next result allows us to switch in estimates between $|\kappa_t(y)-x|$ and $|y - \chi_t(x)|$. It follows if we combine Proposition~\ref{appA-15} and~\eqref{appA-e44} from Proposition~\ref{appA-19}.
\begin{corollary}\label{appA-21}
Assume that $0 < \epsilon_\kappa < \frac 14\min\left\{\epsilon_\nu \alpha_{\max}^{-1}, \epsilon_B, \eta\right\}$. There exists a constant $C>0$ such that for all  such that for all $x,y\in\rd$ and  $0 < s \leq t \leq 1$
\begin{align*}
    e^{-Ct^{\epsilon_B}}|\kappa_{t}(y)-x| - Ct^{\frac 1{\alpha(x)} + \epsilon_\kappa}
    \leq
    |\kappa_{t-s}(y)-\chi_s(x)|
    \leq
    e^{Ct^{\epsilon_B}}|\kappa_{t}(y)-x| + Ct^{\frac 1{\alpha(x)} + \epsilon_\kappa}.
\end{align*}
\end{corollary}

\begin{remark}\label{appA-23}
It follows from Corollary~\ref{appA-21} that the set $\{|x-y|\leq t^\delta\}$ in the indicator functions of  Lemma~\ref{appA-05}  can be replaced by the sets $\{|y-\chi_t(x)|\leq t^\delta\}$ or $\{|x-\kappa_t(y)|\leq t^\delta\}$.  Namely, under assumptions of Lemma~\ref{appA-05} we have
\begin{align}\label{appA-e46}
    e^{-C  t^{\gamma\delta/2}}  \I_{\{|y-\chi_t(x)|\leq  t^\delta\}}
    &\leq t^{w(x)-w(y)}
    \leq  e^{Ct^{\gamma\delta/2}}  \I_{\{|y-\chi_t(x)|\leq t^\delta\}} + t^{-w_{\max}}  \I_{\{|y-\chi_t(x)|>  t^\delta\}},
\\\label{appA-e48}
    e^{-C  t^{\gamma\delta/2}}  \I_{\{|x-\kappa_t(y)|\leq  t^\delta\}}
    &\leq t^{w(x)-w(y)}\leq  e^{Ct^{\gamma\delta/2}}  \I_{\{|x-\kappa_t(y)|\leq t^\delta\}} + t^{-w_{\max}}  \I_{\{|x-\kappa_t(y)|>  t^\delta\}}.
\end{align}
 Indeed, using the triangle inequality, the definitions of the flows $\chi_t$ and $\kappa_t$  and
\begin{align*}
    |y-\kappa_t(y)|
    \leq C t^{\frac 12},
    \quad
    |x-\chi_t(x)|
    \leq C t^{\frac 12},
\end{align*}
(cf.\ Remark~\ref{appA-11}), we see that if
\begin{align*}
 |y-x|\leq C_1 t^\delta,
\end{align*}
then there exists constants $C_1,\, C_3>0$  such that
\begin{align}\label{appA-e50}
    |y-\chi_t(x)|\leq C_2 t^\delta, \quad
    |x-\chi_t(y)|\leq C_3 t^\delta.
\end{align}
On the other hand,  each of the inequalities~\eqref{appA-e50} implies that for some $C_1'>0$ we have   $|y-x|\leq C_1' t^\delta$.

\end{remark}

\section{Properties of the kernels $K^{0;c}$ and $K^{1;c}$}\label{appB}

In order to show that the parametrix expansion is convergent, we have introduced in~\eqref{pfti-e40}, \eqref{pfti-e42} and~\eqref{pfti-e44} the kernels $f_{t,a,c}(x)$, $K^{0;c}(x,y)$ and $K^{1;c}(x,y)$,
\begin{align*}
    f_{t,a,c}(x)   &= t^{-ad} e^{-c|x|t^{-a}},\\
    K^{0;c}_t(x,y) &= f_{t,\zeta(y),c}(\kappa_t(y)-x) = t^{-\zeta(y) d} e^{-c|\kappa_t(y)-x|t^{-\zeta(y)}},\\
    K^{1;c}_t(x,y) &= \I_{\{|y-\chi_t(x)|\leq  t^\delta\}}  f_{t,\zeta(x),c}( y-\chi_t(x))  +  t^{-N} \I_{\{|y-\chi_t(x)|> t^\delta\}} f_{t,\zeta_{\min},c} (y-\chi_t(x)),
\end{align*}
where $0 <\delta < \frac 12\alpha_{\max}^{-1}$ and $N$ is some sufficiently large parameter. In this appendix we prove some basic estimates for these kernels  used in the main part of our paper.

\begin{proposition}\label{appB-03}
   Let $w(x): \rd\to (0,\infty)$ be a bounded $\gamma$-H\"older continuous function, and $v: \rd\to (0,\infty)$ be such that $0<v_{\min}\leq v(x)\leq v_{\max}<\infty$.  For any $0<c'<c<\infty$ there exists a constant $C>0$ such that for all $x,y\in \rd$, $t\in (0,1]$,
  \begin{align}
   \label{appB-e02}
   t^{v(y)(w(x)-w(y))} K^{0;c}_t(x,y)
        &\leq C K^{1;c'}_t(x,y),\\
   \label{appB-e04}  t^{v(y)(w(x)-w(y))}  p^{0}_t(x,y)
      &\leq  C\left( p^{0}_t(x,y) +  K^{1;c'}_t(x,y)\right);
   \end{align}
   we use  $N:= v_{\max}w_{\max}+ d/\alpha_{\min}$ in the definition of $K^{1;c'}_t(x,y)$.
\end{proposition}

\begin{proof}
We begin with~\eqref{appB-e02}. Consider first the case  $|y-\chi_t(x)|>  t^\delta$.   We have  $\delta< \frac 12 \alpha_{\max}^{-1}$ and $\zeta_{\min}=\alpha_{\max}^{-1} -\sfrak$; recall also that $\mathfrak{s}< \frac 12 \alpha_{\max}^{-1},$ see Section~\ref{zero}. Then $\delta-\zeta_{\min}<0$, and thus for any  $0<c'<c$ and $\gamma = (c-c')/c$
\begin{align*}
    e^{- c(\frac{c-c'}{c} t^\delta  - t^{1/\alpha_{\max}})t^{-\zeta_{\min}}}
    =e^{-(c-c') t^{\delta-\zeta_{\min}}} e^{c t^{\mathfrak{s}}}
    \leq C.
\end{align*}
By Corollary~\ref{appA-21}  we have  for any $c'<c$
\begin{align}\label{appB-e06}
\begin{aligned}
    &t^{v(y)(w(x)-w(y))} K_t^{0;c}(x,y)\\
    &\quad\leq t^{-v(y) w(y)}f_{t,\zeta(y),c}(\kappa_t(y)-x) \\
    &\quad\leq C t^{-v_{\max} w_{\max} -d \zeta(y)} e^{-c (|y-\chi_t(x)|- t^{1/\alpha_{\max}})t^{-\zeta(y)}} \\
    &\quad\leq C t^{-v_{\max} w_{\max}-d \zeta(y)} e^{-  c' |y-\chi_t(x)|t^{-\zeta_{\min}}} e^{- c( \frac{(c-c')}{c}  t^\delta  - t^{1/\alpha_{\max}})t^{-\zeta_{\min}}}  \\
    &\quad\leq  C' t^{-v_{\max} w_{\max}-d \zeta_{\max}}  e^{-  c'|y-\chi_t(x)|t^{-\zeta_{\min}}}\\
    &\quad\leq  C' t^{-N} e^{-\frac c2 |y-\chi_t(x)|t^{-\zeta_{\min}}}\\
    &\quad\leq C'K^{1;c'}_t(x,y),
\end{aligned}
\end{align}
where $w_{\max} = \sup_{y\in\rd} |w(y)|$.  Thus, \eqref{appB-e02} holds true for $|y-\chi_t(x)|>  t^\delta$.

By  Remark~\ref{appA-23} and Corollary~\ref{appA-21} we have  for $|y-\chi_t(x)|\leq  t^\delta$ and $0<c'<c^*<c$
\begin{align*}
    f_{t,\zeta(y),c}(\kappa_t(y)-x)
    \leq C  f_{t,\zeta(x),c^*} (\kappa_t(y)-x)
    \leq C' f_{t,\zeta(x),c'} (y-\chi_t(x))
    \leq  C'K^{1;c'}_t(x,y),
\intertext{which implies that}
    t^{v(y) (w(x)-w(y))} K_t^{0;c}(x,y)
    \leq  C'K^{1;c'}_t(x,y).
\end{align*}
Together, these estimates prove~\eqref{appB-e02}.

\medskip
Let us prove~\eqref{appB-e04}. By Remark~\ref{appA-23} we have $t^{w(x)-w(y)}\asymp 1$ if $|y-\chi_t(x)|\leq t^\delta$; in this case the~\eqref{appB-e04} follows trivially. Consider the case $|y-\chi_t(x)|> t^\delta$. Using  the estimate~\eqref{pfti-e52} for $p_t^0(x,y)$, Remark~\ref{appA-23} and the definition of $K_t^{1;c}$ we get, with the same calculation as in~\eqref{appB-e06}  for all $0<c'<c$
\begin{align*}
    t^{v(y)(w(x)-w(y))} p_t^0(x,y)&\leq C t^{-v(y)w(y)}f_{t,\zeta(y),c}(\kappa_t(y)-x)\\
    &\leq C t^{-v_{\max} w_{\max}-\sfrak d} f_{t,\zeta(y),c}(\kappa_t(y)-x) \\
    &\leq C t^{-v_{\max} w_{\max}-d/\alpha_{\max}} e^{-  c' |y-\chi_t(x)|t^{-\zeta_{\min}}} \\
    &\leq C K^{1;c'}_t(x,y).
\qedhere
\end{align*}
\end{proof}

\begin{proposition}\label{appB-05}
    If $|u|\leq t^{1/\alpha(x)}$  and $0<c'<c$,  then  there exists some $C>0$ such that for all $x,y\in\rd$ and $t\in (0,1]$
    \begin{align*}
        K_t^{0;c}(x+u,y)\leq C K_t^{1;c'}(x,y).
    \end{align*}
    The parameter $N$ in the definition of $K_t^{1;c'}(x,y)$ satisfies  $N> {d}/{\alpha_{\min}}$.
\end{proposition}
\begin{proof}
Starting from the definition of $K_t^{0;c}(x,y)$ we distinguish between two cases.

\medskip\noindent\emph{Case 1}: $t^{1/\alpha(y)} \leq  t^{1/\alpha(x)}$, i.e.\ $\alpha(y)\leq \alpha(x)$. Since $\zeta(\cdot)=\frac{1}{\alpha(\cdot)}- \sfrak$, cf.~\eqref{pfti-e12}, we know that $t^{\zeta(y)}<t^{\zeta(x)}$, or  $t^{-\zeta(x)} < t^{-\zeta(y)}$. For  $|u|\leq t^{1/\alpha(x)}$,  $w\in \rd$, we have
\begin{align*}
    \frac{|w-u |}{t^{\zeta(y)}}
    \geq \frac{|w-u |}{t^{\zeta(x)}}
    \geq \frac{|w|- |u|}{t^{\zeta(x)}}
    \geq \frac{|w|-t^{1/\alpha(x)}}{t^{\zeta(x)}}
    \geq  \frac{|w|}{t^{\zeta(x)}}-t^{\sfrak}
    \geq \frac{|w|}{t^{\zeta(x)}}-1.
\end{align*}
The required estimate follows now from the definition of $K_t^{1;c'}(x,y)$.

\medskip\noindent\emph{Case 2}: $t^{1/\alpha(y)}> t^{1/\alpha(x)}$, i.e.\ $\alpha(y)> \alpha(x)$. Observe that
\begin{align*}
    \frac{|w-u |}{t^{\zeta(y)}}
    \geq \frac{|w|- |u|}{t^{\zeta(y)}}
    \geq \frac{|w|-t^{1/\alpha(y)}}{t^{\zeta(y)}}
    \geq \frac{|w|}{t^{\zeta(y)}}-t^{1/\alpha(y)-\zeta(y)}
    \geq \frac{|w|}{t^{\zeta(y)}}-1.
\end{align*}

If we take  $w=\kappa_t(y)-x$, we get for $|u|< t^{1/\alpha(x)}$ using the definition~\eqref{pfti-e40} of $f_{t,a,c}$
\begin{align*}
    & K_t^{0;c}(x+u,y) = f_{t,\zeta(y),c}(\kappa_t(y)-x-u) \\
    &\leq C \left[t^{d(\zeta(x)-\zeta(y))} f_{t,\zeta(x),c} (\kappa_t(y)-x)\I_{\{\alpha(y)\leq \alpha(x)\}}
        + f_{t,\zeta(y),c} (\kappa_t(y)-x)\I_{\{\alpha(y)>\alpha(x)\}}  \right].
\end{align*}
The first term on the right can be estimated by Lemma~\ref{appA-05} and Corollary~\ref{appA-21}
\begin{align*}
    t^{d(\zeta(x)-\zeta(y))} &f_{t,\zeta(x),c} (\kappa_t(y)-x)\I_{\{\alpha(y)\leq \alpha(x)\}}\\
    &\leq C t^{d(\zeta(x)-\zeta(y))} f_{t,\zeta(x),c'} (y-\chi_t(x)) \\
    &\leq  C t^{d(\zeta(x)-\zeta(y))} \I_{\{|y-\chi_t(x)|\leq t^\delta\}} f_{t,\zeta(x),c'} (y-\chi_t(x))\\
    &\qquad\qquad\quad \mbox{}+ C t^{-N}  \I_{\{|y-\chi_t(x)|>t^\delta\}} f_{t,\zeta(x),c'} (y-\chi_t(x))\\
    &= K^{1;c'}_t(x,y);
\end{align*}
we require $N> {d}/{\alpha_{\min}}$ in the definition of $K^{1;c'}_t(x,y)$.
For the second term we observe that $\alpha(y)>\alpha(x)$ and use Corollary~\ref{appA-21} to get
\begin{align*}
    \frac{|\kappa_t(y)-x|}{t^{\zeta(y)}}
    &\geq c \frac{|y-\chi_t(x)|}{t^{\zeta(y)}} - c' t^{1/\alpha(x)- \zeta(y)+\epsilon_\kappa} \\
    &= c \frac{|y-\chi_t(x)|}{t^{\zeta(y)}} - c' t^{1/\alpha(x)- 1/\alpha(y) + \sfrak +\epsilon_\kappa} \\
    &\geq c \frac{|y-\chi_t(x)|}{t^{\zeta(y)}} - c' t^{\sfrak+ \epsilon_\kappa}.
\end{align*}
From the definition of $f_{t,\zeta(y),c}(x)$ and~\eqref{pfti-e64} we get
\begin{align*}
    f_{t,\zeta(y),c} (\kappa_t(y)-x)\I_{\{\alpha(y)>\alpha(x)\}}\leq  C f_{t,\zeta(y),c'} (y-\chi_t(x))\leq C K^{1;c'}_t(x,y),
\end{align*}
with $N> d \alpha_{\min}^{-1}- d\alpha_{\max}^{-1}$, and the estimate is complete.
\end{proof}

\begin{proposition}\label{appB-07}
    Let $\delta< \zeta_{\min}$  and $c>0$. There exists a constant $C>0$ \textup{(}depending on $c$\textup{)} such that
    \begin{align}\label{appB-e08}
        \int_\rd  K^{1;c}_t(x,y)\,dy \leq C.
    \end{align}
\end{proposition}
\begin{proof}
We have
\begin{align*}
    \int_\rd  K^{1;c}_t(x,y)\,dy
    &\leq  \int_\rd  f_{t, \zeta(x),c}(y-\chi_t(x))\,dy + t^{-N} \int_{ |y-\chi_t(x)|> c t^\delta} f_{t, \zeta_{\min},c}(y-\chi_t(x))\,dy \\
    & \leq C + t^{-N- d \zeta_{\min}} \int_{|z|>1} e^{- c|z|t^{-\zeta_{\min}+\delta}}\,dz \\
    &\leq C + t^{-N- d \zeta_{\min}} e^{- \frac c2 t^{-\zeta_{\min}+\delta}} \int_{|z|>1} e^{- \frac c2 |z|}\,dz
     \leq C'.
\qedhere
\end{align*}
\end{proof}

\section{Properties of $p_t^{z,\cut}(x)$}\label{appC}
Fix $x=z$ and consider the `frozen' stable-like jump kernel $\mu(x,du)|_{x=z}$. Denote by $\psi_{t}^{z,\cut}(\xi)$ be the characteristic exponent of the additive process obtained by removing, dynamically depending on time $t$, the large jumps from the frozen kernel $\mu(x,dz)|_{x=z}$; in other words,
\begin{align}\label{psiz}
    \psi_{t}^{z,\cut}(\xi)
    = \int_{|u|\leq t^{\zeta(z)}} \left(1- e^{i \xi \cdot u} + i \xi \cdot u \I_{\{ |u|\leq (1\wedge t)^{1/\alpha(z)}\}}\right) \mu(z,du),
    \quad\xi\in\rd.
\end{align}
In this section we collect the properties of
the density expressed in terms of the (inverse) Fourier transform, cf.~\eqref{pfti-e16},
\begin{align*}
   p_t^{z,\cut}(x)
   = (2\pi)^{-d} \int_\rd e^{- i \xi x -\int_0^t \psi_r^{z,\cut}(\xi)\,dr}\,d\xi, \quad x\in \rd,\; t>0.
\end{align*}

\begin{proposition}\label{appC-03}
    There exist constants $\sigma_0>0$ and $C>0$ such that for all $z\in \rd$, $t\in (0,1]$ and $\xi\in\rd$
    \begin{align*}
        e^{-\int_0^t  \Re \psi_{r}^{z,\cut}(\xi)\,dr}
        \leq C e^{- \sigma_0 t|\xi|^{\alpha(z)}}.
    \end{align*}
\end{proposition}
\begin{proof}
We have
\begin{align*}
    \Re \psi_{t}^{z,\cut}(\xi)
    = \Re\psi^z(\xi) -\int_{|u|>t^{\zeta(z)}}  \left(1-\cos (\xi \cdot  u)\right)\mu(z, du).
\end{align*}
By the scaling property of the stable L\'evy measure $\mu(z, \cdot)$,
\begin{align*}
    \Re\psi^z(\xi)
    = \lambda(z)|\xi|^{\alpha(z)}q\left(z,\frac{\xi}{|\xi|}\right),
    \quad
    q\left(z,\ell\right) = \int_{\Sph^{d-1}} (1-\cos (\ell\cdot\ell'))\, \sigma(z,d\ell').
\end{align*}
The function $q\left(z,\cdot\right)$ is continuous and, because of~\eqref{M1}, strictly positive. Since $\lambda(z)\geq \lambda_{\min}>0$, this gives
\begin{align*}
    \Re\psi^z(\xi)
    \geq \sigma_0|\xi|^{\alpha(z)}, \quad \sigma_0>0.
\end{align*}
On the other hand,
\begin{align*}
    \int_{|v|>t^{\zeta(z)}}  \left(1-\cos(\xi \cdot u)\right) \mu(z, du)
    &\leq 2 \mu(z,\{ u:\, |u| > t^{\zeta(z)}\})\\
    &= C t^{-\zeta(z) \alpha(z)}\\
    &= Ct^{-1+\sfrak\alpha(z)}
\end{align*}
(we use $\zeta(z)\alpha(z) = 1-\sfrak\alpha(z)$, see~\eqref{pfti-e12}). Thus,
\begin{gather}\label{appC-e02}
    \int_0^t \Re  \psi_r^{z,\cut} (\xi)\, dr
    \geq t\sigma_0 |\xi|^{\alpha(z)} - \frac{Ct^{\sfrak\alpha(z)}}{\sfrak\alpha(z)}
    \geq t\sigma_0 |\xi|^{\alpha(z)} - \frac{Ct^{\sfrak\alpha_{\min}}}{\sfrak\alpha_{\min}}.
\qedhere
\end{gather}
\end{proof}

Recall that $L_x^{t,z,\cut}$ is the pseudo-differential operator with symbol $-\psi_t^{z,\cut}(\xi)$, cf.~\eqref{pfti-e32}.

\begin{proposition}\label{appC-05}
    For any fixed $z\in \rd$ the function $p_t^{z,\cut} (y-x)$ satisfies the equation
    \begin{align}\label{appC-e04}
        \frac{d}{dt} p_t^{z,\cut}(y-x)
        = L^{t,z,\cut}_x   p_t^{z,\cut}(y-x).
    \end{align}
\end{proposition}
\begin{proof}
    It is enough to check that the Fourier transforms of the left- and right-hand side of~\eqref{appC-e04} coincide. By the definition of $p^{z,\cut}_{t}(y-x)$ we have
    \begin{align}\label{appC-e06}
        \int_\rd e^{-i \xi x}  p^{z,\cut}_{t}(y-x)\,dx
        = e^{-i\xi y}  e^{-\int_0^{t}\psi_r^{z,\cut}(\xi)\,dr},
    \end{align}
    and, in view of Proposition~\ref{appC-03}, we can interchange integration and differentiation
    \begin{align*}
    \int_\rd e^{-i \xi x} \frac{d}{dt}  p^{z,\cut}_{t}(y-x)\,dx
    &=\frac{d}{dt}  \int_\rd e^{-i \xi x}   p^{z,\cut}_{t}(y-x)\,dx\\
    &= -\psi_t^{z,\cut}(\xi) e^{-i\xi y}e^{-\int_0^{t}  \psi_r^{z,\cut}(\xi)\,dr}.
    \end{align*}
    On the other hand, using~\eqref{appC-e06} and Fubini's theorem, we get
    \begin{align*}
    \int_\rd &e^{-i \xi x}  L^{t,z,\cut}_x    p^{z,\cut}_{t}(y-x)\,dx \\
    &= \int_\rd e^{-i \xi x}\int_{|u|\leq t^{\zeta(z)}} \big(p^{z,\cut}_{t}(y-x-u) -  p^{z,\cut}_{t}(y-x)\\
    &\qquad\qquad\qquad\qquad\qquad - \nabla_x  p^{z,\cut}_{t}(y-x) \cdot u  \I_{\{|u|\leq t^{1/\alpha(z)}\}} \big)\, \mu(z,du)\, dx \\
    &= \int_{|u|\leq t^{\zeta(z)}} \int_\rd e^{-i \xi x} \big(p^{z,\cut}_{t}(y-x-u) - p^{z,\cut}_{t}(y-x)\\
    &\qquad\qquad\qquad\qquad\qquad - \nabla_x p^{z,\cut}_{t}(y-x) \cdot u  \I_{\{|u|\leq t^{1/\alpha(z)}\}} \big)\,dx\, \mu(z,du) \\
    & = e^{-iy \xi} \left(\int_{|u|\leq t^{\zeta(z)}} \left(e^{i\xi u} -1 -i\xi  \cdot u  \I_{\{|u|\leq t^{1/\alpha(z)}\}}\right) \mu(z,du)\right) e^{-\int_0^{t}  \psi_r^{z,\cut}(\xi)\,dr}\\
    &= -\psi_t^{z,\cut}(\xi) e^{-i\xi y}e^{-\int_0^{t}\psi_r^{z,\cut}(\xi)\,dr},
    \end{align*}
    finishing the proof.
\end{proof}

Recall from~\eqref{pfti-e40} that $f_{t,a,c}(x) = t^{-da} e^{-c|x| t^{-a}}$ with $a,c>0$, $t>0$ and $x\in\rd$.
\begin{proposition}\label{appC-07}
    For any $k, l \in\nat\cup\{0\}$ and $c>0$ there exists some $C=C_{k,c}>0$ such that
    \begin{align}\label{appC-e08}
        \left| \partial^l _t \nabla^k_x p_t^{z,\cut}(x)\right|
        \leq C t^{-\sfrak d-k/\alpha(z)- l } f_{t,\zeta(z),c}(x),
        \quad x\in \rd, \, 0<t\leq 1.
    \end{align}
    In particular,
    \begin{align}\label{appC-e10}
        \sup_{x\in \rd} \left| \partial^l _t \nabla^k_x p_t^{z,\cut} (x)\right|
        \leq C t^{-d/\alpha_{\min} -k/\alpha(z)- l },
        \quad 0<t\leq 1.
    \end{align}
\end{proposition}
\begin{proof}
    The idea for the estimates follows closely~\cite{KK11,KK13,K14}, therefore we just outline the essential steps and refer to these papers for details.

    By the Cauchy-Poincar\'e theorem we can shift the integration contour in the definition of $p_t^{z,\cut}(x)$  from $\rd$ to $\rd + iv$  for any  $v \in\rd$.  This gives
    \begin{align*}
        p_t^{z,\cut}(x)
        = (2\pi)^{-d} \int_{\rd} e^{-ix( \xi +iv) -  \int_0^t  \psi_r^{z,\cut} (\xi+iv)\,dr}\,d\xi.
    \end{align*}
    We have
    \begin{align*}
    \psi_r^{z,\cut} (\xi+iv)&=  \int_{|u|\leq r^\zeta(z)} \left(1- e^{-v \cdot u}-v \cdot u\right) \mu(z,du)\\
    &\qquad\mbox{}+  \int_{|u|\leq r^{\zeta(z)}} e^{-v \cdot  u} \left(1-e^{i\xi u} +  i\xi\cdot u\right) \mu(z,du)\\
    & =  \psi_r^{z,\cut} (iv)+ \int_{|u|\leq r^{\zeta(z)}} e^{-v \cdot  u} \left(1-e^{i\xi u} +  i\xi\cdot u\right) \mu(z,du).
    \end{align*}
   Take $v = -ct^{-\zeta(z)}|x|^{-1}x$. Then, for any $r\leq t  \leq 1$ and $u$ with $|u|\leq r^{\zeta}$, we have $|v\cdot u|\leq c$ and thus
    \begin{align}\label{appC-e12}
    \begin{aligned}
        \left|\psi_r^{z,\cut}(iv)\right|
        &= \int_{|u|\leq r^{\zeta(z)}} \left(1- e^{- v \cdot u} -v  \cdot u \right) \mu(z,du)\\
        & \leq C  \int_{|u|\leq r^{\zeta(z)} } |u|^2 \mu(z,du) \leq C r^{(2-\alpha(z))\zeta(z)}
        \leq C
    \end{aligned}.
    \end{align}
    In addition,
      \begin{align}\label{appC-e14}
     \begin{split}
        \Re\int_{|u|\leq r^{\zeta(z)}}& e^{-v \cdot  u} \left(1-e^{i\xi\cdot u} +  i\xi\cdot u\right) \mu(z,du)\\
        &= \int_{|u|\leq r^{\zeta(z)}} e^{-v \cdot  u} \left(1-\cos {\xi\cdot u}\right) \mu(z,du)\\
        &\geq e^{-c} \int_{|u|\leq r^{\zeta(z)}}  \left(1-\cos {\xi\cdot u}\right) \mu(z,du)=e^{-c}\Re \psi_r^{z,\cut}(\xi).
     \end{split}
    \end{align}
   Similarly,
   \begin{align}\label{appC-e16}
    \left|\psi_r^{z,\cut} (\xi+iv)\right|
    \leq C (1+|\xi|^{\alpha(z)}), \quad r\in (0,t], \; \xi\in \rd.
   \end{align}
    Denote
    \begin{align*}
        H(t,x,w)
        = -i x \cdot w - \int_0^t  \psi_r^{z,\cut} (w)\,dr, \quad w\in \Cc^d.
    \end{align*}
    Then
    \begin{align*}
        \Re H(t,x,\xi+iv)
        &  = x\cdot v - \int_0^t  \Re \psi_r^{z,\cut} (\xi+iv) \,dr \\
        &\leq x\cdot v  - \int_0^t   \Re \psi_r^{z,\cut}(iv)\,dr - e^{-c} \int_0^t  \Re\psi_r^{z,\cut}(\xi)\,dr\\
        &\leq x\cdot v  -  t \sigma_0 |\xi|^{\alpha(z)} + C;
    \end{align*}
    in the last estimate we apply~\eqref{appC-e02} and~\eqref{appC-e12}.

     Recall that $p_t^{z,\cut}$ is given by a Fourier transform. Thus, by a change of variables,~\eqref{appC-e16}  and the definition of $\zeta(z)$, cf.~\eqref{pfti-e12},
    \begin{align*}
    \begin{aligned}
    \left| \partial^l _t \nabla^k p_t^{z,\cut}(x)\right|
    &\leq  C e^{-c|x|t^{-\zeta(z)}}  \int_\rd (1+|\xi|^{\alpha(z)})^l  \left(|\xi|^k + c^kt^{-k\zeta(z)}\right) e^{-\sigma_0 t|\xi|^{\alpha(z)}}\, d\xi\\
    &\leq  C t^{-(d+k)/\alpha(z) -l  } e^{ -c|x|t^{-\zeta(z)}},
    \end{aligned}
    \end{align*}
    which proves~\eqref{appC-e08} and~\eqref{appC-e10}.
\end{proof}

Let $\delta < \frac 12 \alpha_{\max}^{-1} <  \zeta_{\min}$, cf.~\eqref{pfti-e46}  and~\eqref{pfti-e12}. Recall that the flows $\chi_t$,  $\kappa_t$  are defined in~\eqref{set-e18} and~\eqref{pfti-e24}, respectively.
\begin{proposition}\label{appC-09}
\begin{enumerate}
\item\label{appC-09-a}
    If $|y-\chi_t(x)|\leq t^\delta$, then there exist constants $C,\epsilon >0$ such that for all $t\in (0,1]$
    \begin{align}\label{appC-e18}
        \sup_{w\in \rd} \left|p_t^{y,\cut}(w) - p_t^{x,\cut}(w)\right|
        \leq C t^{\eta\delta- \sfrak(2-\alpha(x))} (1+ |\log t|)  t^{-d/\alpha(x)}.
    \end{align}

\item\label{appC-09-b}
     Suppose that $\mathfrak{g}, \mathfrak{s}$ are small enough, so that $\epsilon_0 := \left(\eta\delta- \sfrak(d+2-\alpha_{\min})-\gfrak d\right) \wedge (\delta d)>0$.    Then for all $t\in (0,1]$
    \begin{align}\label{appC-e20}
        \sup_{x\in \rd} \int_\rd \left|p_t^{y,\cut}(\kappa_t(y)-x) - p_t^{x,\cut}(\kappa_t(y)-x)\right| dy
        \leq C t^\epsilon.
    \end{align}

\item\label{appC-09-c}
    There exists some $\epsilon\in (0,1)$ such that  for all $t\in (0,1]$
    \begin{align}\label{appC-e22}
        \sup_{x,y\in\rd} \left|p_t^{y,\cut}(\kappa_t(y)-x) - p_t^{x,\cut}(\kappa_t(y)-x)\right|
        \leq C t^{-d/\alpha_{\min} + \epsilon}.
    \end{align}
\end{enumerate}
\end{proposition}

\begin{proof}
\ref{appC-09-a}\ \
Assume first that  $\zeta(y)\leq \zeta(x)$.  Observe that for $r\leq t \leq 1$
    \begin{align*}
        \left|\psi^{x,\cut}_r(\xi) - \psi^{y,\cut}_r(\xi)\right|
        &\leq  \left|\int_{|u|\leq r^{\zeta(x)}} \left(1- e^{i\xi\cdot u} + i\xi \cdot u\right) \left(\mu(x,du)-\mu(y,du)\right)\right|\\
        &\qquad\mbox{} + \left|\int_{r^{\zeta(x)} \leq |u|\leq r^{\zeta(y)}}  \left(1- e^{i\xi\cdot u} + i\xi \cdot u\right) \mu(y,du)\right|\\
        &=\mathrm{I}_1+\mathrm{I}_2.
    \end{align*}
    The integrand $h(u)= 1- e^{i\xi\cdot u} + i\xi \cdot u$, $|u|\leq 1$, satisfies the conditions~\eqref{appD-e06} and~\eqref{appD-e08} with $C_h= C |\xi|^2$. Therefore, we can apply Proposition~\ref{appD-03} and get
    \begin{align*}
        \mathrm{I}_1
        \leq C |\xi|^2 \left(r^{\zeta(x)(2-\alpha(x))}+r^{\zeta(x)(2-\alpha(y))}\right) |\log r| \left(|x-y|^\eta \wedge 1\right).
    \end{align*}
    Since  $|y-\chi_t(x)|\leq t^\delta$, we know from Corollary~\ref{appA-07} that $r^{\zeta(x)(2-\alpha(x))}\asymp r^{\zeta(x)(2-\alpha(y))}$. Therefore,
    \begin{align*}
        \mathrm{I}_1
        \leq C |\xi|^2  r^{\zeta(x)(2-\alpha(x))} |\log r| \left(|x-y|^\eta\wedge 1\right).
    \end{align*}

    For $\mathrm{I}_2$ we use the estimate $|1- e^{i\xi\cdot u} + i\xi \cdot u |\leq C |\xi|^2 |u|^2$, the boundedness of $\lambda(\cdot)$ and~\eqref{appD-e06}, to get
    \begin{align*}
        \mathrm{I}_2
        &\leq C \lambda_{\max} \underbrace{\sigma(y, \Sph^{d-1})}_{=1} |\xi|^2 \left|\int_{r^{\zeta(x)}}^{r^{\zeta(y)}} \rho^{1-\alpha(y)}\, d\rho\right|\\
        &\leq \frac{C}{2-\alpha(y)} \left|r^{\zeta(x)(2-\alpha(y))}- r^{\zeta(y)(2-\alpha(y))}\right|\\
        &\leq   C r^{\zeta(x)(2-\alpha(x))} |\log r| \left(|x-y|^\eta  \wedge 1\right);
    \end{align*}
    in the last line we use the estimate $|1- r^x|\leq  C |x|\cdot |\log r|$ and Corollary~\ref{appA-07} together with the fact that $\zeta(\cdot)=1/\alpha(\cdot)-\sfrak$ inherits the  $\eta$-H\"older continuity from $\alpha(\cdot)$. Thus,
    \begin{align}\label{appC-e24}
        \left|\psi^{x,\cut}_r(\xi) - \psi^{y,\cut}_r(\xi)\right|\leq C |\xi|^2 t^{\eta\delta}  r^{\zeta(x)(2-\alpha(x))}  |\log r|.
    \end{align}
    For $z_1,z_2\in \Cc$  such that $\Re z_k\leq 0$, $k=1,2$, we have
    \begin{align*}
        |e^{z_1}-e^{z_2}|
        = \left|\int_{z_1}^{z_2} e^w \,dw\right|
        &\leq |z_1-z_2| e^{-|\Re z_1|\wedge |\Re z_2|}\\
        &\leq |z_1-z_2| \left(e^{-|\Re z_1|} +  e^{-|\Re z_2|}\right).
    \end{align*}
    Using this estimate, \eqref{appC-e02} and~\eqref{appC-e24},  we obtain  for all $w\in \rd$
    \begin{align*}
        \sup_{w\in \rd} &\left|p_t^{x,\cut}(w) -  p_t^{y,\cut}(w)\right| \\
        &= (2\pi)^{-d} \left|\int_\rd e^{-i\xi \cdot x} \left(e^{-\int_0^t  \psi^{x,\cut}_r(\xi)\,dr} - e^{-\int_0^t  \psi^{y,\cut}_r(\xi)\,dr}\right) d\xi\right| \\
        &\leq C  \left|\int_\rd \left[\int_0^t \left|\psi^{x,\cut}_r(\xi) - \psi^{y,\cut}_r(\xi)\right| dr  \right]
        \left(e^{-c t|\xi|^{\alpha(x)}}+e^{-c t|\xi|^{\alpha(y)}}\right) d\xi\right|\\
        &\leq  C  t^{\eta\delta - \sfrak(2-\alpha(x))}(1+|\log t|) \int_\rd  \left[|\xi| t^{1/\alpha(x)}\right]^2   \left(e^{-c t|\xi|^{\alpha(x)}}+e^{-c t|\xi|^{\alpha(y)}}\right) d\xi.
     \end{align*}
    A change of variables according to $\xi \to t^{-1/\alpha(z)} \xi$ (with $z=x,y$) gives
    \begin{align*}
        \int_\rd   \left(|\xi|t^{1/\alpha(x)}\right)^2 e^{- c t |\xi|^{\alpha(z)}} \, d\xi
        &= t^{-d/\alpha(z)} t^{2/\alpha(x)-2/\alpha(z)}  \int_\rd  |\xi|^2 e^{- c |\xi|^{\alpha(z)}} \, d\xi \\
        &\leq  C t^{-d/\alpha(x)}t^{(d+2)(1/\alpha(x)-1/\alpha(z))}.
     \end{align*}
    If $|y- \chi_t(x)|< t^\delta$, we can bound $t^{(d+2)(1/\alpha(x)-1/\alpha(z))}$ by a constant, cf.\ Corollary~\ref{appA-07}), so
    \begin{align*}
        \int_\rd \left[ |\xi|t^{1/\alpha(x)} \right]^2 e^{- c t |\xi|^{\alpha(z)}} \,d\xi
        \leq C t^{-d/\alpha(x)}, \quad z=x,y.
    \end{align*}
    This proves
    \begin{align*}
        \sup_{w\in \rd} \left|p_t^{x,\cut}(w) -  p_t^{y,\cut}(w)\right|
        \leq C t^{\eta\delta- \sfrak(2-\alpha(x))} (1+|\log t|)   t^{-d/\alpha(x)}.
    \end{align*}
The other case when $\zeta(y)>\zeta(x)$ can be treated in a similar way; the crucial observation here is that the  condition $|y-\chi_t(x)|\leq t^\delta$ guarantees that $x$ and $y$ are `close'.

\medskip\noindent\ref{appC-09-b}\ \
We split the integral appearing on the left-hand side of~\eqref{appC-e20} into three parts $\mathrm{I}_{1} +\mathrm{I}_{2}+\mathrm{I}_{3}$
\begin{align*}
    \mathrm{I}_{k}
    &=
    \int_{A_k} \left|p_t^{x,\cut}(\kappa_t(y)-x)-  p_t^{y,\cut}(\kappa_t(y)-x)\right| dy,
    \quad k=1,2,3,\\
\intertext{with the three separate integrals ranging over}
    A_1 &= \left\{y : |y- \chi_t(x)|\leq t^{\zeta(x)-\gfrak}\right\},\\
    A_2 &= \left\{y : t^{\zeta(x)-\gfrak}<|y- \chi_t(x)|\leq t^\delta\right\},\\
    A_3 &= \left\{y : |y- \chi_t(x)|> t^\delta\right\},
\end{align*}
respectively, with $0 < \gfrak < \zeta(x)-\delta$ and $\delta< \frac 12 \alpha_{\max}^{-1} < \zeta_{\min}$.  If we choose $\sfrak$ and $\gfrak$ small enough,  we can use~\eqref{appC-e18} to see
\begin{align*}
    \mathrm{I}_{1}
    &\leq C t^{\eta\delta- \sfrak(2-\alpha(x))}  (1+ |\log t|)  t^{-d/\alpha(x)} \int_{|y- \chi_t(x)| \leq t^{\zeta(x)-\gfrak}}\,dy\\
    &\leq C t^{\eta\delta- \sfrak(2-\alpha(x))} (1+ |\log t|)   t^{-d/\alpha(x)+  d (\zeta(x)-\gfrak)} \\
    &= C t^{\eta\delta- \sfrak(d+2-\alpha(x))-\gfrak d}   (1+ |\log t|)\\
    &\leq C t^{\epsilon_1},
\end{align*}
where $2\epsilon_1 := \eta\delta- \sfrak(d+2-\alpha_{\min})-\gfrak d \in  (0,1)$.

For $\mathrm{I}_{2}$ we use~\eqref{appC-e08} with $k=0$,
\begin{align*}
    \mathrm{I}_{2}
    &\leq \int_{t^{\zeta(x) - \gfrak}<|y- \chi_t(x)|\leq t^\delta}  \left(p_t^{x,\cut}(\kappa_t(y)-x) +  p_t^{y,\cut}(\kappa_t(y)-x)\right) dy\\
    &\leq C  t^{-\sfrak d} \int_{ t^{\zeta(x)-\gfrak}<|y- \chi_t(x)|\leq t^\delta}\left(f_{t,\zeta(x),c}(\kappa_t(y)-x)+ f_{t,\zeta(y),c}(\kappa_t(y)-x)\right) dy.
\intertext{With~\eqref{appA-e46} and the definition of $f_{t,\zeta(\cdot),c}$ we get}
    \mathrm{I}_{2}
    &\leq C t^{-\sfrak d} \int_{ t^{\zeta(x)-\gfrak}<|y- \chi_t(x)|\leq t^\delta} f_{t,\zeta(x),c}(\kappa_t(y)-x) \,dy\\
    &= C t^{-d/\alpha(x)} \int_{ t^{\zeta(x)-\gfrak}<|y- \chi_t(x)|\leq t^\delta}e^{-c|\kappa_t(y)-x|t^{-\zeta(x)}} \,dy.
\intertext{Using Corollary~\ref{appA-21} with $s=t$ and $1/\alpha(x) - \zeta(x) + \epsilon_\kappa = \mathfrak s + \epsilon_\kappa> 0$ we get}
    \mathrm{I}_{2}
    &\leq Ct^{-d/\alpha(x)} \int_{ t^{\zeta(x)-\gfrak}<|y- \chi_t(x)|\leq t^\delta}e^{-c' |y-\chi_t(x)|t^{-\zeta(x)}} e^{c't^{1/\alpha(x)-\zeta(x)+\epsilon_\kappa}} \,dy\\
    &\leq Ct^{-d/\alpha(x)} t^{\delta d} e^{-c t^{-\gfrak}}
    \leq C t^{\epsilon_2},
\end{align*}
where $\epsilon_2 = \delta d\in (0,1)$.

Let us, finally, estimate $\mathrm{I}_{3}$. Using first Proposition~\ref{appC-07} and then Corollary~\ref{appA-21} with $s=t$ we get for any $K\geq 1$
\begin{align*}
    \mathrm{I}_{3}
    &\leq  C t^{-K} \int_{|y- \chi_t(x)|> t^\delta} e^{-c |\kappa_t(y)-x|t^{-\zeta_{\min}}} \, dy\\
    &\leq C t^{-K} \int_{|y- \chi_t(x)|> t^\delta} e^{-c' |y-\chi_t(x)|t^{-\zeta_{\min}}} e^{c't^{1/\alpha(x)-\zeta_{\min}+\epsilon_\kappa}} \,dy
\end{align*}
     Since $1/\alpha(x)-\zeta_{\min}+\epsilon_\kappa \geq 1/\alpha(x)-\zeta(x)+\epsilon_\kappa = \sfrak+\epsilon_\kappa > 0$ we see
\begin{align*}
    \mathrm{I}_{3}
    \leq t^{-K}  t^{\zeta_{\min} d} e^{-\frac12c't^{\delta-\zeta_{\min}}} \int_{|w|> 1} e^{-  \frac 12c' |w|} \,dw
    \leq C t^{K'},
\end{align*}
where $K'>0$ is arbitrary.

Combining the estimates for $\mathrm{I}_1, \mathrm{I}_2, \mathrm{I}_3$ gives~\eqref{appC-e20}  with $\epsilon_0:= (2\epsilon_1)\wedge \epsilon_2$.

\medskip\noindent\ref{appC-09-c}\ \
Let  $0<\delta < \frac 12 \alpha_{\max}^{-1} < \zeta_{\min}$  be as before, and write
\begin{align*}
    \sup_{x,y\in\rd} \left|p_t^{y,\cut}(\kappa_t(y)-x) - p_t^{x,\cut}(\kappa_t(y)-x)\right| \leq \mathrm{J}_1 + \mathrm{J}_2,
\intertext{where we set}
    \mathrm{J}_1
    := \sup_{x,y: |y-\chi_t(x)|\leq t^\delta} \left|p_t^{y,\cut}(\kappa_t(y)-x) - p_t^{x,\cut}(\kappa_t(y)-x)\right|,\\
    \mathrm{J}_2
    := \sup_{x,y: |y-\chi_t(x)|> t^\delta} \left|p_t^{y,\cut}(\kappa_t(y)-x) - p_t^{x,\cut}(\kappa_t(y)-x)\right|.
\end{align*}
From part~\ref{appC-09-a} we have the following estimate for $\mathrm{J}_1$:
\begin{align*}
    \mathrm{J}_1
    \leq C t^{-d/\alpha_{\min}+\epsilon_1}.
\end{align*}
In order to estimate $\mathrm{J}_2$, we slightly modify the calculation for $\mathrm{I}_{3}$ in the proof of part~\ref{appC-09-b} and get
\begin{align*}
    \mathrm{J}_2
    &\leq C t^{-K} \sup_{x,y: |y-\chi_t(x)|> t^\delta}  e^{-c' |y-\chi_t(x)|t^{-\zeta_{\min}}} e^{c't^{1/\alpha(x)-\zeta_{\min}+\epsilon_\kappa}} \\
    &\leq C t^{-K} e^{-c't^{\delta-\zeta_{\min}}}
    \leq C t^{K'},
\end{align*}
for an arbitrary $K'>0$. This completes the proof of~\eqref{appC-e22}.
\end{proof}

The following results are needed in Section~\ref{pftii} where we prove  Theorems~\ref{t2} and~\ref{t3}.
Recall the definition~\eqref{appA-e02} of the flows $\chi_t(x)$ and $\kappa_t(y)$,  and the definition~\eqref{appA-e22} of the auxiliary `intermediate' flow  $\chi_s^t(x)$ which connects $\chi_t(x)$ and $\kappa_t(y)$.

\begin{lemma}\label{appC-11}
 There exists a constant $C>0$ such that for all $x,y\in\rd$ and $t\in (0,1]$
\begin{align}\label{appC-e26}
    \left|p_t^{x,\cut}(\kappa_t(y)-x)-  p_t^{x,\cut}(y-\chi_t^t(x))\right|
    \leq  C t^{\epsilon_B-\sfrak(d+1)} K^{1;c}_t(x,y).
\end{align}
\end{lemma}
\begin{proof}
Denote the expression on the left-hand side of~\eqref{appC-e26} by $R^{(1)}_t(x,y)$.  From the definition of $\kappa_s$ and $\chi_s^t$ we have
\begin{align*}
     R^{(1)}_t(x,y)
     &= \left|\int_0^t \frac{d}{ds} \left(p_t^{x,\cut} (\kappa_{t-s}(y)-\chi_s^t(x))\right) ds\right|\\
     &= \left|\int_0^t   \left(\nabla p_t^{x,\cut}\right) (\kappa_{t-s}(y)-\chi_s^t(x)) \left(B_{t-s}(\kappa_{t-s}(y))- B_{t-s}(\chi_s^t(x))\right)ds\right|.
\intertext{Now we use the fact that $B_t$ is Lipschitz (cf.\ Proposition~\ref{appA-13}), the estimate~\eqref{appC-e08} for the gradient $\nabla p_t^{x,\cut}$,
    and~\eqref{pfti-e48}, and  we get}
    R^{(1)}_t(x,y)
    &\leq C t^{-\sfrak d-1/\alpha(x)}  \int_0^t   f_{t,\zeta(x),c}(\kappa_{t-s}(y)-\chi_s^t(x)) (t-s)^{-1+\epsilon_B} |\kappa_{t-s}(y)- \chi_s^t(x)|\,ds\\
    &\leq C t^{-\sfrak (d+1)} \int_0^t  f_{t,\zeta(x),c'}(\kappa_{t-s}(y)-\chi_s^t(x)) (t-s)^{-1+\epsilon_B}\,ds\\
    &\leq C t^{-\sfrak (d+1)} f_{t,\zeta(x),c''}(y-\chi_t^t(x))\int_0^t (t-s)^{-1+\epsilon_B}\,ds\\
    &\leq C t^{\epsilon_B-\sfrak (d+1)} f_{t,\zeta(x),c''}(y-\chi_t(x)) e^{-ct^{\epsilon_\kappa}}\\
    &\leq C t^{\epsilon_B-\sfrak (d+1)} K^{1;c''}_t(x,y).
\end{align*}
In the third and the second lines from below we use the definition of $f_{t,\zeta(x),c''}$ as well as~\eqref{appA-e24} (twice) and~\eqref{appA-e44}, respectively, which allow us to switch from $|\kappa_{t-s}(y)-\chi_s^t(x)|$ to $|y-\chi_t(x)|$, $0<s\leq t\leq 1$. The last line requires the estimate~\eqref{pfti-e64} only.
\end{proof}

\begin{lemma}\label{appC-13}
    For any $\epsilon_\kappa$ satisfying $0 < \epsilon_\kappa < \frac 14 \min\left\{\epsilon_\nu \alpha_{\max}^{-1}, \epsilon_B, \eta\right\}$, there exists a constant $C>0$,  such that  for all $x,y\in\rd$ and $t\in (0,1]$
    \begin{align}\label{appC-e28}
        \left|p_t^{x,\cut}(y-\chi_t^t(x)) -  p_t^{x,\cut}(y-\chi_t(x))\right|
        \leq C t^{\epsilon_\kappa-\sfrak d} K_t^{1;c}(x,y);
    \end{align}
    one has $N> d\alpha_{\min}^{-1}- d\alpha_{\max}^{-1}$ in the definition of $K_t^{1;c}(x,y)$.
\end{lemma}
\begin{proof}
By the Taylor formula, \eqref{appA-e44}  and the estimate~\eqref{appC-e08}  for $\nabla p_t^{x,\cut}$   we get
\begin{align*}
    &\left| p_t^{x,\cut}(y-\chi_t^t(x)) -  p_t^{x,\cut}(y-\chi_t(x))\right| \\
    &\qquad\leq C |\chi_t^t(x)- \chi_t(x)| \int_0^1\left|\nabla p_t^{x,\cut}\left(y- \chi_t(x) + \theta( \chi_t(x)-\chi_t^t(x))\right)\right| d\theta \\
    &\qquad\leq C  t^{1/\alpha(x)+\epsilon_\kappa} t^{-\sfrak d -1/\alpha(x)}  f_{t,\zeta(x),c} (y- \chi_t(x))e^{c |\chi_t(x)-\chi_t^t(x) |t^{-\zeta(x)}} \\
    &\qquad\leq  C  t^{1/\alpha(x)+\epsilon_\kappa} t^{-\sfrak d -1/\alpha(x)}  f_{t,\zeta(x),c} (y- \chi_t(x))e^{c' t^{\sfrak+\epsilon_\kappa}}\\
    &\qquad \leq C  t^{\epsilon_\kappa -\sfrak d }  f_{t,\zeta(x),c} (y- \chi_t(x))\\
    &\qquad\leq  C t^{\epsilon_\kappa-\sfrak d} K_t^{1;c}(x,y);
\end{align*}
in the last line we use the estimate~\eqref{pfti-e64}.
\end{proof}

\begin{corollary}\label{appC-15}
We have $\displaystyle
\quad \sup_{x\in\rd} \int_\rd p_t^0(x,y)\,dy\leq C \quad
$
for all $t\in (0,1]$  with an absolute constant $C$.
\end{corollary}
\begin{proof}
Recall that $p_t^0(x,y)= p_t^{y,\cut}(\kappa_t(y)-x)$. By the triangle inequality, Lemma~\ref{appC-11} and Lemma~\ref{appC-13}  we have
\begin{align*}
    &p_t^{y,\cut}(\kappa_t(y)-x)\\
    &\leq \left|p_t^{y,\cut}(\kappa_t(y)-x)-p_t^{x,\cut}(\kappa_t(y)-x)\right|
        + \left|p_t^{x,\cut}(\kappa_t(y)-x)-p_t^{x,\cut}(y-\chi_t^t(x))\right|\\
    &\qquad \mbox{}+ \left|p_t^{x,\cut}(y-\chi_t^t(x))-p_t^{x,\cut}(y-\chi_t(x))\right| + p_t^{x,\cut}(y-\chi_t(x))\\
    &\leq \left|p_t^{y,\cut}(\kappa_t(y)-x)-p_t^{x,\cut}(\kappa_t(y)-x)\right|
        + C t^{\epsilon_\kappa-\sfrak d} K_t^{1;c}(x,y)
        + p_t^{x,\cut}(y-\chi_t(x));
\end{align*}
in the last inequality use $\epsilon_\kappa - \sfrak d< \epsilon_B - \sfrak(d+1)$, i.e.\ $\sfrak < \epsilon_B- \epsilon_\kappa$. Since $\epsilon_\kappa < \frac 14\eta \min\{\epsilon_\nu, \epsilon_B\}\leq \frac 14 \epsilon_B$, we have $\epsilon_B - \epsilon_\kappa > \frac3 4 \epsilon_B$, and the required inequality follows because of \eqref{pfti-e74}.

The first term on the right hand side is integrable in $y$ because of Proposition~\ref{appC-09}.\ref{appC-09-b}, the second term is integrable in $y$ because of Proposition~\ref{appB-07}, and the third term is integrable in $y$ because of the shift invariance of Lebesgue measure.
\end{proof}

The next Proposition shows how close $p_t^{y,\cut}(\cdot)$ and $t^{-d/\alpha(x)} g^x \left(\frac{\cdot}{t^{1/\alpha(x)}}\right)$ are; the latter is the transition probability density  of an $\alpha(x)$-stable process whose characteristics are frozen in the starting point  $x$.

\begin{proposition}\label{appC-17}
Let $g^x(w)$ be  the transition probability density of an  $\alpha(x)$-stable random variable with characteristic exponent~\eqref{set-e22}. There exists some $\epsilon_R\in (0,1)$ such that for all $t\in (0,1]$
\begin{align}
\label{appC-e30}
    \sup_{x,y\in\rd} \left|p_t^{y,\cut}(\kappa_t(y)-x) - \frac{1}{t^{d/\alpha(x)}} g^x \left(\frac{y- \chi_t(x)}{t^{1/\alpha(x)}}\right)\right|
    \leq C t^{-d/\alpha_{\min}+\epsilon_R},
\\\label{appC-e32}
    \sup_{x\in\rd}\int_\rd \left|p_t^{y,\cut}(\kappa_t(y)-x) - \frac{1}{t^{d/\alpha(x)}} g^x \left(\frac{y- \chi_t(x)}{t^{1/\alpha(x)}}\right)\right| dy
    \leq C t^{\epsilon_R}.
\end{align}
\end{proposition}

\begin{proof}
We have
\begin{align*}
    &\left|p_t^{y,\cut}(\kappa_t(y)-x)- \frac{1}{t^{d/\alpha(x)}} g^x \left(\frac{y- \chi_t(x)}{t^{1/\alpha(x)}}\right)\right|
    \leq \mathrm{J}_1+\mathrm{J}_2+\mathrm{J}_3,\\
    &\qquad\mathrm{J}_1
    := \left|p_t^{y,\cut}(\kappa_t(y)-x)- p_t^{x,\cut}(\kappa_t(y)-x)\right|,\\
    &\qquad\mathrm{J}_2
    := \left|p_t^{x,\cut}(\kappa_t(y)-x)- p_t^{x,\cut}(y-\chi_t(x))\right|,\\
    &\qquad\mathrm{J}_3
    := \left|p_t^{x,\cut}(y-\chi_t(x))-\frac{1}{t^{d/\alpha(x)}} g^x \left(\frac{y-\chi_t(x)}{t^{1/\alpha(x)}}\right)\right|.
\end{align*}
With  Proposition~\ref{appC-09}.\ref{appC-09-c}  we see that for some $\epsilon_1\in (0,1)$
\begin{align*}
    \sup_{x,y\in\rd} \mathrm{J}_1 \leq C t^{-d/\alpha_{\min}+\epsilon_1}.
\end{align*}
Lemmas~\ref{appC-11}, \ref{appC-13} and the estimate
\begin{align*}
    \sup_{x,y\in\rd} K^{1;c}_t(x,y)
    \leq C t^{-d/\alpha_{\min}}
\end{align*}
which follows from the definition of $K^{1;c}_t(x,y)$  imply
\begin{align*}
    \sup_{x,y\in\rd} \mathrm{J}_2
    \leq C t^{-d/\alpha_{\min}+\epsilon_\kappa-\sfrak d}.
\end{align*}
Let us estimate $\mathrm{J}_3$. Recall that $\sfrak = \frac{1}{16 d}\eta \min\left\{\epsilon_\nu, \epsilon_B\right\}$ and  $0 < \epsilon_\kappa <\frac 14 \min\left\{\epsilon_\nu \alpha_{\max}^{-1}, \epsilon_B, \eta\right\}$. Without loss of generality we may take $\epsilon_\kappa= \frac 18 \eta \min\left\{\epsilon_\nu, \epsilon_B\right\}$, and so  $\sfrak < \epsilon_\kappa / d$. Recall that the density  $g^x(w)$ has the characteristic exponent $\psi^{x,\upsilon} (\xi)$ which is defined as sum of the function $\psi^x(\xi)$,
\begin{align*}
    \psi^x ( \xi)
    = \int_\rd \left(1-e^{i\xi u} + i\xi u   \I_{\{|u|\leq 1\}}\right) \mu(x,du),
\end{align*}
see~\eqref{set-e22}, and the drift $\upsilon(x)$, see~\eqref{set-e20}.   Therefore,
\begin{align}\label{appC-e34}
    t^{-d/\alpha(x)} g^x (wt^{-1/\alpha(x)})
    & = (2 \pi)^{-d}  t^{-d/\alpha(x)}   \int_\rd e^{-i \xi t^{-1/\alpha(x)} w - \psi^{x,\upsilon} (\xi)} \,d\xi \\
    &=\notag (2 \pi)^{-d} \int_\rd e^{-i \xi w - \psi^{x,\upsilon} (t^{1/\alpha(x)} \xi)} \,d\xi \\
    &=\notag (2 \pi)^{-d} \int_\rd e^{-i \xi w - \psi^x(t^{1/\alpha(x)} \xi)+ i\xi t^{1/\alpha(x)}\upsilon(x)} \,d\xi,
\end{align}
Since $\mu(x,du)$ is an $\alpha(x)$-stable L\'evy measure, we get by scaling
\begin{align*}
    \psi^x &(t^{1/\alpha(x)} \xi)\\
    &= \int_\rd \left(1-e^{i\xi t^{1/\alpha(x)} u} + i\xi u t^{1/\alpha(x)}  \I_{\{|u|\leq 1\}}\right) \mu(x,du)\\
    &= t \int_\rd \left(1-e^{i\xi u} + i\xi u  \I_{\{|u|\leq t^{1/\alpha(x)}\}}\right) \mu(x,du)\\
    &= \int_0^t\!\!\int_\rd \left(1-e^{i\xi u} + i\xi u  \I_{\{|u|\leq t^{1/\alpha(x)}\}}\right) \mu(x,du)\,ds\\
    &= \int_0^t\!\!\int_{|u|\leq s^{\zeta(x)}} \left(1- e^{i\xi u} + i\xi u \I_{\{|u|\leq s^{1/\alpha(x)}\}}\right) \mu(x,du)\,ds\\
    &\quad \mbox{}+ \int_0^t\!\!\int_{|u|>s^{\zeta(x)}} \left(1-e^{i\xi u}\right)  \mu(x,du)\,ds
            + i\xi \int_0^t\!\!\int_{s^{1/\alpha(x)}<|u|\leq t^{1/\alpha(x)}}\hspace{-10pt} u \,\mu(x,du)\,ds\\
    &= \int_0^t \psi_s^{x,\cut}(\xi)\,ds+ \int_0^t\!\!\int_{|u|>s^{\zeta(x)}} \left(1-e^{i\xi u}\right) \mu(x,du)\,ds
            + i\xi t^{1/\alpha(x)} \upsilon(x);
\end{align*}
in the last line we use  the definition~\eqref{pfti-e14} of $\psi_s^{x,\cut}(\xi)$ and~\eqref{appA-e26}, \eqref{appA-e30}.

Substituting this into~\eqref{appC-e34} we obtain
\begin{align*}
    t^{-d/\alpha(x)} &g^x (wt^{-1/\alpha(x)}) \\
    &= (2\pi)^{-d} \int_\rd e^{- iw \xi -\int_0^t \psi_s^{x,\cut}(\xi)\,ds
        +  \int_0^t \int_{|u|>s^{1/\alpha(x)}} \left(1-e^{i\xi u}\right) \mu(x,du)\,ds}\, d\xi \\
    &= \int_\rd p^{x,\cut}_{t} (w-z) \, P_t^{x,\tail}(dz),
\end{align*}
where  $P_t^{x,\tail}(dz)$   is the exponential (for the convolution) of the measure $\Lambda_t^{x,\tail}$, i.e.
\begin{align*}
    P_t^{x,\tail}(A)
    = e^{-\Lambda_t^{x,\tail}(\rd)}  \sum_{k=0}^\infty \frac 1{k!}\,(\Lambda_t^{x,\tail})^{* k}(A),
    \quad A \in \Bscr(\rd),
\intertext{with intensity  measure}
    \Lambda_t^{x,\tail}(A)
    = \int_0^t \mu(x,\{ u\in A:\, |u|> s^{\zeta(x)} \})\,ds.
\end{align*}
Because of the scaling property of $\mu$ and $\zeta(x)=1/\alpha(x)-\sfrak$ (cf.~\eqref{pfti-e12})  we have
\begin{align}\label{appC-e36}
\begin{aligned}
    \Lambda_t^{x,\tail}(\rd)
    &= \mu(x, \{u: |u|\geq 1\}) \int_0^t s^{-\zeta(x)\alpha(x)}\,ds\\
    &= \mu(x, \{u: |u|\geq 1\}) \, \frac{t^{1-\zeta(x)\alpha(x)}}{1-\zeta(x)\alpha(x)}\\
    &= \frac{t^{\sfrak\alpha(x)}}{\sfrak\alpha(x)} \, \mu(x, \{u: |u|\geq 1\})\\
    &\leq \frac{t^{\sfrak\alpha_{\min}}}{\sfrak\alpha_{\min}} \, \mu(x, \{u: |u|\geq 1\}) \\
    &\leq C t^{\sfrak\alpha_{\min}}.
\end{aligned}
\end{align}
From this we conclude that
\begin{align}\label{appC-e38}
    \left|1- e^{- \Lambda_t^\tail(\rd)}\right| \leq  C t^{\sfrak\alpha_{\min}}.
\end{align}
Further,
\begin{align}\label{appC-e40}
    &\left|p_t^{x,\cut}(y-\chi_t(x))-\frac{1}{t^{d/\alpha(x)}} g^x \left(\frac{y- \chi_t(x)}{t^{1/\alpha(x)}}\right)\right|\\
    &\notag\qquad\leq  p_t^{x,\cut}(y-\chi_t(x)) \left|1- e^{- \Lambda_t^\tail(\rd)}\right|\\
    &\notag\qquad\quad \mbox{} +e^{-\Lambda_t^{x,\tail}(\rd)}  \sum_{k=1}^\infty \int_\rd p_t^{x,\cut} (y-\chi_t(x)-z) \frac{1}{k!} \, (\Lambda_t^{x,\tail})^{* k}(dz).
\end{align}
Combining this with the estimate~\eqref{appC-e10} for $p_t^{x,\cut}$ and~\eqref{appC-e38} we get
\begin{align*}
    \mathrm{J}_3\leq  C t^{-d/\alpha_{\min}} t^{\sfrak\alpha_{\min}},
\end{align*}
and the proof of~\eqref{appC-e30} is complete.

\medskip
We will now prove~\eqref{appC-e32}. Denote by $\mathrm{I}_1$, $\mathrm{I}_2$ and $\mathrm{I}_3$ the integrals of $\mathrm{J}_1$, $\mathrm{J}_2$ and $\mathrm{J}_3$ with respect to $y$. From~\eqref{appC-e20} we know that
\begin{align*}
    \mathrm{I}_1\leq C t^{\epsilon_1}
    \quad\text{with}\quad
    2\epsilon_1 := \left(\eta\delta- \sfrak(d+2-\alpha_{\min})-\gfrak d\right) \wedge \delta d.
\end{align*}
For $\mathrm{I}_2$ we get with Lemma~\ref{appC-11} and Lemma~\ref{appC-13} combined with~\eqref{pfti-e62}
\begin{align*}
    \mathrm{I}_2 \leq C t^{\epsilon_2} \quad \text{with}\quad \epsilon_2:=\epsilon_\kappa - \sfrak d.
\end{align*}
Integrating~\eqref{appC-e40} in $y$  and using~\eqref{appC-e36}, \eqref{appC-e38},  we see
\begin{align*}
    \mathrm{I}_3
    \leq C t^{\epsilon_3}
    \quad\text{with}\quad
    \epsilon_3=\sfrak\alpha_{\min}.
\end{align*}
This shows that~\eqref{appC-e30} holds true with $\epsilon_R:= \min\{\epsilon_1,\epsilon_2,\epsilon_3\}$.
\end{proof}

We can use the estimate~\eqref{appC-e32} to obtain the following  properties  of $p^0_t(x,y)$.

\begin{proposition}\label{appC-19}
For every $f\in C_\infty(\rd)$
\begin{align}\label{appC-e42}
\lim_{|x|\to\infty} \int_\rd p_t^0(x,y)f(y)\,dy =0 \quad \text{for any}\quad  t>0,
\end{align}
\begin{align}\label{appC-e44}
\lim_{t\to 0}\sup_{x\in \rd} \left|\int_\rd p_t^0(x,y)f(y)\,dy -f(x)\right|= 0.
\end{align}
\end{proposition}
\begin{proof}[Proof of~\eqref{appC-e42}] Take $f\in C_\infty(\rd)$, $R>0$,  and split
\begin{align*}
\bigg|\int_\rd p_t^0(x,y) f(y)\,dy \bigg| & =  \bigg|\bigg(\int_{|y|\leq R } +\int_{|y|> R } \bigg) p_t^0(x,y+x) f(y+x)\,dy \bigg|\\
&= I_1(t,x)+ I_2(t,x).
\end{align*}
Since $x\mapsto p_t^0(x,x+y)f(x+y)$ is in $C_\infty(\rd)$, we can use the dominated convergence theorem to see that $I_1(\cdot,y)\in C_\infty(\rd)$.  For  $I_2(t,x)$ that statement follows from Corollary~\ref{appC-15} and   the fact that $f\in C_\infty(\rd)$.

\medskip\noindent
\emph{Proof of~\eqref{appC-e44}}. Split
\begin{align*}
    \Big|\int_\rd p_t^0(x,y)f(y)\,dy- f(x)\Big|
    &\leq   \int_\rd \bigg|  p_t^0(x,y)- \frac{1}{t^{d/\alpha(x)}} g^x \left(\frac{y- \chi_t(x)}{t^{1/\alpha(x)}}\right)\bigg| |f(y)|\,dy\\
    &\qquad\quad \mbox{}+ \left|\int_\rd \frac{1}{t^{d/\alpha(x)}} g^x \left(\frac{y- \chi_t(x)}{t^{1/\alpha(x)}}\right)( f(y)- f(x)) \,dy\right|.
\end{align*}
Then the first term converges to $0$ as $t\to 0$ by~\eqref{appC-e32}, and the second converges to $0$;  for this we first changes variables and use then dominated convergence.  Note that $g^x\in L^1$ -- it is the transition probability density of a stable random variable -- and that $\chi_t(x)\to x$ as $t\to 0$; because of the continuity of $f$ we get
\begin{gather*}
    \lim_{t\to 0}\left|\int_\rd g^x (y)\left( f(t^{1/\alpha(x)} y+\chi_t(x))- f(x)\right) dy \right| = 0.
\qedhere
\end{gather*}
\end{proof}

Although the estimate~\eqref{appC-e08} from Proposition~\ref{appC-07} is sufficient for many applications -- for example in order to prove Theorem~\ref{t1} or the representation of the density~\eqref{set-e24}, \eqref{set-e26} in Theorem~\ref{t2} -- the right-hand side of~\eqref{appC-e08} is nevertheless somewhat crude; if we take, for example, $k=0$, then
\begin{align*}
    1
    = \int_\rd p_t^{z,\cut}(x)\,dx
    \leq \int_\rd t^{-\sfrak d} f_{t,\zeta(z),c}(x) \, dx
    = C t^{-\sfrak d}, \quad t\leq 1,
\end{align*}
shows that the estimate gets worse as $t\to 0$.

For the pointwise upper bound in Theorem~\ref{t3} we need the following, less explicit but more precise, (compound kernel) estimate of $p_t^{z,\cut}(x)$; similar compound kernel estimates were first established in~\cite{KK13,K14}.

We will need the following (signed) kernel. Recall that $\zeta(y) = \frac 1{\alpha(y)}-\sfrak$, hence $\alpha(y)\zeta(y) = 1-\alpha(y)\sfrak < 1$, $t^{1/(\zeta(y)\alpha(y))}\leq t$ and $s^{\zeta(y)}\leq t^{1/\alpha(y)}$ for all $0<s<t^{1/(\zeta(y)\alpha(y))}$. Define
\begin{align}\label{appC-e46}\begin{aligned}
    \Lambda(t,y,du)
    &:= -\int_0^{t^{1/(\zeta(y)\alpha(y))}} \mu(y,du) \I_{\{s^{\zeta(y)} < |u| \leq t^{1/\alpha(y)}\}}\,ds\\
    &\qquad\mbox{}    + \int_{t^{1/(\zeta(y)\alpha(y))}}^t \mu(y,du) \I_{\{t^{1/\alpha(y)} < |u| \leq s^{\zeta(y)}\}}\,ds.
\end{aligned}
\end{align}
 From the polar representation~\eqref{set-e06} we see that  the total variation of this kernel satisfies
\begin{align}\label{appC-e48}\begin{aligned}
    |\Lambda|(t,y,\rd)
    &:= \int_0^{t^{1/(\zeta(y)\alpha(y))}} \mu\Big(y,\{u: s^{\zeta(y)} < |u| \leq t^{1/\alpha(y)}\}\Big)\,ds\\
    &\qquad\mbox{}+ \int_{t^{1/(\zeta(y)\alpha(y))}}^t \mu\Big(y,\{u:  t^{1/\alpha(y)} < |u| \leq s^{\zeta(y)}\}\Big)\,ds\\
    &= C\left( \int_0^{t^{1/(\zeta(y)\alpha(y))}}(t^{-1}- s^{-\alpha(y) \zeta(y)} )\,ds+\int_{t^{1/(\zeta(y)\alpha(y))}}^t (s^{-\alpha(y)\zeta(y)}- t^{-1})\, ds\right)\\
    &= C\left(t^{-1+  1/(\zeta(y)\alpha(y))} - \frac{t^{(1-\alpha(y) \zeta(y))/(\alpha(y) \zeta(y))} }{1-\alpha(y) \zeta(y)} \right.\\
    &\left.\qquad\mbox{}+ \frac{t^{1-\alpha(y)\zeta(y)}-t^{(1-\alpha(y) \zeta(y))/(\alpha(y) \zeta(y))}}{1-\alpha(y) \zeta(y)} - t^{-1}(t-t^{1/(\zeta(y)\alpha(y))})\right)\\
    &= C\left(2 t^{-1+1/(\zeta(y)\alpha(y))}- \frac{2t^{(1-\alpha(y) \zeta(y))/(\alpha(y) \zeta(y))} }{1-\alpha(y) \zeta(y)}+ \frac{t^{1-\alpha(y)\zeta(y)}}{{1-\alpha(y) \zeta(y)}}-1\right)\\
    &\leq  C\left(t^{1-\alpha(y)\zeta(y)}+ t^{-1+1/(\zeta(y)\alpha(y))}\right)\\
    &\leq Ct^{\sfrak\alpha_{\min}}.
\end{aligned}
\end{align}

Finally, we denote by $P_t(y,du)$ the convolution-exponential of the signed kernel $\Lambda(t,y,du)$, i.e.
\begin{align}\label{appC-e50}
    P_t(y,du) := e^{-\Lambda(t,y,\rd)}\left[\delta_0(du)+ \sum_{k=1}^\infty \frac{1}{k!}\,\Lambda^{*k} (t,y,du)\right].
\end{align}
 Because of~\eqref{appC-e48} the total variation of this kernel is uniformly bounded in $t$ and $y$:
\begin{align}\label{appC-e52}
    \sup_{t\in (0,1], y\in \rd} |P_t|(y,\rd)<\infty.
\end{align}
Define
\begin{align}\label{appC-e54}
  k(t,z)
    := \int_0^t  \int u \,\Lambda(t,z,du)\,ds - \int_0^t\int u\I_{\{s^{1/\alpha(z)}<|u|\leq s^{\zeta(z)}\}}\,\mu(z,du)\,ds,
\end{align}
and observe that the following bound holds: There is some $C>0$ such that
\begin{align}\label{appC-e56}
    |k(t,z)|\leq Ct^{1/(2\alpha_{\max})}.
\end{align}
Indeed, using the convention that $\int_a^b = -\int_b^a$ for $a > b$ we arrive at:
\begin{align*}
    |k(t,z)|
    &\leq  C\int_0^t \left|\int_{t^{1/\alpha(z)}}^{s^{\zeta(z)}} r^{-\alpha(z)}\,dr\right| + \left(\int_{s^{1/\alpha(z)}}^{s^{\zeta(z)}} r^{-\alpha(z)}\,dr\right) ds;
\intertext{therefore, we get for  $\alpha(z)\neq 1$}
    |k(t,z)|&\leq \frac{C}{|1-\alpha(z)|} \int_0^t \left(s^{(1-\alpha(z))\zeta(z)} + t^{(1-\alpha(z))/\alpha(z)} + s^{(1-\alpha(z))/\alpha(z)}\right) ds \\
    &= \frac{C}{|1-\alpha(z)|} \int_0^t \left(s^{\zeta(z)-1+\alpha(z)\sfrak} + t^{1/\alpha(z)-1} + s^{1/\alpha(z)-1} \right) ds \\
    &= \frac{C}{|1-\alpha(z)|} \left( \frac{t^{\zeta(z)+\alpha(z)\sfrak}}{\zeta(z)+\alpha(z)\sfrak} + t^{1/\alpha(z)} + \alpha(z)t^{1/\alpha(z)} \right)\\
    &\leq C t^{1/(2\alpha_{\max})}.
\end{align*}
This gives~\eqref{appC-e56}, since $t\leq 1$ and $\alpha(z),\zeta(z)$ are uniformly bounded away from $0$ and $2$, see~\eqref{pfti-e12}. If $\alpha(z)=1$, the above calculation gives a bound of the form $C(t|\log t| +t)$, which also yields~\eqref{appC-e56}.

\begin{proposition}\label{appC-21}
For all $x,y\in \rd$ and sufficiently small $t\in (0,t_0]$ we have
\begin{align}\label{appC-e58}
    p_t^{z,\cut}(x)
    \leq C \int_\rd  t^{-d /\alpha(z)}   e^{-c|x-k(t,z) -u| t^{1/\alpha(z)}} \,|P_t| (z,du).
\end{align}
\end{proposition}

\begin{proof}
Recall that  $p_t^{z,\cut}(x)$ is the probability density corresponding to the exponent $\int_0^t\psi_s^{z,\cut}(\xi)\,ds$, see~\eqref{pfti-e16}; in this definition, $\psi_s^{z,\cut}(\xi)$ has truncated jumps of size less or equal than $s^{\zeta(z)}$. In the following calculations we need to adjust this truncation to $s^{1/\alpha(z)}$. We define
\begin{align*}
    \psi_{t}^{z,\cut,2}(\xi)
    := \int_{|u|\leq t^{1/\alpha(z)}} \left(1- e^{i \xi \cdot u} + i \xi \cdot u \right) \mu(z,du),
\end{align*}
and  decompose for $s\in (0,t)$ the exponent $\int_0^t \psi_s^{z,\cut}(z)ds$ as follows:
\begin{align*}
    \int_0^t\psi_s^{z,\cut}(\xi)\,ds
    &=  t \psi_t^{z,\cut,2}(\xi) + \int_0^t \int \left(1- e^{i \xi \cdot u} \right) \Lambda(t,z,du) +  i \xi \cdot  k(t,z).
\end{align*}
Define
\begin{align*}
    p_t^{z,\cut,2}(x)
    := (2\pi)^{-d} \int_\rd e^{- i \xi x - t\psi_t^{z,\cut,2}(\xi)}\,d\xi, \quad x\in \rd,\; t>0.
\end{align*}
Then we can write $p_t^{z,\cut}(x)$ as the convolution of  $p_t^{z,\cut,2}(x- k(t,z))$ and $P_t(z,du)$;
\begin{align}\label{appC-e64}
    p_t^{z,\cut}(x)= \int_\rd  p_t^{z,\cut,2}(x- k(t,z)-u) \, P_t(z,du).
\end{align}
The following estimate can be derived in the same way as~\eqref{appC-e08} in Proposition~\ref{appC-07}:
\begin{align*}
    p_t^{z,\cut,2}(x) \leq C t^{-d/\alpha(z)} e^{-c |x-k(t,z)|t^{-1/\alpha(z)}}, \quad x\in \rd, \, t\in (0,t_0];
\end{align*}
combining this with~\eqref{appC-e64} yields~\eqref{appC-e58}.
\end{proof}

\section{Estimate in the Wasserstein metric}\label{appD}

In this appendix we will establish estimates for a stable-like jump kernel $\mu(x,du)$ in the Wasserstein metric. Recall that $\mu(x,du)$ can be represented by a spherical decomposition~\eqref{set-e06}
\begin{align*}
    \mu(x,A)
    = \lambda(x) \int_0^\infty \int_{\Sph^{d-1}} \I_A(r\ell) r^{-1-\alpha(x)} \,\sigma(x,d\ell) \,dr,
    \quad x\in\rd, \; A\in\Bscr(\rd\setminus\{0\}),
\end{align*}
with a stability index $\alpha(x)\in (0,2)$ of variable order, the intensity $\lambda(x)\geq 0$ and the spherical  probability kernels $\sigma(x,d\ell)$ on $\Sph^{d-1}\subset\rd$. In particular, $\mu(x,du)$ enjoys the following scaling property $\mu(x,tA) = t^{-\alpha(x)}\mu(x,A)$.

The Kantorovich--Rubinstein theorem provides an alternative description of the Wasserstein distance $W_1(P,Q)$ between two probability measures $P, Q$ on $\Sph^{d-1}$:
\begin{align}\label{appD-e02}
    W_1(P,Q) = \sup_{f \,:\, \Lip(f)=1} \left\{ \int_{\Sph^{d-1}} f(\ell) \,P(d\ell)- \int_{\Sph^{d-1}} f(\ell) \,Q(d\ell)\right\}.
\end{align}
If we combine~\eqref{appD-e02} and~\eqref{M2}, we get for any Lipschitz function $f$ with Lipschitz constant $\Lip(f)$ and all $x,y\in\rd$ the following estimate
\begin{align}\label{appD-e04}
    \left|\int_{\Sph^{d-1}} f(\ell)\,\sigma(x,d\ell) - \int_{\Sph^{d-1}} f(\ell)\,\sigma(y,d\ell)\right|
    \leq  C \Lip(f) (|x-y|^\eta\wedge 1).
\end{align}
\begin{proposition}\label{appD-03}
    Let $\mu(z,du)$ be a stable-like jump kernel with~\eqref{M0}--\eqref{M2} and  suppose that the function $h:\rd\to \real$ satisfies the following assumptions:
    \begin{alignat}{2}
    |\rho^{-2} h(\rho\ell)|
        &\leq\label{appD-e06} C_h, &\quad&\rho\in (0,r_0],\;\ell\in \Sph^{d-1},\\
    |\rho^{-2} (h(\rho\ell_1)-h(\rho\ell_2))|
        &\leq\label{appD-e08} C_h |\ell_1-\ell_2|, &\quad& \rho\in (0,r_0],\; \ell_1,\ell_2\in \Sph^{d-1}.
    \end{alignat}
  Then, there is a constant $C>0$ such that for all $z_1,z_2\in \rd$ and  $r\in (0,r_0]$, $r_0\leq 1$. the following estimate holds
    \begin{align}\label{appD-e10}\begin{aligned}
        &\left|\int_{|u|\leq   r}  h(u) \left[\mu(z_1,du)- \mu(z_2,du)\right]\right|\\
        &\qquad\leq
        C C_h  \left(r^{2-\alpha(z_1)}+r^{2-\alpha(z_2)}\right)(1+ |\log r|) \left(| z_1-z_2|^\eta   \wedge 1\right).
    \end{aligned}\end{align}

\end{proposition}
\begin{proof}
     We introduce polar coordinates  $u = \rho \ell$ and get
    \begin{align*}
        &\left|\int_{|u|\leq r} h(u) \, \mu(z_1,du) - \int_{|u|\leq r} h(u)\,\mu( z_2,du)\right|\\
        &\qquad\leq  \lambda(z_1)  \left|\int_{\Sph^{d-1}}
            \left(\int_0^r h(\rho\ell) \rho^{-1-\alpha(z_1)}\, d\rho\right) \left[ \sigma( z_1,d\ell) - \sigma(z_2,d\ell) \right]\right|\\
        &\qquad\quad \mbox{} + \lambda(z_1)  \left|\int_0^r \int_{\Sph^{d-1}} h(\rho\ell) \left(\rho^{-1-\alpha(z_1)} - \rho^{-1-\alpha(z_2)}\right)  \sigma( z_2,d\ell)\, d\rho\right|\\
        &\qquad\quad \mbox{} +\left|\left(\lambda(z_1)-\lambda(z_2)\right)
        \int_{\Sph^{d-1}} \int_0^r h(\rho\ell) \rho^{-1-\alpha(z_2)}\, d\rho \, \sigma( z_2,d\ell)\right|\\
        &\qquad=:\mathrm{I}_1+\mathrm{I}_2+\mathrm{I}_3.
    \end{align*}
    Write $H(z,\ell):= \int_0^r h(\rho \ell) \rho^{-1-\alpha(z)}\, d\rho$. Condition~\eqref{appD-e08}  implies  that $\ell\mapsto H(z,\ell)$ is Lipschitz continuous with Lipschitz constant
    \begin{align*}
    C_H(z)= C_h  \frac{r^{2-\alpha(z)}}{(2-\alpha(z))} \leq  C' C_h r^{2-\alpha(z)}.
    \end{align*}
    Indeed,
    \begin{align*}
    \left|H(z,\ell_1)- H(z,\ell_2)\right|&\leq \int_0^r \left|h(\rho \ell_1)- h(\rho\ell_2)\right|  \rho^{-1-\alpha(z)}\, d\rho\\
    &\leq C_h\int_0^r \rho^{1-\alpha(z)}d\rho \\
    &= C_H(z).
    \end{align*}

    Because of~\eqref{M0}, $\lambda(\cdot)$ is bounded, and from~\eqref{appD-e08} and~\eqref{appD-e04} we get
    \begin{align*}
        \mathrm{I}_1\leq  C C_h r^{2-\alpha(z_1)} (|z_1-z_2|^\eta\wedge 1).
    \end{align*}
    Similarly, \eqref{appD-e06} and the H\"older continuity of $\lambda(\cdot)$, see~\eqref{M2},  yield
    \begin{align*}
        \mathrm{I}_3\leq  C C_h  r^{2-\alpha(z_2)} (|z_1-z_2|^\eta\wedge 1).
    \end{align*}
    In order to estimate $\mathrm{I}_2$, we use the boundedness of $\lambda(\cdot)$, cf.~\eqref{M0}, the H\"older continuity of $\alpha(\cdot)$, cf.~\eqref{M2}, and~\eqref{appD-e06}. Together, we get
    \begin{align*}
    \mathrm{I}_2
    &\leq \lambda_{\max} C_h \sigma(z_2,\Sph^{d-1}) \int_0^{r} \left|\rho^{1-\alpha( z_1)}- \rho^{1-\alpha(z_2)}\right| d\rho\\
    &\leq C C_h \int_0^1 \left|r^{2-\alpha(z_1)}w^{1-\alpha(z_1)} - r^{2-\alpha( z_2)}w^{1-\alpha(z_2)}\right| dw \\
    &\leq C C_h r^{2-\alpha(z_1)}\left(\int_0^1\left|w^{1-\alpha(z_1)} - w^{1-\alpha( z_2)}\right| dw
        + |1-r^{\alpha(z_1)-\alpha(z_2)}| \int_0^1  w^{1-\alpha(z_2)} dw\right)\\
    &\leq C C_h r^{2-\alpha(z_1)}\left(\int_0^1 \left|\int_{1-\alpha( z_2)}^{1-\alpha( z_1)} w^\theta \,d\theta\right| |\log w|\,dw
        + |\alpha(z_1)-\alpha(z_2) |\log r|\right)\\
    &\leq C C_h r^{2-\alpha(z_1)} (1+|\log r|) |\alpha(z_1)-\alpha(z_2)| \\
    &\leq C C_h r^{2-\alpha(z_1)} (1+|\log r|) \left(|z_1-z_2|^\eta \wedge 1\right),
    \end{align*}
    where we use  $r\leq r_0$ and the estimate $|1-r^{z_1}|\leq  C |z_1|\cdot |\log r|$. Together, we get~\eqref{appD-e10}.
\end{proof}

\section{Proof of Proposition~\ref{exa-07}}\label{appE}

In this section we will verify the (in-)equalities~\eqref{exa-e06}, \eqref{exa-e08} and~\eqref{exa-e10} appearing in Proposition~\ref{exa-07}. We use the construction and the bounds from Section~\ref{pfti}.

\begin{proof}[Proof of~\eqref{exa-e06}]  We decompose the operator
\begin{align*}
    L = (L^\trunc-\lambda\, \id)+ \Upsilon^\tail
\end{align*}
which leads to the following decomposition of the corresponding kernels
\begin{align*}
    \Phi_t(x,y)
        &= \widetilde \Phi_t(x,y) + \Delta_t(x,y).\\
\intertext{Recall that $\Phi_t$ and $\Delta_t$ are defined via the zero-order approximation $p_t^0$ from~\eqref{pfti-e26}}
    \Delta_t(x,y)
        &=\Upsilon^\tail  p_t^0(x,y),\\
    \widetilde \Phi_t(x,y)
        &= (L^\trunc-\lambda\, \id)_xp_t^0(x,y)\\
        &= \Phi_t^\trunc(x,y)-\lambda p_t^0(x,y),\\
    \Phi_t^\trunc(x,y)
        &=L^\trunc_xp_t^0(x,y).
\end{align*}
The space-time convolution `$\cstar$' was introduced in~\eqref{para-e08}. We have
\begin{alignat*}{3}
    p &= p^0+p^0\cstar\Psi, &\qquad&&
    \Psi &= \sum_{k=1}^\infty  \Phi^{\cstar k},\\
    p^\trunc
    &= p^0+p^0\cstar\Psi^\trunc, &\qquad&&
    \Psi^\trunc &= \sum_{k=1}^\infty  (\Phi^\trunc)^{\cstar k}.
\end{alignat*}
Define
\begin{align*}
    \widetilde p:=p^0+p^0\cstar\widetilde\Psi
    \et
    \widetilde \Psi := \sum_{k=1}^\infty  \widetilde \Phi^{\cstar k}.
\end{align*}
Using the geometric summation formula $(a+b)^k - b^k =  \sum_{i=1}^k (a+b)^{i-1} b a^{k-i}$ we get
\begin{align}
    \Psi-\widetilde \Psi
    &\notag= \sum_{k=1}^\infty  (\widetilde \Phi+\Delta)^{\cstar k}-\sum_{k=1}^\infty  \widetilde \Phi^{\cstar k}\\
    &\label{appE-e02}= \Delta + \sum_{k=2}^\infty\left[  \Delta \cstar  \widetilde \Phi^{\cstar (k-1)}+\sum_{i=2}^{k-1} (\widetilde \Phi+\Delta)^{\cstar (i-1)}\cstar \Delta \cstar  \widetilde \Phi^{\cstar (k-i)}
    + (\widetilde \Phi+\Delta)^{\cstar (k-1)}\cstar \Delta\right]\\
    &\notag=\Delta +\Delta \cstar  \widetilde \Psi +      \Psi \cstar \Delta \cstar  \widetilde \Psi + \Psi \cstar \Delta.
\end{align}
The change of the order of summation in the last identity is possible due to bounds similar to~\eqref{para-e30}.
Fubini's theorem and the definition of $\Delta$ show
\begin{align*}
    \Delta + \Delta \cstar  \widetilde \Psi
    = \Upsilon^\tail  p^0 + (\Upsilon^\tail  p^0) \cstar \widetilde \Psi
    = \Upsilon^\tail  \left(p^0+p^0\cstar \widetilde \Psi\right)
    = \Upsilon^\tail \widetilde p,
\end{align*}
which gives
\begin{align*}
    \Psi-\widetilde \Psi
    = \Upsilon^\tail\widetilde p + \Psi\cstar(\Upsilon^\tail\widetilde p).
\end{align*}
If we insert this into the definition of $p$ and $\widetilde p$, we get
\begin{align}\label{appE-e06}
\begin{aligned}
    p
    &=\widetilde p + p^0\cstar(\Psi-\widetilde\Psi)\\
    &=\widetilde p + p^0\cstar\left(\Upsilon^\tail\widetilde p + \Psi\cstar\left(\Upsilon^\tail \widetilde p\right)\right)\\
    &=\widetilde p + \left(p^0+p^0\cstar\Psi\right) \cstar \left(\Upsilon^\tail\widetilde p\right)\\
    &=\widetilde p+p\cstar\left(\Upsilon^\tail\widetilde p\right).
\end{aligned}
\end{align}

If we replace in~\eqref{appE-e02} $\widetilde\Phi \rightsquigarrow \Psi^\trunc$ and $\Delta\rightsquigarrow -\lambda p^0$, the same calculation as in~\eqref{appE-e02} shows that
\begin{align*}
\begin{aligned}
    \widetilde\Psi- \Psi^\trunc
    &=-\lambda \left(p^0 +p^0\cstar  \Psi^\trunc +   \widetilde\Psi \cstar p^0 \cstar  \Psi^\trunc  + \widetilde\Psi \cstar p^0\right)\\
    &=-\lambda p^\trunc-\lambda  \widetilde\Psi \cstar p^\trunc.
  \end{aligned}
\end{align*}
Inserting this into the definition of $\widetilde p$ and $p^\trunc$ yields
\begin{align*}
\begin{aligned}
    \widetilde p
    &=p^\trunc+ p^0\cstar(\widetilde\Psi-\Psi^\trunc)\\
    &=p^\trunc-\lambda p^0\cstar p^\trunc -\lambda p^0\cstar \widetilde\Psi \cstar p^\trunc\\
    &=p^\trunc-\lambda \widetilde p\cstar p^\trunc.
\end{aligned}
\end{align*}

If we interpret the last identity as an equation for $\widetilde p$ in the space $L^\infty(dx)\otimes L^\infty(dy)\otimes L^1([0,T],dt)$, we observe that
\textbf{i)} its solution is unique;
and \textbf{ii)} because of the semigroup property for $p^\trunc$, the function
\begin{align}\label{appE-e12}
    \widetilde p_t(x,y)=e^{-\lambda t}p^\trunc_t(x,y)
\end{align}
is a solution to the equation.

Combining~\eqref{appE-e06} and~\eqref{appE-e12}, finishes the proof of~\eqref{exa-e06}.
\end{proof}

\begin{proof}[Proof of~\eqref{exa-e08}]
We apply Theorem~\ref{t2}. In the SDE~\eqref{exa-e04} we assume that $b(x)=0$ and $b_t(x)=0$, thus $\chi_t(x)=x$. Because of~\eqref{set-e26} we have,  uniformly for all $t>0$,
\begin{align*}
    \int_{|y-x|\leq t^{1/\alpha}}|R_t(x,y)|\, dy\leq Ct^{\epsilon_R}.
\end{align*}
Since $g^x$ is the probability density of $a(x) Z_1$, $a(x)$ is bounded and $Z_1$ is an $\alpha$-stable random variable, we have
\begin{align*}
    \int_{|y-x|\leq t^{1/\alpha}}t^{d/\alpha}g^{x}\left(\frac{y-x}{t^{1/\alpha}}\right)\, dy
    =\int_{|w|\leq 1}g^{x}\left(w\right)\, dw
    =\mathds{appC-e50}(|a(x) Z_1| \leq 1)\geq c.
\end{align*}
Thus, \eqref{set-e24}, gives
\begin{align*}
    \int_{|y-x|\leq t^{1/\alpha}}p_t(x,y)\, dy
    \geq c- Ct^{\epsilon_R},
\end{align*}
and the assertion follows for all $t\leq t_0$ with some sufficiently small $t_0>0$.
\end{proof}

\begin{proof}[Proof of~\eqref{exa-e10}]
Observe that for $|x|< 1/2$ the transition density $p^\trunc_t(x,y)$ dominates the transition density of the process $Y$ obtained from $Z$ by killing upon exiting the ball $\{|y|\leq 1/2\}$. The latter transition density satisfies (the analogue of)~\eqref{exa-e10}; this can be easily seen if we consider another process $U=(U^1, \dots, U^d)$ with one-dimensional i.i.d.\ $\alpha$-stable components $U^i$, which is killed upon exiting from the hypercube
\begin{align*}
    \left[-\tfrac{1}{2\sqrt{d}},\:\tfrac{1}{2\sqrt{d}}\right]^d
    \subset \{|y|\leq 1/2\}.
\end{align*}
The required bound follows from the known estimates for the transition density for $U^i$, see e.g.~\cite[Example 1]{BGR10}.
\end{proof}

\end{document}